\documentclass[draft]{article}
\usepackage{amsfonts,amsmath,amssymb,color,enumerate,epic,latexsym,times}
\definecolor{monblu}{cmyk}{0.80,0.00,0.00,0.70}
\definecolor{monmag}{cmyk}{0.00,1.00,0.00,0.00}
\newtheorem{claim}{Claim}
\newtheorem{lemma}{Lemma}
\newtheorem{proposition}{Proposition}
\newtheorem{definition}{Definition}
\newenvironment{proof}{{\bf Proof:}}{~$\dashv$\\}
\newenvironment{proofclaim}{{\em Proof:}\/}{~$\dashv$}
\def\N{\mathbb{N}}
\def\PEARL{\mathbf{PEARL}}
\def\Axiom{\mathbf{A}}
\def\Rule{\mathbf{R}}
\def\WK{\mathbf{WK}}
\def\FIK{\mathbf{FIK}}
\def\LIK{\mathbf{LIK}}
\def\Fo{\mathbf{Fo}}
\def\PSPACE{\mathbf{PSPACE}}
\def\reflexive{\mathbf{ref}}
\def\symmetric{\mathbf{sym}}
\def\transitive{\mathbf{tra}}
\def\ureflexive{\mathbf{uref}}
\def\usymmetric{\mathbf{usym}}
\def\utransitive{\mathbf{utra}}
\def\dreflexive{\mathbf{dref}}
\def\dsymmetric{\mathbf{dsym}}
\def\dtransitive{\mathbf{dtra}}
\def\partition{\mathbf{par}}
\def\allfra{\mathbf{all}}
\def\fbcfra{\mathbf{fbc}}
\def\bcfra{\mathbf{bc}}
\def\fcfra{\mathbf{fc}}
\def\qfcfra{\mathbf{qfc}}
\def\qbcfra{\mathbf{qbc}}
\def\qucfra{\mathbf{quc}}
\def\qdcfra{\mathbf{qdc}}
\def\fbcfra{\mathbf{fbc}}
\def\fdcfra{\mathbf{fdc}}
\def\fucfra{\mathbf{fuc}}
\def\bdcfra{\mathbf{bdc}}
\def\bucfra{\mathbf{buc}}
\def\ducfra{\mathbf{duc}}
\def\fbdcfra{\mathbf{fbdc}}
\def\fbucfra{\mathbf{fbuc}}
\def\fducfra{\mathbf{fduc}}
\def\bducfra{\mathbf{bduc}}
\def\fbducfra{\mathbf{fbduc}}
\def\dcfra{\mathbf{dc}}
\def\ucfra{\mathbf{uc}}
\def\card{\mathtt{Card}}
\def\At{\mathbf{At}}
\def\IS{\mathbf{IS}}
\def\K{\mathbf{K}}
\def\IK{\mathbf{IK}}
\def\S{\mathbf{S}}
\def\Log{\mathtt{Log}}
\def\CK{\mathbf{CK}}
\def\IPL{\mathbf{IPL}}
\def\L{\mathbf{L}}
\begin{document}
\title{Intuitionistic modal logics: a minimal setting}
\author{Philippe Balbiani$^{(1)}$\footnote{Email address: philippe.balbiani@irit.fr.}
\hspace{0.31cm}
\c{C}i\u{g}dem Gencer$^{(1,2)}$\footnote{Email addresses: cigdem.gencer@irit.fr and cigdemgencer@aydin.edu.tr.}}
\date{$^{(1)}$Toulouse Institute of Computer Science Research
\\
CNRS--INPT--UT3, Toulouse, France
\\
$^{(2)}$Faculty of Arts and Sciences
\\
Ayd\i n University, Istanbul, Turkey}
\maketitle
\begin{abstract}
We introduce an intuitionistic modal logic strictly contained in the intuitionistic modal logic $\IK$ and being an appropriate candidate for the title of ``minimal normal intuitionistic modal logic''.
\end{abstract}
Modal connectives.
Intuitionistic modal logics.
Axiomatization and completeness.
Decidability and complexity.
\section{Introduction}\label{section:introduction}
Among researchers interested in mixing together Boolean concepts and modal concepts, the search for minimality has prompted multifarious debates.
Most of these debates have now faded.
See~\cite[Chapter~$1$]{Blackburn:et:al:2001} and~\cite[Chapter~$3$]{Chagrov:Zakharyaschev:1997} for an historical perspective.
In fact, they have not survived the end of the syntactic era of modal logic.
Before the advent, in the 1960s, of the concept of Kripke frames, it was unclear which modal logic is the most appropriate candidate for the title of ``minimal modal logic''.
Today, although the modal logic $\K$ does not contain any of the formulas ${\square}p{\rightarrow}{\square}{\square}p$ (``when a proposition is logically necessary, it is logically bound to be logically necessary'') and ${\lozenge}p{\rightarrow}{\square}{\lozenge}p$ (``when a proposition is logically possible, it is logically bound to be logically possible'') expressing the properties usually associated to the modal concepts of necessity and possibility, everyone accepts the fact that $\K$ is the minimal modal logic~\cite[Chapter~$2$]{Hughes:Cresswell:1996}.
Among researchers interested in combining together intuitionistic concepts and modal concepts, the quest for minimality has caused numerous discussions.
Some of these discussions are still alive.
See~\cite{Olivetti:2022,IMLA:2017,Stewart:et:al:2018} for a general introduction to them.
Indeed, two approaches are in direct opposition: the intuitionistic approach set out by Fischer Servi~\cite{FischerServi:1977,FischerServi:1978,FischerServi:1984} giving rise to the intuitionistic modal logic $\IK$ and the constructive approach expounded by Wijesekera~\cite{Wijesekera:1990} giving rise to the intuitionistic modal logics $\WK$ and $\CK$.
And there is no consensus of what should be the ``minimal intuitionistic modal logic''.
The {\em raison d'\^etre}\/ of the intuitionistic approach set out by Fischer Servi is mainly the fact that a minimal intuitionistic modal logic should contain the formulas whose standard translation in a first-order language are intuitionistically valid.
Therefore, the supporters of this approach consider intuitionistic modal logic $\IK$ that contains the formulas ${\lozenge}(p{\vee}q){\rightarrow}{\lozenge}p{\vee}{\lozenge}q$ and ${\neg}{\lozenge}{\bot}$ despite their non-constructive character~\cite{Amati:Pirri:1994,Ewald:1986,FischerServi:1977,FischerServi:1978,FischerServi:1984,Plotkin:Stirling:1986,Simpson:1994}.
The justification of the constructive approach expounded by Wijesekera chiefly rests on the fact that the formulas of a minimal intuitionistic modal logic should have a constructive character.
As a result, the upholders of this approach consider intuitionistic modal logic $\WK$ that does not contain the formula ${\lozenge}(p{\vee}q){\rightarrow}{\lozenge}p{\vee}{\lozenge}q$~\cite{Kojima:IGPL:2012,Wijesekera:1990}.
The most radical of them also consider intuitionistic modal logic $\CK$ that does not contain either the formula ${\neg}{\lozenge}{\bot}$~\cite{Alechina:et:al:2001,Arisaka:et:al:2015,Bierman:dePaiva:2000,Dalmonte:et:al:2021,Mendler:Scheele:2014}.
There is no sense in comparing the arguments for and against the intuitionistic approach giving rise to the intuitionistic modal logic $\IK$ and the constructive approach giving rise to the intuitionistic modal logics $\WK$ and $\CK$.
To convince the reader of this opinion, it suffices to mention how these approaches differently define the truth condition of $\lozenge$-formulas in their relational semantics, although they consider the same truth condition of $\square$-formulas.
For the intuitionistic approach, ${\lozenge}A$ holds at state $s$ in model $(W,{\leq},{R},V)$ if $A$ holds at some state $t$ such that $s{R}t$.
For the constructive approach, ${\lozenge}A$ holds at state $s$ in model $(W,{\leq},{R},V)$ if for all states $t$, if $s{\leq}t$ then $A$ holds at some state $u$ such that $t{R}u$.
This difference in the definition of the truth condition of $\lozenge$-formulas shows that the connective $\lozenge$ {\it \`a la}\/ Fischer Servi and the connective $\lozenge$ {\it \`a la}\/ Wijesekera are as separate as are, for example, the connectives $\vee$ and $\wedge$ in any intermediate logic.
%
%

%
%
%Let us consider a propositional language with $\rightarrow$, $\bot$, $\top$, $\vee$, $\wedge$, $\square$ and $\lozenge$.
Concerning our intuitionistic modal logics, if the criterion of minimality is based on the relation of inclusion then we are forced to accept that the minimal intuitionistic modal logic is axiomatically presented by considering the standard axiomatization of Intuitionistic Propositional Logic.\footnote{See~\cite[Chapter~$2$]{Chagrov:Zakharyaschev:1997} for an introduction to the standard axiomatization of Intuitionistic Propositional Logic.}
If we want to be able to define the Lindenbaum-Tarski Algebras of our intuitionistic modal logics then we can only accept that the minimal intuitionistic modal logic is axiomatically presented by adding the inference rules $\frac{p{\leftrightarrow}q}{{\square}p{\leftrightarrow}{\square}q}$ and $\frac{p{\leftrightarrow}q}{{\lozenge}p{\leftrightarrow}{\lozenge}q}$ to the standard axiomatization of Intuitionistic Propositional Logic.
%\footnote{Of course, in the propositional language with $\rightarrow$, $\bot$, $\top$, $\vee$, $\wedge$, $\square$ and $\lozenge$ that we consider here, for all formulas $A,B$, we write $A{\leftrightarrow}B$ as an abbreviation instead of $(A{\rightarrow}B){\wedge}(B{\rightarrow}A)$.}
In this paper, we want our intuitionistic modal logics to be normal.
This means that we want their axiomatical presentations to contain the axioms ${\square}p{\wedge}{\square}q{\rightarrow}{\square}(p{\wedge}q)$, ${\lozenge}(p{\vee}q){\rightarrow}{\lozenge}p{\vee}{\lozenge}q$, ${\square}{\top}$ and $\neg{\lozenge}{\bot}$ and the inference rules $\frac{p{\rightarrow}q}{{\square}p{\rightarrow}{\square}q}$ and $\frac{p{\rightarrow}q}{{\lozenge}p{\rightarrow}{\lozenge}q}$, seeing that they are usually associated to the concept of normality in modal logics.
By doing so, we therefore introduce an appropriate candidate for the title of minimal normal intuitionistic modal logic~---~an intuitionistic modal logic strictly contained in $\IK$ and comparable neither with $\WK$, nor with $\CK$.
Coming back to the relational semantics of intuitionistic modal logics, we make a clean sweep of the tradition by considering a new truth condition of $\lozenge$-formulas saying that ${\lozenge}A$ holds at state $s$ in model $(W,{\leq},{R},V)$ if there exists a state $t$ such that $s{\geq}t$ and there exists a state $u$ where $A$ holds and such that $t{R}u$.\footnote{This definition of the satisfiability of $\lozenge$-formulas has only been considered once in the literature~---~by P\v{r}enosil~\cite{Prenosil:2014}~---~and it was from the perspective of Duality Theory.}
While keeping the truth condition of $\square$-formulas that is commonly used in the intuitionistic approach and the constructive approach, we axiomatize validity in the class of all models.
The resulting axiomatization constitutes the intuitionistic modal logic $\L_{\min}$ that we put forward as an appropriate candidate for the title of ``minimal normal intuitionistic modal logic''.
We also axiomatize validity in classes of models satisfying some well-known conditions of confluence.
Among the resulting axiomatizations, one can find $\IK$ which corresponds to validity~---~with respect to our truth conditions of $\lozenge$-formulas and $\square$-formulas~---~in the class of all forward and backward confluent models.
To read Simpson's doctoral thesis~\cite[Page~$49$]{Simpson:1994}, the new truth condition of $\lozenge$-formulas that we consider has been broached in their time by Plotkin and Stirling who, keeping the truth condition of $\square$-formulas that is commonly used in the intuitionistic approach, the constructive approach and our new approach, have axiomatized validity in the class of all models.
Simpson's comment saying that the resulting axiomatization~---~that we have never seen~---~is ``rather complicated'' explains why Plotkin and Stirling have never presented it to a large public.
We just hope here that the reader does not apply the same comment to the axiomatization we present.
In Section~\ref{section:syntax}, we introduce the syntax of our intuitionistic modal logics.
In Section~\ref{section:semantics}, we introduce the relational semantics of our intuitionistic modal logics.
In Section~\ref{section:definability}, we study the correspondence between elementary conditions on frames and formulas.
In Section~\ref{section:axiomatization}, we axiomatically present different intuitionistic modal logics.
In Sections~\ref{section:theories}, \ref{section:existence:lemmas:and:lindenbaum:lemma} and \ref{section:canonical:frame:and:canonical:model}, we present results needed for the proofs of completeness of these intuitionistic modal logics given in Section~\ref{section:soundness:and:completeness}.
In Section~\ref{section:finite:frame:property}, we show that the membership problem in our minimal intuitionistic modal logic is decidable.
Section~\ref{section:miscellaneous} presents miscellaneous properties of our intuitionistic modal logics.
From now on, for all sets $\Sigma$, $\card(\Sigma)$ denotes the {\em cardinal of $\Sigma$.}
Moreover, for all sets $W$ and for all binary relations $R,S$ on $W$, ${R}{\circ}{S}$ denotes the {\em composition of $R$ and $S$,} i.e. the binary relation $T$ on $W$ such that for all $s,t{\in}W$, $s{T}t$ if and only if there exists $u{\in}W$ such that $s{R}u$ and $u{S}t$.
In other respect, for all sets $W$ and for all preorders $\leq$ on $W$, a subset $U$ of $W$ is {\em $\leq$-closed}\/ if for all $s,t{\in}W$, if $s{\in}U$ and $s{\leq}t$ then $t{\in}U$.
Finally, $\IPL$ denotes Intuitionistic Propositional Logic.
\section{Syntax}\label{section:syntax}
In this section, we introduce the syntax of our intuitionistic modal logics.
\begin{definition}[Atoms and formulas]
Let $\At$ be a countably infinite set (with typical members called {\em atoms}\/ and denoted $p$, $q$, etc).
Let $\Fo$ be the countably infinite set (with typical members called {\em formulas}\/ and denoted $A$, $B$, etc) of finite words over $\At{\cup}\{{\rightarrow},{\top},{\bot},{\vee},{\wedge},{\square},{\lozenge},(,)\}$ defined by
$$A\ {::=}\ p{\mid}(A{\rightarrow}A){\mid}{\top}{\mid}{\bot}{\mid}(A{\vee}A){\mid}(A{\wedge}A){\mid}{\square}A{\mid}{\lozenge}A$$
where $p$ ranges over $\At$.
For all formulas $A$, the {\em length of $A$}\/ (denoted ${\parallel}A{\parallel}$) is the number of symbols in $A$.
%By induction on $A$, the reader may easily define ${\parallel}A{\parallel}$.
%
%
\end{definition}
We follow the standard rules for omission of the parentheses.
For all formulas $A,B$, we write $\neg A$ as an abbreviation instead of $A{\rightarrow}{\bot}$ and $A{\leftrightarrow}B$ as an abbreviation instead of $(A{\rightarrow}B){\wedge}(B{\rightarrow}A)$.
\begin{definition}[Accessibility between sets of formulas]
Let ${\bowtie}$ be the binary relation of {\em accessibility between sets of formulas}\/ such that for all sets $\Delta,\Lambda$ of formulas, $\Delta{\bowtie}\Lambda$ if and only if for all formulas $B$, the following conditions hold:\footnote{Obviously, there exists sets $\Delta,\Lambda$ of formulas such that $\Delta{\bowtie}\Lambda$ and not $\Lambda{\bowtie}\Delta$.
However, we use a symmetric symbol to denote the binary relation of accessibility between sets of formulas, seeing that in its definition, $\square$-formulas and $\lozenge$-formulas play symmetric roles.}
\begin{itemize}
\item if ${\square}B{\in}\Delta$ then $B{\in}\Lambda$,
\item if $B{\in}\Lambda$ then ${\lozenge}B{\in}\Delta$.
\end{itemize}
For all sets $\Gamma$ of formulas, let ${\bowtie^{\Gamma}}$ be the binary relation between sets of formulas such that for all sets $\Delta,\Lambda$ of formulas, $\Delta{\bowtie^{\Gamma}}\Lambda$ if and only if for all formulas $A,B$, the following conditions hold:
\begin{itemize}
\item if $A{\not\in}\Gamma$ and $A{\vee}{\square}B{\in}\Delta$ then $B{\in}\Lambda$,
\item if $B{\in}\Lambda$ then ${\lozenge}B{\in}\Delta$.
\end{itemize}
\end{definition}
The binary relation ${\bowtie}$ between sets of formulas and for all sets $\Gamma$ of formulas, the binary relation ${\bowtie^{\Gamma}}$ between sets of formulas are used in Sections~\ref{section:existence:lemmas:and:lindenbaum:lemma} and \ref{section:canonical:frame:and:canonical:model} when we investigate the completeness of some intuitionistic modal logics.
In particular, the restriction of ${\bowtie}$ to prime theories is the accessibility relation of the canonical frame introduced in Section~\ref{section:canonical:frame:and:canonical:model}.
\begin{definition}[Closed sets of formulas]
A set $\Sigma$ of formulas is {\em closed}\/ if for all formulas $A,B$,
\begin{itemize}
\item if $A{\rightarrow}B{\in}\Sigma$ then $A{\in}\Sigma$ and $B{\in}\Sigma$,
\item if $A{\vee}B{\in}\Sigma$ then $A{\in}\Sigma$ and $B{\in}\Sigma$,
\item if $A{\wedge}B{\in}\Sigma$ then $A{\in}\Sigma$ and $B{\in}\Sigma$,
\item if ${\square}A{\in}\Sigma$ then $A{\in}\Sigma$,
\item if ${\lozenge}A{\in}\Sigma$ then $A{\in}\Sigma$.
\end{itemize}
For all formulas $A$, let $\Sigma_{A}$ be the least closed set of formulas containing $A$.
\end{definition}
Obviously, for all atoms $p$, $\Sigma_{p}{=}\{p\}$.
Moreover, for all formulas $A,B$,
\begin{itemize}
\item $\Sigma_{A{\rightarrow}B}{=}\{A{\rightarrow}B\}{\cup}\Sigma_{A}{\cup}\Sigma_{B}$,
\item $\Sigma_{\top}{=}\{\top\}$,
\item $\Sigma_{\bot}{=}\{\bot\}$,
\item $\Sigma_{A{\vee}B}{=}\{A{\vee}B\}{\cup}\Sigma_{A}{\cup}\Sigma_{B}$,
\item $\Sigma_{A{\wedge}B}{=}\{A{\wedge}B\}{\cup}\Sigma_{A}{\cup}\Sigma_{B}$,
\item $\Sigma_{{\square}A}{=}\{{\square}A\}{\cup}\Sigma_{A}$,
\item $\Sigma_{{\lozenge}A}{=}\{{\lozenge}A\}{\cup}\Sigma_{A}$.
\end{itemize}
Therefore, by induction on $A$, the reader may easily verify that $\Sigma_{A}$ is finite.
More precisely,
\begin{lemma}\label{size:of:least:closed:set:containing:some:given:formula}
For all formulas $A$, $\card(\Sigma_{A}){\leq}{\parallel}A{\parallel}$.
\end{lemma}
\begin{proof}
By induction on $A$.
\medskip
\end{proof}
The concept of closed set of formulas is used in Section~\ref{section:finite:frame:property} when we investigate the complexity of some intuitionistic modal logics.
\section{Semantics}\label{section:semantics}
In this section, we introduce the relational semantics of our intuitionistic modal logics.
\begin{definition}[Frames]
A {\em frame}\/ is a relational structure of the form $(W,{\leq},{R})$ where $W$ is a nonempty set of {\em states,} $\leq$ is a preorder on $W$ and ${R}$ is a binary relation on $W$.
For all frames $(W,{\leq},{R})$ and for all $s,t{\in}W$, when we write ``$s{\geq}t$'' we mean ``$t{\leq}s$''.
Let ${\mathcal C}_{\allfra}$ be the class of all frames.
\end{definition}
\begin{definition}[Confluences]
A frame $(W,{\leq},{R})$ is {\em forward confluent}\/ if for all $s,t{\in}W$, if $s{\geq}{\circ}{R}t$ then $s{R}{\circ}{\geq}t$.
A frame $(W,{\leq},{R})$ is {\em backward confluent}\/ if for all $s,t{\in}W$, if $s{R}{\circ}{\leq}t$ then $s{\leq}{\circ}{R}t$.
A frame $(W,{\leq},{R})$ is {\em downward confluent}\/ if for all $s,t{\in}W$, if $s{\leq}{\circ}{R}t$ then $s{R}{\circ}{\leq}t$.
A frame $(W,{\leq},{R})$ is {\em upward confluent}\/ if for all $s,t{\in}W$, if $s{R}{\circ}{\geq}t$ then $s{\geq}{\circ}{R}t$.
Let ${\mathcal C}_{\fcfra}$ be the class of all forward confluent frames, ${\mathcal C}_{\bcfra}$ be the class of all backward confluent frames, ${\mathcal C}_{\dcfra}$ be the class of all downward confluent frames and ${\mathcal C}_{\ucfra}$ be the class of all upward confluent frames.
We also write ${\mathcal C}_{\fbcfra}$ to denote the class of all forward and backward confluent frames, ${\mathcal C}_{\fdcfra}$ to denote the class of all forward and downward confluent frames, etc.
\end{definition}
The elementary conditions characterizing forward confluent frames and backward confluent frames have been considered in~\cite{FischerServi:1984} where they have been called ``connecting properties'' and in~\cite[Chapter~$3$]{Simpson:1994} where they have been called ``$(\mathbf{F}1)$'' and ``$(\mathbf{F}2)$''.
They have also been considered in~\cite{Celani:2001,Marin:et:al:2021,Murai:Sano:2022,Nomura:Sano:Tojo:2015,Plotkin:Stirling:1986}.
The elementary conditions characterizing downward confluent frames and upward confluent frames have been considered in~\cite{Bozic:Dosen:1984,Dosen:1985} where they have received no specific name.
Other elementary conditions have also been considered in the literature~\cite{Bozic:Dosen:1984,Dosen:1985,Fairtlough:Mendler:1997,Kojima:Igarashi:2011,Sotirov:1980,Sotirov:1984,Vakarelov:1981}: ${\leq}{\circ}{R}{\subseteq}{R}$, ${R}{\circ}{\leq}{\subseteq}{R}$, ${\geq}{\circ}{R}{\subseteq}{R}$, ${R}{\circ}{\geq}{\subseteq}{R}$, ${\leq}{\circ}{R}{\circ}{\leq}{\subseteq}{R}$, ${\geq}{\circ}{R}{\circ}{\geq}{\subseteq}{R}$, etc.
In particular, we suggest the reader to consult the multifarious elementary conditions studied by Sotirov in his survey about intuitionistic modal logic~\cite{Sotirov:1984}.
\begin{definition}[Quasi-confluences]
A frame $(W,{\leq},{R})$ is {\em quasi-forward confluent}\/ if for all $s,t{\in}W$, if $s{\geq}{\circ}{R}t$ then $s(({\leq}{\circ}{R}{\circ}{\leq}){\cap}({\geq}{\circ}{R}{\circ}{\geq})){\circ}{\geq}t$.
A frame $(W,{\leq},{R})$ is {\em quasi-backward confluent}\/ if for all $s,t{\in}W$, if $s{R}{\circ}{\leq}t$ then $s{\leq}{\circ}{(({\leq}{\circ}{R}{\circ}{\leq}){\cap}({\geq}{\circ}{R}{\circ}{\geq}))}t$.
A frame $(W,{\leq},{R})$ is {\em quasi-downward confluent}\/ if for all $s,t{\in}W$, if $s{\leq}{\circ}{R}t$ then $s{(({\leq}{\circ}{R}{\circ}{\leq}){\cap}({\geq}{\circ}{R}{\circ}{\geq}))}{\circ}{\leq}t$.
A frame $(W,{\leq},{R})$ is {\em quasi-upward confluent}\/ if for all $s,t{\in}W$, if $s{R}{\circ}{\geq}t$ then $s{\geq}{\circ}(({\leq}{\circ}{R}{\circ}{\leq}){\cap}({\geq}{\circ}{R}{\circ}{\geq}))t$.
Let ${\mathcal C}_{\qfcfra}$ be the class of all quasi-forward confluent frames, ${\mathcal C}_{\qbcfra}$ be the class of all quasi-backward confluent frames, ${\mathcal C}_{\qdcfra}$ be the class of all quasi-downward confluent frames and ${\mathcal C}_{\qucfra}$ be the class of all quasi-upward confluent frames.
\end{definition}
Obviously, for all frames $(W,{\leq},{R})$, if $(W,{\leq},{R})$ is forward confluent then $(W,{\leq},{R})$ is quasi-forward confluent, if $(W,{\leq},{R})$ is backward confluent then $(W,{\leq},{R})$ is quasi-backward confluent, if $(W,{\leq},{R})$ is downward confluent then $(W,{\leq},{R})$ is quasi-downward confluent and if $(W,{\leq},{R})$ is upward confluent then $(W,{\leq},{R})$ is quasi-upward confluent.
\begin{definition}[Reflexivity, symmetry and transitivity]
A frame $(W,{\leq},{R})$ is {\em reflexive}\/ if for all $s{\in}W$, $s{R}s$.
A frame $(W,{\leq},{R})$ is {\em symmetric}\/ if for all $s,t{\in}W$, if $s{R}t$ then $t{R}s$.
A frame $(W,{\leq},{R})$ is {\em transitive}\/ if for all $s,t,u{\in}W$, if $s{R}t$ and $t{R}u$ then $s{R}u$.
Let ${\mathcal C}_{\reflexive}$ be the class of all reflexive frames, ${\mathcal C}_{\symmetric}$ be the class of all symmetric frames and ${\mathcal C}_{\transitive}$ be the class of all transitive frames.
A {\em partition}\/ is a reflexive, symmetric and transitive frame.
Let ${\mathcal C}_{\partition}$ be the class of all partitions.
\end{definition}
Obviously, for all symmetric frames $(W,{\leq},{R})$, $(W,{\leq},{R})$ is forward confluent if and only if $(W,{\leq},{R})$ is backward confluent and $(W,{\leq},{R})$ is downward confluent if and only if $(W,{\leq},{R})$ is upward confluent.
\begin{definition}[Up- and down- reflexivity, symmetry and transitivity]
A frame $(W,
$\linebreak$
{\leq},{R})$ is {\em up-reflexive}\/ if for all $s{\in}W$, $s{\leq}{\circ}{R}{\circ}{\leq}s$.
A frame $(W,{\leq},{R})$ is {\em down-reflexive}\/ if for all $s{\in}W$, $s{\geq}{\circ}{R}{\circ}{\geq}s$.
A frame $(W,{\leq},{R})$ is {\em up-symmetric}\/ if for all $s,t{\in}W$, if $s{R}t$ then $t{\leq}{\circ}{R}{\circ}{\leq}s$.
A frame $(W,{\leq},{R})$ is {\em down-symmetric}\/ if for all $s,t{\in}W$, if $s{R}t$ then $t{\geq}{\circ}{R}{\circ}{\geq}s$.
A frame $(W,{\leq},{R})$ is {\em up-transitive}\/ if for all $s,t,u,v{\in}W$, if $s{R}t$, $t{\leq}u$ and $u{R}v$ then $s{\leq}{\circ}{R}{\circ}{\leq}v$.
A frame $(W,{\leq},{R})$ is {\em down-transitive}\/ if for all $s,t,u,v{\in}W$, if $s{R}t$, $t{\geq}u$ and $u{R}v$ then $s{\geq}{\circ}{R}{\circ}{\geq}v$.
Let ${\mathcal C}_{\ureflexive}$ be the class of all up-reflexive frames, ${\mathcal C}_{\dreflexive}$ be the class of all down-reflexive frames, ${\mathcal C}_{\usymmetric}$ be the class of all up-symmetric frames, ${\mathcal C}_{\dsymmetric}$ be the class of all down-symmetric frames, ${\mathcal C}_{\utransitive}$ be the class of all up-transitive frames and ${\mathcal C}_{\dtransitive}$ be the class of all down-transitive frames.
\end{definition}
Obviously, for all frames $(W,{\leq},{R})$, if $(W,{\leq},{R})$ is reflexive then $(W,{\leq},{R})$ is up-reflexive and down-reflexive and if $(W,{\leq},{R})$ is symmetric then $(W,{\leq},{R})$ is up-symmetric and down-symmetric.
However, there exists transitive frames which are neither up-transitive, nor down-transitive.
Witness, the frame $(W^{\prime},{\leq^{\prime}},{R^{\prime}})$ defined by $W^{\prime}{=}\{a,b,c,d,e\}$, $b{\leq^{\prime}}c$ and $c{\leq^{\prime}}d$ and $a{R^{\prime}}c$, $b{R^{\prime}}e$ and $d{R^{\prime}}e$.
\begin{lemma}\label{lemma:about:up:and:down:frames:and:the:corresponding:intersectional:frames}
Let $(W,{\leq},{R})$ be a frame.
Let $(W^{\prime},{\leq^{\prime}},{R^{\prime}})$ be the frame defined by $W^{\prime}{=}
$\linebreak$
W$, ${\leq^{\prime}}{=}{\leq}$ and ${R^{\prime}}{=}({\leq}{\circ}{R}{\circ}{\leq}){\cap}({\geq}{\circ}{R}{\circ}{\geq})$.
\begin{enumerate}
\item If $(W,{\leq},{R})$ is up-reflexive and down-reflexive then $(W^{\prime},{\leq^{\prime}},{R^{\prime}})$ is reflexive,
\item If $(W,{\leq},{R})$ is up-symmetric and down-symmetric then $(W^{\prime},{\leq^{\prime}},{R^{\prime}})$ is symmetric,
\item If $(W,{\leq},{R})$ is up-transitive and down-transitive then $(W^{\prime},{\leq^{\prime}},{R^{\prime}})$ is transitive.
\end{enumerate}
\end{lemma}
\begin{proof}
$\mathbf{(1)}$~Suppose $(W,{\leq},{R})$ is up-reflexive and down-reflexive.
For the sake of the contradiction, suppose $(W^{\prime},{\leq^{\prime}},{R^{\prime}})$ is not reflexive.
Hence, there exists $s{\in}W^{\prime}$ such that not $s{R^{\prime}}s$.
Since $(W,{\leq},{R})$ is up-reflexive and down-reflexive, then $s{\leq}{\circ}{R}{\circ}{\leq}s$ and $s{\geq}{\circ}{R}{\circ}{\geq}s$.
Thus, $s{R^{\prime}}s$: a contradiction.
$\mathbf{(2)}$~Suppose $(W,{\leq},{R})$ is up-symmetric and down-symmetric.
For the sake of the contradiction, suppose $(W^{\prime},{\leq^{\prime}},{R^{\prime}})$ is not symmetric.
Consequently, there exists $s,t{\in}W^{\prime}$ such that $s{R^{\prime}}t$ and not $t{R^{\prime}}s$.
Hence, $s{\leq}{\circ}{R}{\circ}{\leq}t$ and $s{\geq}{\circ}{R}{\circ}{\geq}t$.
Thus, there exists $u,v{\in}W$ such that $s{\leq}u$, $u{R}v$ and $v{\leq}t$ and there exists $w,x{\in}W$ such that 
$s{\geq}w$, $w{R}x$ and $x{\geq}t$.
Since $(W,{\leq},{R})$ is up-symmetric and down-symmetric, then $v{\geq}{\circ}{R}{\circ}{\geq}u$ and $x{\leq}{\circ}{R}{\circ}{\leq}w$.
Since $s{\leq}u$, $v{\leq}t$, $s{\geq}w$ and $x{\geq}t$, then $t{\leq}{\circ}{R}{\circ}{\leq}s$ and $t{\geq}{\circ}{R}{\circ}{\geq}s$.
Consequently, $t{R^{\prime}}s$: a contradiction.
$\mathbf{(3)}$~Suppose $(W,{\leq},{R})$ is up-transitive and down-transitive.
For the sake of the contradiction, suppose $(W^{\prime},{\leq^{\prime}},{R^{\prime}})$ is not transitive.
Hence, there exists $s,t,u{\in}W^{\prime}$ such that $s{R^{\prime}}t$, $t{R^{\prime}}u$ and not $s{R^{\prime}}u$.
Thus, $s{\leq}{\circ}{R}{\circ}{\leq}t$, $s{\geq}{\circ}{R}{\circ}{\geq}t$, $t{\leq}{\circ}{R}{\circ}{\leq}u$ and $t{\geq}{\circ}{R}{\circ}{\geq}u$.
Consequently, $s{\leq}{\circ}{R}{\circ}{\leq}{\circ}{R}{\circ}{\leq}u$ and $s{\geq}{\circ}{R}{\circ}{\geq}{\circ}{R}{\circ}{\geq}u$.
Hence, there exists $v,w{\in}W$ such that $s{\leq}v$, $v{R}{\circ}{\leq}{\circ}{R}w$ and $w{\leq}u$ and there exists $x,y{\in}W$ such that $s{\geq}x$, $x{R}{\circ}{\geq}{\circ}{R}y$ and $y{\geq}u$.
Since $(W,{\leq},{R})$ is up-transitive and down-transitive, then $v{\leq}{\circ}{R}{\circ}{\leq}w$ and $x{\geq}{\circ}{R}{\circ}{\geq}y$.
Since $s{\leq}v$, $w{\leq}u$, $s{\geq}x$ and $y{\geq}u$, then $s{\leq}{\circ}{R}{\circ}{\leq}u$ and $s{\geq}{\circ}{R}{\circ}{\geq}u$.
Thus, $s{R^{\prime}}u$: a contradiction.
\medskip
\end{proof}
\begin{definition}[Valuations and models]
A {\em valuation on a frame $(W,{\leq},{R})$}\/ is a function $V\ :\ \At{\longrightarrow}\wp(W)$ such that for all atoms $p$, $V(p)$ is $\leq$-closed.
A {\em model}\/ is a $4$-tuple consisting of the $3$ components of a frame and a valuation on that frame.
A {\em model based on the frame $(W,{\leq},{R})$}\/ is a model of the form $(W,{\leq},{R},V)$.
\end{definition}
\begin{definition}[Satisfiability]
With respect to a model $(W,{\leq},{R},V)$, for all $s{\in}W$ and for all formulas $A$, the {\em satisfiability of $A$ at $s$ in $(W,{\leq},{R},V)$}\/ (in symbols $(W,{\leq},{R},V),
$\linebreak$
s{\models}A$) is inductively defined as follows:
\begin{itemize}
\item $(W,{\leq},{R},V),s{\models}p$ if and only if $s{\in}V(p)$,
\item $(W,{\leq},{R},V),s{\models}A{\rightarrow}B$ if and only if for all $t{\in}W$, if $s{\leq}t$ and $(W,{\leq},{R},V),t{\models}
$\linebreak$
A$ then $(W,{\leq},{R},V),t{\models}B$,
\item $(W,{\leq},{R},V),s{\models}{\top}$,
\item $(W,{\leq},{R},V),s{\not\models}{\bot}$,
\item $(W,{\leq},{R},V),s{\models}A{\vee}B$ if and only if either $(W,{\leq},{R},V),s{\models}A$, or $(W,{\leq},{R},
$\linebreak$
V),s{\models}B$,
\item $(W,{\leq},{R},V),s{\models}A{\wedge}B$ if and only if $(W,{\leq},{R},V),s{\models}A$ and $(W,{\leq},{R},V),s{\models}
$\linebreak$
B$,
\item $(W,{\leq},{R},V),s{\models}{\square}A$ if and only if for all $t{\in}W$, if $s{\leq}{\circ}{R}t$ then $(W,{\leq},{R},V),t
$\linebreak$
{\models}A$,
\item $(W,{\leq},{R},V),s{\models}{\lozenge}A$ if and only if there exists $t{\in}W$ such that $s{\geq}{\circ}{R}t$ and $(W,
$\linebreak$
{\leq},{R},V),t{\models}A$.
\end{itemize}
\end{definition}
For all models $(W,{\leq},{R},V)$, for all $s{\in}W$ and for all formulas $A$, we write $s{\models}A$ instead of $(W,{\leq},{R},V),s{\models}A$ when $(W,{\leq},{R},V)$ is clear from the context.
\begin{lemma}[Heredity Property]\label{lemma:HB:monotonicity}
Let $(W,{\leq},{R},V)$ be a model.
For all formulas $A$ and for all $s,t{\in}W$, if $s{\models}A$ and $s{\leq}t$ then $t{\models}A$.
\end{lemma}
\begin{proof}
By induction on $A$.
\medskip
\end{proof}
\begin{lemma}\label{lemma:satisfiability:does:not:change:if:intersectional:update:of:a:model}
Let $(W,{\leq},{R},V)$ be a model.
Let $(W^{\prime},{\leq^{\prime}},{R^{\prime}},V^{\prime})$ be the model defined by $W^{\prime}{=}W$, ${\leq^{\prime}}{=}{\leq}$, ${R^{\prime}}{=}({\leq}{\circ}{R}{\circ}{\leq}){\cap}({\geq}{\circ}{R}{\circ}{\geq})$ and $V^{\prime}{=}V$.
For all formulas $A$ and for all $s{\in}W^{\prime}$, $(W^{\prime},{\leq^{\prime}},{R^{\prime}},V^{\prime}),s{\models}A$ if and only if $(W,{\leq},{R},V),s{\models}A$.
\end{lemma}
\begin{proof}
By induction on $A$.
\medskip
\end{proof}
As mentioned in Footnote~$2$, our definition of the satisfiability of $\lozenge$-formulas is similar to the definition of the satisfiability of $\lozenge$-formulas considered by P\v{r}enosil~\cite{Prenosil:2014}.
However, our semantics is more general, seeing that we do not restrict the discussion to the class of all frames $(W,{\leq},{R})$ such that $({\leq}{\circ}{R}{\circ}{\leq}){\cap}({\geq}{\circ}{R}{\circ}{\geq}){=}{R}$.
As for Fischer Servi~\cite{FischerServi:1984} and Wijesekera~\cite{Wijesekera:1990}, they have defined the satisfiability of $\lozenge$-formulas as follows:
\begin{itemize}
\item $(W,{\leq},{R},V),s{\models_{\mathtt{FS}}}{\lozenge}A$ if and only if there exists $t{\in}W$ such that $s{R}t$ and $(W,{\leq},
$\linebreak$
{R},V),t{\models_{\mathtt{FS}}}A$,
\item $(W,{\leq},{R},V),s{\models_{\mathtt{W}}}{\lozenge}A$ if and only if for all $t{\in}W$, if $s{\leq}t$ then there exists $u{\in}W$ such that $t{R}u$ and $(W,{\leq},{R},V),u{\models_{\mathtt{W}}}A$.
\end{itemize}
The definition of the satisfiability of formulas considered by Fischer Servi necessitates to restrict the discussion to the class of all forward confluent frames, otherwise the Heredity Property described in Lemma~\ref{lemma:HB:monotonicity} would not hold.
The definition of the satisfiability of formulas considered by Wijesekera does not necessitate to restrict the discussion to a specific class of frames.
The reader may easily verify that in the class of all forward confluent frames, the definition of the satisfiability of formulas considered by Fischer Servi, the definition of the satisfiability of formulas considered by Wijesekera and our definition of the satisfiability of formulas are equivalent.
\begin{definition}[Truth and validity]
A formula $A$ is {\em true in a model $(W,{\leq},{R},V)$}\/ (in symbols $(W,{\leq},{R},V){\models}A$) if for all $s{\in}W$, $s{\models}A$.
A formula $A$ is {\em valid in a frame $(W,{\leq},{R})$}\/ (in symbols $(W,{\leq},{R}){\models}A$) if for all models $(W,{\leq},{R},V)$ based on $(W,{\leq},
$\linebreak$
{R})$, $(W,{\leq},{R},V){\models}A$.
A set $\Sigma$ of formulas is {\em valid in a frame $(W,{\leq},{R})$}\/ (in symbols $(W,{\leq},{R}){\models}\Sigma$) if for all formulas $A$, if $A{\in}\Sigma$ then $(W,{\leq},{R}){\models}A$.
A formula $A$ is {\em valid on a class ${\mathcal C}$ of frames}\/ (in symbols ${\mathcal C}{\models}A$) if for all frames $(W,{\leq},{R})$ in ${\mathcal C}$, $(W,{\leq},{R}){\models}A$.
For all classes ${\mathcal C}$ of frames, let $\Log({\mathcal C}){=}\{A{\in}\Fo\ :\ {\mathcal C}{\models}A\}$.
\end{definition}
In Section~\ref{section:axiomatization}, we propose finite axiomatizations of $\Log({\mathcal C}_{\allfra})$, $\Log({\mathcal C}_{\fcfra})$, $\Log({\mathcal C}_{\bcfra})$, $\Log({\mathcal C}_{\dcfra})$, $\Log({\mathcal C}_{\fbcfra})$, $\Log({\mathcal C}_{\fdcfra})$, $\Log({\mathcal C}_{\bdcfra})$, $\Log({\mathcal C}_{\fbdcfra})$, $\Log({\mathcal C}_{\qfcfra})$, $\Log({\mathcal C}_{\qbcfra})$,
\linebreak$
\Log({\mathcal C}_{\qdcfra})$, $\Log({\mathcal C}_{\reflexive})$, $\Log({\mathcal C}_{\symmetric})$, $\Log({\mathcal C}_{\transitive})$, $\Log({\mathcal C}_{\reflexive}{\cap}{\mathcal C}_{\symmetric})$, $\Log({\mathcal C}_{\reflexive}{\cap}{\mathcal C}_{\transitive})$,
\linebreak$
\Log({\mathcal C}_{\partition})$, $\Log({\mathcal C}_{\ureflexive})$, $\Log({\mathcal C}_{\dreflexive})$, $\Log({\mathcal C}_{\usymmetric})$, $\Log({\mathcal C}_{\dsymmetric})$, $\Log({\mathcal C}_{\utransitive})$ and
\linebreak$
\Log({\mathcal C}_{\dtransitive})$.
We do not know how to axiomatize $\Log({\mathcal C}_{\ucfra})$, $\Log({\mathcal C}_{\fucfra})$, $\Log({\mathcal C}_{\bucfra})$, $\Log({\mathcal C}_{\ducfra})$, $\Log({\mathcal C}_{\fbucfra})$, $\Log({\mathcal C}_{\fducfra})$, $\Log({\mathcal C}_{\bducfra})$, $\Log({\mathcal C}_{\fbducfra})$, $\Log({\mathcal C}_{\qucfra})$ and
\linebreak$
\Log({\mathcal C}_{\symmetric}{\cap}{\mathcal C}_{\transitive})$.
\begin{lemma}\label{lemma:some:specific:formulas:valid:non:valid}
Let ${\mathcal C}$ be a class of frames.
If ${\mathcal C}$ is contained in ${\mathcal C}_{\dcfra}$ and ${\mathcal C}$ contains ${\mathcal C}_{\fbducfra}$ then ${\square}{\bot}{\vee}{\lozenge}{\top}$ is in $\Log({\mathcal C})$, ${\square}{\bot}$ is not in $\Log({\mathcal C})$ and ${\lozenge}{\top}$ is not in $\Log({\mathcal C})$.
\end{lemma}
\begin{proof}
Suppose ${\mathcal C}$ is contained in ${\mathcal C}_{\dcfra}$ and ${\mathcal C}$ contains ${\mathcal C}_{\fbducfra}$.
For the sake of the contradiction, suppose either ${\square}{\bot}{\vee}{\lozenge}{\top}$ is not in $\Log({\mathcal C})$, or ${\square}{\bot}$ is in $\Log({\mathcal C})$, or ${\lozenge}{\top}$ is in $\Log({\mathcal C})$.
In the first case, there exists a frame $(W,{\leq},{R})$ in ${\mathcal C}$ such that $(W,{\leq},{R}){\not\models}{\square}{\bot}
$\linebreak$
{\vee}{\lozenge}{\top}$.
Since ${\mathcal C}$ is contained in ${\mathcal C}_{\dcfra}$, then $(W,{\leq},{R})$ is in ${\mathcal C}_{\dcfra}$.
Moreover, there exists a valuation $V\ :\ \At{\longrightarrow}\wp(W)$ on $(W,{\leq},{R})$ such that $(W,{\leq},{R},V){\not\models}{\square}{\bot}{\vee}{\lozenge}{\top}$.
Hence, there exists $s{\in}W$ such that $s{\not\models}{\square}{\bot}{\vee}{\lozenge}{\top}$.
Thus, $s{\not\models}{\square}{\bot}$ and $s{\not\models}{\lozenge}{\top}$.
Consequently, there exists $t{\in}W$ such that $s{\leq}{\circ}{R}t$.
Since $(W,{\leq},{R})$ is in ${\mathcal C}_{\dcfra}$, then $s{R}{\circ}{\leq}t$.
Hence, there exists $u{\in}W$ such that $s{R}u$.
Thus, $s{\models}{\lozenge}{\top}$: a contradiction.
In the second case, let $(W^{\prime},{\leq^{\prime}},{R^{\prime}})$ be the frame defined by $W^{\prime}{=}\{a\}$ and ${R^{\prime}}{=}\{(a,a)\}$.
Obviously, $(W^{\prime},{\leq^{\prime}},{R^{\prime}})$ is in ${\mathcal C}_{\fbducfra}$.
Since ${\mathcal C}$ contains ${\mathcal C}_{\fbducfra}$, then $(W^{\prime},{\leq^{\prime}},{R^{\prime}})$ is in ${\mathcal C}$.
Since ${\square}{\bot}$ is in $\Log({\mathcal C})$, then $(W^{\prime},{\leq^{\prime}},{R^{\prime}}){\models}{\square}{\bot}$.
Let $V^{\prime}\ :\ \At{\longrightarrow}\wp(W^{\prime})$ be a valuation on $(W^{\prime},{\leq^{\prime}},{R^{\prime}})$.
Since $(W^{\prime},{\leq^{\prime}},{R^{\prime}}){\models}{\square}{\bot}$, then $(W^{\prime},{\leq^{\prime}},{R^{\prime}},V^{\prime}){\models}{\square}{\bot}$.
Consequently, $a{\models}{\square}{\bot}$.
Hence, not $a{R^{\prime}}a$: a contradiction.
In the third case, let $(W^{\prime\prime},{\leq^{\prime\prime}},{R^{\prime\prime}})$ be the frame defined by $W^{\prime\prime}{=}\{a\}$ and ${R^{\prime\prime}}{=}\emptyset$.
Obviously, $(W^{\prime\prime},{\leq^{\prime\prime}},{R^{\prime\prime}})$ is in ${\mathcal C}_{\fbducfra}$.
Since ${\mathcal C}$ contains ${\mathcal C}_{\fbducfra}$, then $(W^{\prime\prime},{\leq^{\prime\prime}},{R^{\prime\prime}})$ is in ${\mathcal C}$.
Since ${\lozenge}{\top}$ is in $\Log({\mathcal C})$, then $(W^{\prime\prime},{\leq^{\prime\prime}},{R^{\prime\prime}}){\models}{\lozenge}{\top}$.
Let $V^{\prime\prime}\ :\ \At{\longrightarrow}\wp(W^{\prime\prime})$ be a valuation on $(W^{\prime\prime},{\leq^{\prime\prime}},{R^{\prime\prime}})$.
Since $(W^{\prime\prime},{\leq^{\prime\prime}},{R^{\prime\prime}}){\models}{\lozenge}{\top}$, then $(W^{\prime\prime},{\leq^{\prime\prime}},{R^{\prime\prime}},V^{\prime\prime}){\models}{\lozenge}{\top}$.
Thus, $a{\models}{\lozenge}{\top}$.
Consequently, $a{R^{\prime\prime}}a$: a contradiction.
\medskip
\end{proof}
\begin{lemma}\label{lemma:same:valid:formulas}
\begin{enumerate}
\item $\Log({\mathcal C}_{\reflexive}){=}\Log({\mathcal C}_{\ureflexive}{\cap}{\mathcal C}_{\dreflexive})$,
\item $\Log({\mathcal C}_{\symmetric}){=}\Log({\mathcal C}_{\usymmetric}{\cap}{\mathcal C}_{\dsymmetric})$,
\item $\Log({\mathcal C}_{\transitive}){=}\Log({\mathcal C}_{\allfra})$,
\item $\Log({\mathcal C}_{\reflexive}{\cap}{\mathcal C}_{\symmetric}){=}\Log({\mathcal C}_{\ureflexive}{\cap}{\mathcal C}_{\dreflexive}{\cap}{\mathcal C}_{\usymmetric}{\cap}{\mathcal C}_{\dsymmetric})$,
\item $\Log({\mathcal C}_{\reflexive}{\cap}{\mathcal C}_{\transitive}){=}\Log({\mathcal C}_{\ureflexive}{\cap}{\mathcal C}_{\dreflexive})$,
\item $\Log({\mathcal C}_{\partition}){=}\Log({\mathcal C}_{\ureflexive}{\cap}{\mathcal C}_{\dreflexive}{\cap}{\mathcal C}_{\usymmetric}{\cap}{\mathcal C}_{\dsymmetric})$.
\end{enumerate}
\end{lemma}
\begin{proof}
$\mathbf{(1)}$~Since ${\mathcal C}_{\ureflexive}{\cap}{\mathcal C}_{\dreflexive}$ contains ${\mathcal C}_{\reflexive}$, then it suffices to demonstrate that
\linebreak$
\Log({\mathcal C}_{\reflexive}){\subseteq}\Log({\mathcal C}_{\ureflexive}{\cap}{\mathcal C}_{\dreflexive})$.
For the sake of the contradiction, suppose $\Log({\mathcal C}_{\reflexive}){\not\subseteq}
$\linebreak$
\Log({\mathcal C}_{\ureflexive}{\cap}{\mathcal C}_{\dreflexive})$.
Hence, there exists a formula $A$ such that ${\mathcal C}_{\reflexive}{\models}A$ and ${\mathcal C}_{\ureflexive}{\cap}
$\linebreak$
{\mathcal C}_{\dreflexive}{\not\models}A$.
Thus, there exists an up-reflexive and down-reflexive frame $(W,{\leq},{R})$ such that $(W,{\leq},{R}){\not\models}A$.
Consequently, there exists a valuation $V\ :\ \At{\longrightarrow}\wp(W)$ on $(W,{\leq},{R})$ such that $(W,{\leq},{R},V){\not\models}A$.
Hence, there exists $s{\in}W$ such that $(W,{\leq},{R},
$\linebreak$
V),s{\not\models}A$.
Let $(W^{\prime},{\leq^{\prime}},{R^{\prime}})$ be the frame defined by $W^{\prime}{=}W$, ${\leq^{\prime}}{=}\leq$ and ${R^{\prime}}{=}({\leq}{\circ}{R}{\circ}
$\linebreak$
{\leq}){\cap}({\geq}{\circ}{R}{\circ}{\geq})$.
Since $(W,{\leq},{R})$ is up-reflexive and down-reflexive, then by Lemma~\ref{lemma:about:up:and:down:frames:and:the:corresponding:intersectional:frames}, $(W^{\prime},{\leq^{\prime}},{R^{\prime}})$ is reflexive.
Let $V^{\prime}\ :\ \At{\longrightarrow}\wp(W^{\prime})$ be the valuation on $(W^{\prime},{\leq^{\prime}},{R^{\prime}})$ such that for all atoms $p$, $V^{\prime}(p){=}V(p)$.
Since $(W,{\leq},{R},V),s{\not\models}A$, then by Lemma~\ref{lemma:satisfiability:does:not:change:if:intersectional:update:of:a:model}, $(W^{\prime},{\leq^{\prime}},{R^{\prime}},V^{\prime}),s{\not\models}A$.
Thus, $(W^{\prime},{\leq^{\prime}},{R^{\prime}},V^{\prime}){\not\models}A$.
Consequently, $(W^{\prime},{\leq^{\prime}},{R^{\prime}}){\not\models}A$.
Since $(W^{\prime},{\leq^{\prime}},{R^{\prime}})$ is reflexive, then ${\mathcal C}_{\reflexive}{\not\models}A$: a contradiction.
$\mathbf{(2)}$~Similar to the proof of Item~$\mathbf{(1)}$.
$\mathbf{(3)}$~Since ${\mathcal C}_{\allfra}$ contains ${\mathcal C}_{\transitive}$, then it suffices to demonstrate that $\Log({\mathcal C}_{\transitive}){\subseteq}
$\linebreak$
\Log({\mathcal C}_{\allfra})$.
For the sake of the contradiction, suppose $\Log({\mathcal C}_{\transitive}){\not\subseteq}\Log({\mathcal C}_{\allfra})$.
Hence, there exists a formula $A$ such that ${\mathcal C}_{\transitive}{\models}A$ and ${\mathcal C}_{\allfra}{\not\models}A$ and .
Thus, there exists a frame $(W,{\leq},{R})$ such that $(W,{\leq},{R}){\not\models}A$.
Consequently, there exists a valuation $V\ :\ \At{\longrightarrow}\wp(W)$ on $(W,{\leq},{R})$ such that $(W,{\leq},{R},V){\not\models}A$.
Hence, there exists $s{\in}W$ such that $(W,{\leq},{R},V),s{\not\models}A$.
Let $(W^{\prime},{\leq^{\prime}},{R^{\prime}})$ be the frame defined by $W^{\prime}{=}W{\times}\{0,1\}$, ${\leq^{\prime}}{=}\{((t,j),(u,k)):\ (t,j),(u,k){\in}W^{\prime}$ are such that $t{\leq}u\}$ and ${R^{\prime}}{=}\{((t,j),(u,k)):\ (t,j),(u,k){\in}W^{\prime}$ are such that $t{R}u$, $j{=}0$ and $k{=}1\}$.
Obviously, $(W^{\prime},{\leq^{\prime}},{R^{\prime}})$ is transitive.
Let $V^{\prime}\ :\ \At{\longrightarrow}\wp(W^{\prime})$ be the valuation on $(W^{\prime},{\leq^{\prime}},{R^{\prime}})$ such that for all atoms $p$, $V^{\prime}(p){=}V(p){\times}\{0,1\}$.
\begin{claim}
For all formulas $B$ and for all $(t,j){\in}W^{\prime}$, $(W^{\prime},{\leq^{\prime}},{R^{\prime}},V^{\prime}),(t,j){\models}B$ if and only if $(W,{\leq},{R},V),t{\models}B$.
\end{claim}
\begin{proofclaim}
By induction on $B$.
\end{proofclaim}
Since $(W,{\leq},{R},V),s{\not\models}A$, then $(W^{\prime},{\leq^{\prime}},{R^{\prime}},V^{\prime}),(s,0){\not\models}A$.
Thus, $(W^{\prime},{\leq^{\prime}},{R^{\prime}},
$\linebreak$
V^{\prime}){\not\models}A$.
Consequently, $(W^{\prime},{\leq^{\prime}},{R^{\prime}}){\not\models}A$.
Since $(W^{\prime},{\leq^{\prime}},{R^{\prime}})$ is transitive, then ${\mathcal C}_{\transitive}{\not\models}
$\linebreak$
A$: a contradiction.
$\mathbf{(4)}$~Similar to the proof of Item~$\mathbf{(1)}$.
$\mathbf{(5)}$~Since ${\mathcal C}_{\reflexive}$ contains ${\mathcal C}_{\reflexive}{\cap}{\mathcal C}_{\transitive}$, then by Item~$\mathbf{(1)}$, it suffices to demonstrate that $\Log({\mathcal C}_{\reflexive}{\cap}{\mathcal C}_{\transitive}){\subseteq}\Log({\mathcal C}_{\reflexive})$.
For the sake of the contradiction, suppose $\Log({\mathcal C}_{\reflexive}{\cap}
$\linebreak$
{\mathcal C}_{\transitive}){\not\subseteq}\Log({\mathcal C}_{\reflexive})$.
Hence, there exists a formula $A$ such that ${\mathcal C}_{\reflexive}{\cap}{\mathcal C}_{\transitive}{\models}A$ and ${\mathcal C}_{\reflexive}{\not\models}
$\linebreak$
A$.
Thus, there exists a reflexive frame $(W,{\leq},{R})$ such that $(W,{\leq},{R}){\not\models}A$.
Consequently, there exists a valuation $V\ :\ \At{\longrightarrow}\wp(W)$ on $(W,{\leq},{R})$ such that $(W,{\leq},{R},
$\linebreak$
V){\not\models}A$.
Hence, there exists $s{\in}W$ such that $(W,{\leq},{R},V),s{\not\models}A$.
Let $(W^{\prime},{\leq^{\prime}},{R^{\prime}})$ be the frame defined by $W^{\prime}{=}W{\times}\{0,1\}$, ${\leq^{\prime}}{=}\{((t,j),(u,k)):\ (t,j),(u,k){\in}W^{\prime}$ are such that $t{\leq}u\}$ and ${R^{\prime}}{=}\{((t,j),(u,k)):\ (t,j),(u,k){\in}W^{\prime}$ are such that either $(t,j){=}(u,k)$, or $t{R}u$, $j{=}0$ and $k{=}1\}$.
Obviously, $(W^{\prime},{\leq^{\prime}},{R^{\prime}})$ is reflexive and transitive.
Let $V^{\prime}\ :\ \At{\longrightarrow}\wp(W^{\prime})$ be the valuation on $(W^{\prime},{\leq^{\prime}},{R^{\prime}})$ such that for all atoms $p$, $V^{\prime}(p){=}V(p){\times}\{0,1\}$.
\begin{claim}
For all formulas $B$ and for all $(t,j){\in}W^{\prime}$, $(W^{\prime},{\leq^{\prime}},{R^{\prime}},V^{\prime}),(t,j){\models}B$ if and only if $(W,{\leq},{R},V),t{\models}B$.
\end{claim}
\begin{proofclaim}
By induction on $B$.
\end{proofclaim}
Since $(W,{\leq},{R},V),s{\not\models}A$, then $(W^{\prime},{\leq^{\prime}},{R^{\prime}},V^{\prime}),(s,0){\not\models}A$.
Thus, $(W^{\prime},{\leq^{\prime}},{R^{\prime}},
$\linebreak$
V^{\prime}){\not\models}A$.
Consequently, $(W^{\prime},{\leq^{\prime}},{R^{\prime}}){\not\models}A$.
Since $(W^{\prime},{\leq^{\prime}},{R^{\prime}})$ is reflexive and transitive, then ${\mathcal C}_{\reflexive}{\cap}{\mathcal C}_{\transitive}{\not\models}A$: a contradiction.
$\mathbf{(6)}$~Since ${\mathcal C}_{\reflexive}{\cap}{\mathcal C}_{\symmetric}$ contains ${\mathcal C}_{\partition}$, then by Item~$\mathbf{(4)}$, it suffices to demonstrate that $\Log({\mathcal C}_{\partition}){\subseteq}\Log({\mathcal C}_{\reflexive}{\cap}{\mathcal C}_{\symmetric})$.
For the sake of the contradiction, suppose
\linebreak$
\Log({\mathcal C}_{\partition}){\not\subseteq}\Log({\mathcal C}_{\reflexive}{\cap}{\mathcal C}_{\symmetric})$.
Hence, there exists a formula $A$ such that ${\mathcal C}_{\partition}{\models}A$ and ${\mathcal C}_{\reflexive}{\cap}{\mathcal C}_{\symmetric}{\not\models}A$.
Thus, there exists a reflexive and symmetric frame $(W,{\leq},{R})$ such that $(W,{\leq},{R}){\not\models}A$.
Consequently, there exists a valuation $V\ :\ \At{\longrightarrow}\wp(W)$ on $(W,{\leq},{R})$ such that $(W,{\leq},{R},V){\not\models}A$.
Hence, there exists $s{\in}W$ such that $(W,{\leq},{R},
$\linebreak$
V),s{\not\models}A$.
Let $(W^{\prime},{\leq^{\prime}},{R^{\prime}})$ be the frame defined by $W^{\prime}{=}\{(t,\{t,u\}):\ t,u{\in}W$ are such that $t{R}u\}$, ${\leq^{\prime}}{=}\{((t,\{t,u\}),(v,\{v,w\})):\ (t,\{t,u\}),(v,\{v,w\}){\in}W^{\prime}$ are such that $t{\leq}v\}$ and ${R^{\prime}}{=}\{((t,\{t,u\}),(v,\{v,w\})):\ (t,\{t,u\}),(v,\{v,w\}){\in}W^{\prime}$ are such that $t{R}v$ and $\{t,u\}{=}\{v,w\}\}$.
Obviously, $(W^{\prime},{\leq^{\prime}},{R^{\prime}})$ is a partition.
Let $V^{\prime}\ :\ \At{\longrightarrow}\wp(W^{\prime})$ be the valuation on $(W^{\prime},{\leq^{\prime}},{R^{\prime}})$ such that for all atoms $p$, $V^{\prime}(p){=}\{(t,
$\linebreak$
\{t,u\}):\ (t,\{t,u\}){\in}W^{\prime}$ is such that $t{\in}V(p)\}$.
\begin{claim}
For all formulas $B$ and for all $(t,\{t,u\}){\in}W^{\prime}$, $(W^{\prime},{\leq^{\prime}},{R^{\prime}},V^{\prime}),(t,\{t,u\}){\models}
$\linebreak$
B$ if and only if $(W,{\leq},{R},V),t{\models}B$.
\end{claim}
\begin{proofclaim}
By induction on $B$.
\end{proofclaim}
Since $(W,{\leq},{R},V),s{\not\models}A$, then $(W^{\prime},{\leq^{\prime}},{R^{\prime}},V^{\prime}),(s,\{s\}){\not\models}A$.
Thus, $(W^{\prime},{\leq^{\prime}},{R^{\prime}},
$\linebreak$
V^{\prime}){\not\models}A$.
Consequently, $(W^{\prime},{\leq^{\prime}},{R^{\prime}}){\not\models}A$.
Since $(W^{\prime},{\leq^{\prime}},{R^{\prime}})$ is a partition, then ${\mathcal C}_{\partition}
$\linebreak$
{\not\models}A$: a contradiction.
\medskip
\end{proof}
We do not know whether ${\mathcal C}_{\symmetric}{\cap}{\mathcal C}_{\transitive}$ and ${\mathcal C}_{\usymmetric}{\cap}{\mathcal C}_{\dsymmetric}$ validate the same formulas.
\begin{definition}[Generated subframes]
A frame $(W^{\prime},{\leq^{\prime}},{R^{\prime}})$ is a {\em generated subframe of a frame $(W,{\leq},{R})$}\/ if the following conditions are satisfied:
\begin{itemize}
\item $W^{\prime}{\subseteq}W$,
\item ${\leq^{\prime}}{\subseteq}{\leq}$,
\item ${R^{\prime}}{\subseteq}{R}$,
\item for all $s^{\prime}{\in}W^{\prime}$ and for all $t{\in}W$, if $s^{\prime}{\leq}t$ then $t{\in}W^{\prime}$ and $s^{\prime}{\leq^{\prime}}t$,
\item for all $s^{\prime}{\in}W^{\prime}$ and for all $t{\in}W$, if $s^{\prime}{\geq}t$ then $t{\in}W^{\prime}$ and $s^{\prime}{\geq^{\prime}}t$,
\item for all $s^{\prime}{\in}W^{\prime}$ and for all $t{\in}W$, if $s^{\prime}{R}t$ then $t{\in}W^{\prime}$ and $s^{\prime}{R^{\prime}}t$.
\end{itemize}
For all frames $(W,{\leq},{R})$ and for all $s{\in}W$, let $(W_{s},{\leq_{s}},{R_{s}})$ be the least generated subframe of $(W,{\leq},{R})$ containing $s$.
\end{definition}
\begin{lemma}\label{subframe:lemma}
Let $(W,{\leq},{R})$ be a frame and $s{\in}W$.
Let $V\ :\ \At{\longrightarrow}\wp(W)$ be a valuation on $(W,{\leq},{R})$ and $V_{s}\ :\ \At{\longrightarrow}\wp(W_{s})$ be the valuation on $(W_{s},{\leq_{s}},{R_{s}})$ such that for all atoms $p$, $V_{s}(p){=}V(p){\cap}W_{s}$.
Let $A$ be a formula.
For all $t{\in}W_{s}$, $(W_{s},{\leq_{s}},{R_{s}},V_{s}),t{\models}A$ if and only if $(W,{\leq},{R},V),t{\models}A$.
\end{lemma}
\begin{proof}
By induction on $A$.
\medskip
\end{proof}
\section{Definability}\label{section:definability}
In this section, we study the correspondence between elementary conditions on frames and formulas.
In particular, we prove the modal definability results described in Tables~\ref{table:modal:definability:results:a}, \ref{table:modal:definability:results:b}, \ref{table:modal:definability:results:c} and~\ref{table:modal:definability:results:d}.
\begin{definition}[Modal definability]
A class ${\mathcal C}$ of frames is {\em modally definable}\/ if there exists a formula $A$ such that for all frames $(W,{\leq},{R})$, $(W,{\leq},{R})$ is in ${\mathcal C}$ if and only if $(W,{\leq},{R}){\models}A$.
In that case, we say that the formula $A$ is a {\em modal definition of ${\mathcal C}$.}
\end{definition}
\begin{table}[ht]
\begin{center}
\begin{tabular}{|c|c|}
\hline
class of frames&modal definition
\\
\hline
${\mathcal C}_{\fcfra}$&does not exist
\\
\hline
${\mathcal C}_{\bcfra}$&does not exist
\\
\hline
${\mathcal C}_{\dcfra}$&does not exist
\\
\hline
${\mathcal C}_{\ucfra}$&does not exist
\\
\hline
${\mathcal C}_{\fbcfra}$&does not exist
\\
\hline
${\mathcal C}_{\fdcfra}$&does not exist
\\
\hline
${\mathcal C}_{\fucfra}$&does not exist
\\
\hline
${\mathcal C}_{\bdcfra}$&does not exist
\\
\hline
${\mathcal C}_{\bucfra}$&does not exist
\\
\hline
${\mathcal C}_{\ducfra}$&does not exist
\\
\hline
${\mathcal C}_{\fbdcfra}$&does not exist
\\
\hline
${\mathcal C}_{\fbucfra}$&does not exist
\\
\hline
${\mathcal C}_{\fducfra}$&does not exist
\\
\hline
${\mathcal C}_{\bducfra}$&does not exist
\\
\hline
${\mathcal C}_{\fbducfra}$&does not exist
\\
\hline
\end{tabular}
\caption{Modal definability results proved in Lemma~\ref{lemma:fc:bc:fbc:not:definable}}\label{table:modal:definability:results:a}
\end{center}
\end{table}
\begin{table}[ht]
\begin{center}
\begin{tabular}{|c|c|}
\hline
class of frames&modal definition
\\
\hline
${\mathcal C}_{\qfcfra}$&${\lozenge}(p{\rightarrow}q){\rightarrow}({\square}p{\rightarrow}{\lozenge}q)$
\\
\hline
${\mathcal C}_{\qbcfra}$&$({\lozenge}p{\rightarrow}{\square}q){\rightarrow}{\square}(p{\rightarrow}q)$
\\
\hline
${\mathcal C}_{\qdcfra}$&${\square}(p{\vee}q){\rightarrow}{\lozenge}p{\vee}{\square}q$
\\
\hline
${\mathcal C}_{\qucfra}$&?
\\
\hline
\end{tabular}
\caption{Modal definability results proved in Lemma~\ref{lemma:about:modal:definability:quasi:fbd:frames}}\label{table:modal:definability:results:b}
\end{center}
\end{table}
\begin{table}[ht]
\begin{center}
\begin{tabular}{|c|c|}
\hline
class of frames&modal definition
\\
\hline
${\mathcal C}_{\reflexive}$&does not exist
\\
\hline
${\mathcal C}_{\symmetric}$&does not exist
\\
\hline
${\mathcal C}_{\transitive}$&does not exist
\\
\hline
${\mathcal C}_{\reflexive}{\cap}{\mathcal C}_{\symmetric}$&does not exist
\\
\hline
${\mathcal C}_{\reflexive}{\cap}{\mathcal C}_{\transitive}$&does not exist
\\
\hline
${\mathcal C}_{\symmetric}{\cap}{\mathcal C}_{\transitive}$&does not exist
\\
\hline
${\mathcal C}_{\partition}$&does not exist
\\
\hline
\end{tabular}
\caption{Modal definability results proved in Lemma~\ref{lemma:about:modal:definability:of:ref:sym:tra:par}}\label{table:modal:definability:results:c}
\end{center}
\end{table}
\begin{table}[ht]
\begin{center}
\begin{tabular}{|c|c|}
\hline
class of frames&modal definition
\\
\hline
${\mathcal C}_{\ureflexive}$&${\square}p{\rightarrow}p$
\\
\hline
${\mathcal C}_{\dreflexive}$&$p{\rightarrow}{\lozenge}p$
\\
\hline
${\mathcal C}_{\usymmetric}$&${\lozenge}{\square}p{\rightarrow}p$
\\
\hline
${\mathcal C}_{\dsymmetric}$&$p{\rightarrow}{\square}{\lozenge}p$
\\
\hline
${\mathcal C}_{\utransitive}$&${\square}p{\rightarrow}{\square}{\square}p$
\\
\hline
${\mathcal C}_{\dtransitive}$&${\lozenge}{\lozenge}p{\rightarrow}{\lozenge}p$
\\
\hline
\end{tabular}
\caption{Modal definability results proved in Lemma~\ref{lemma:correspondence:formulas:T:B:4:elementary:conditions}}\label{table:modal:definability:results:d}
\end{center}
\end{table}
%
%
%As mentioned in Footnote~\ref{footnote:prenosil}, our definition of the satisfiability of $\lozenge$-formulas is similar to the definition of the satisfiability of $\lozenge$-formulas considered by P\v{r}enosil~\cite{Prenosil:2014}.
%Questions that P\v{r}enosil left open concern the modal definability of the following classes of frames: ${\mathcal C}_{\fcfra}$, ${\mathcal C}_{\bcfra}$, ${\mathcal C}_{\dcfra}$ and ${\mathcal C}_{\ucfra}$.
Questions that P\v{r}enosil~\cite{Prenosil:2014} does not address concern the modal definability of the following classes of frames: ${\mathcal C}_{\fcfra}$, ${\mathcal C}_{\bcfra}$, ${\mathcal C}_{\dcfra}$ and ${\mathcal C}_{\ucfra}$.
%In Lemmas~\ref{lemma:fc:bc:fbc:not:definable} and~\ref{lemma:dc:uc:duc:not:definable}, we negatively answer to these questions.
In Lemma~\ref{lemma:fc:bc:fbc:not:definable}, we negatively answer to these questions.
\begin{lemma}\label{lemma:fc:bc:fbc:not:definable}
%
%
%The following classes of frames are not modally definable: ${\mathcal C}_{\fcfra}$, ${\mathcal C}_{\bcfra}$ and ${\mathcal C}_{\fbcfra}$.
The following classes of frames are not modally definable: ${\mathcal C}_{\fcfra}$, ${\mathcal C}_{\bcfra}$, ${\mathcal C}_{\dcfra}$, ${\mathcal C}_{\ucfra}$, ${\mathcal C}_{\fbcfra}$, ${\mathcal C}_{\fdcfra}$, ${\mathcal C}_{\fucfra}$, ${\mathcal C}_{\bdcfra}$, ${\mathcal C}_{\bucfra}$, ${\mathcal C}_{\ducfra}$, ${\mathcal C}_{\fbdcfra}$, ${\mathcal C}_{\fbucfra}$, ${\mathcal C}_{\fducfra}$, ${\mathcal C}_{\bducfra}$ and ${\mathcal C}_{\fbducfra}$.
\end{lemma}
\begin{proof}
For the sake of the contradiction, suppose either ${\mathcal C}_{\fcfra}$ is modally definable, or ${\mathcal C}_{\bcfra}$ is modally definable, or ${\mathcal C}_{\dcfra}$ is modally definable, or ${\mathcal C}_{\ucfra}$ is modally definable, or ${\mathcal C}_{\fbcfra}$ is modally definable, or ${\mathcal C}_{\fdcfra}$ is modally definable, or ${\mathcal C}_{\fucfra}$ is modally definable, or ${\mathcal C}_{\bdcfra}$ is modally definable, or ${\mathcal C}_{\bucfra}$ is modally definable, or ${\mathcal C}_{\ducfra}$ is modally definable, or ${\mathcal C}_{\fbdcfra}$ is modally definable, or ${\mathcal C}_{\fbucfra}$ is modally definable, or ${\mathcal C}_{\fducfra}$ is modally definable, or ${\mathcal C}_{\bducfra}$, or ${\mathcal C}_{\fbducfra}$ is modally definable.
Hence, there exists a formula $A$ such that either $A$ modally defines ${\mathcal C}_{\fcfra}$, or $A$ modally defines ${\mathcal C}_{\bcfra}$, or $A$ modally defines ${\mathcal C}_{\dcfra}$, or $A$ modally defines ${\mathcal C}_{\ucfra}$, or $A$ modally defines ${\mathcal C}_{\fbcfra}$, or $A$ modally defines ${\mathcal C}_{\fdcfra}$, or $A$ modally defines ${\mathcal C}_{\fucfra}$, or $A$ modally defines ${\mathcal C}_{\bdcfra}$, or $A$ modally defines ${\mathcal C}_{\bucfra}$, or $A$ modally defines ${\mathcal C}_{\ducfra}$, or $A$ modally defines ${\mathcal C}_{\fbdcfra}$, or $A$ modally defines ${\mathcal C}_{\fbucfra}$, or $A$ modally defines ${\mathcal C}_{\fducfra}$, or $A$ modally defines ${\mathcal C}_{\bducfra}$, or $A$ modally defines ${\mathcal C}_{\fbducfra}$.
%Let $(W^{\prime},{\leq^{\prime}},{R^{\prime}})$ be the frame defined by $W^{\prime}{=}\{a,b,c,d\}$, $a{\leq^{\prime}}c$ and $b{\leq^{\prime}}d$ and $a{R^{\prime}}d$ and $c{R^{\prime}}b$ and $(W^{\prime\prime},{\leq^{\prime\prime}},{R^{\prime\prime}})$ be the frame defined by $W^{\prime\prime}{=}\{a,b,c,d\}$, $a{\leq^{\prime\prime}}c$ and $b{\leq^{\prime\prime}}d$ and $a{R^{\prime\prime}}d$, $c{R^{\prime\prime}}b$ and $c{R^{\prime\prime}}d$.
%Obviously, $(W^{\prime},{\leq^{\prime}},{R^{\prime}})$ is neither forward confluent, nor backward confluent and $(W^{\prime\prime},{\leq^{\prime\prime}},{R^{\prime\prime}})$ is forward confluent and backward confluent.
%Since either $A$ modally defines ${\mathcal C}_{\fcfra}$, or $A$ modally defines ${\mathcal C}_{\bcfra}$, or $A$ modally defines ${\mathcal C}_{\fbcfra}$, then $(W^{\prime},{\leq^{\prime}},{R^{\prime}}){\not\models}A$ and $(W^{\prime\prime},{\leq^{\prime\prime}},{R^{\prime\prime}}){\models}A$.
%Thus, there exists a valuation $V^{\prime}\ :\ \At{\longrightarrow}\wp(W^{\prime})$ on $(W^{\prime},{\leq^{\prime}},{R^{\prime}})$ such that $(W^{\prime},{\leq^{\prime}},{R^{\prime}},V^{\prime}){\not\models}A$.
%Consequently, there exists $s{\in}W^{\prime}$ such that $(W^{\prime},{\leq^{\prime}},{R^{\prime}},V^{\prime}),s{\not\models}A$.
%Let $V^{\prime\prime}\ :\ \At{\longrightarrow}\wp(W^{\prime\prime})$ be the valuation on $(W^{\prime\prime},{\leq^{\prime\prime}},{R^{\prime\prime}})$ such that for all atoms $p$, $V^{\prime\prime}(p){=}V^{\prime}(p)$.
Let $(W^{\prime},{\leq^{\prime}},{R^{\prime}})$ be the frame defined by $W^{\prime}{=}\{a,b,c,d,e,f\}$, $a{\leq^{\prime}}c$, $b{\leq^{\prime}}d$, $c{\leq^{\prime}}e$ and $d{\leq^{\prime}}f$ and $a{R^{\prime}}b$, $a{R^{\prime}}f$ and $e{R^{\prime}}f$ and $(W^{\prime\prime},{\leq^{\prime\prime}},{R^{\prime\prime}})$ be the frame defined by $W^{\prime\prime}{=}\{a,b,c,d,e,f\}$, $a{\leq^{\prime\prime}}c$, $b{\leq^{\prime\prime}}d$, $c{\leq^{\prime\prime}}e$ and $d{\leq^{\prime\prime}}f$ and $a{R^{\prime\prime}}b$, $a{R^{\prime\prime}}d$, $a{R^{\prime\prime}}f$, $c{R^{\prime\prime}}f$ and $e{R^{\prime\prime}}f$.
Obviously, $(W^{\prime},{\leq^{\prime}},{R^{\prime}})$ is neither forward confluent, nor backward confluent, nor downward confluent, nor upward confluent and $(W^{\prime\prime},{\leq^{\prime\prime}},{R^{\prime\prime}})$ is forward confluent, backward confluent downward confluent and upward confluent.
Since either $A$ modally defines ${\mathcal C}_{\fcfra}$, or $A$ modally defines ${\mathcal C}_{\bcfra}$, or $A$ modally defines ${\mathcal C}_{\dcfra}$, or $A$ modally defines ${\mathcal C}_{\ucfra}$, or $A$ modally defines ${\mathcal C}_{\fbcfra}$, or $A$ modally defines ${\mathcal C}_{\fdcfra}$, or $A$ modally defines ${\mathcal C}_{\fucfra}$, or $A$ modally defines ${\mathcal C}_{\bdcfra}$, or $A$ modally defines ${\mathcal C}_{\bucfra}$, or $A$ modally defines ${\mathcal C}_{\ducfra}$, or $A$ modally defines ${\mathcal C}_{\fbdcfra}$, or $A$ modally defines ${\mathcal C}_{\fbucfra}$, or $A$ modally defines ${\mathcal C}_{\fducfra}$, or $A$ modally defines ${\mathcal C}_{\bducfra}$, or $A$ modally defines ${\mathcal C}_{\fbducfra}$, then $(W^{\prime},{\leq^{\prime}},{R^{\prime}}){\not\models}A$ and $(W^{\prime\prime},{\leq^{\prime\prime}},{R^{\prime\prime}}){\models}A$.
Thus, there exists a valuation $V^{\prime}\ :\ \At{\longrightarrow}\wp(W^{\prime})$ on $(W^{\prime},{\leq^{\prime}},{R^{\prime}})$ such that $(W^{\prime},{\leq^{\prime}},{R^{\prime}},V^{\prime}){\not\models}A$.
Consequently, there exists $s{\in}W^{\prime}$ such that $(W^{\prime},{\leq^{\prime}},{R^{\prime}},V^{\prime}),s{\not\models}A$.
Let $V^{\prime\prime}\ :\ \At{\longrightarrow}\wp(W^{\prime\prime})$ be the valuation on $(W^{\prime\prime},{\leq^{\prime\prime}},{R^{\prime\prime}})$ such that for all atoms $p$, $V^{\prime\prime}(p){=}V^{\prime}(p)$.
\begin{claim}
For all formulas $B$,
\begin{itemize}
\item $(W^{\prime\prime},{\leq^{\prime\prime}},{R^{\prime\prime}},V^{\prime\prime}),a{\models}B$ if and only if $(W^{\prime},{\leq^{\prime}},{R^{\prime}},V^{\prime}),a{\models}B$,
\item $(W^{\prime\prime},{\leq^{\prime\prime}},{R^{\prime\prime}},V^{\prime\prime}),b{\models}B$ if and only if $(W^{\prime},{\leq^{\prime}},{R^{\prime}},V^{\prime}),b{\models}B$,
\item $(W^{\prime\prime},{\leq^{\prime\prime}},{R^{\prime\prime}},V^{\prime\prime}),c{\models}B$ if and only if $(W^{\prime},{\leq^{\prime}},{R^{\prime}},V^{\prime}),c{\models}B$,
\item $(W^{\prime\prime},{\leq^{\prime\prime}},{R^{\prime\prime}},V^{\prime\prime}),d{\models}B$ if and only if $(W^{\prime},{\leq^{\prime}},{R^{\prime}},V^{\prime}),d{\models}B$,
\item $(W^{\prime\prime},{\leq^{\prime\prime}},{R^{\prime\prime}},V^{\prime\prime}),e{\models}B$ if and only if $(W^{\prime},{\leq^{\prime}},{R^{\prime}},V^{\prime}),e{\models}B$,
\item $(W^{\prime\prime},{\leq^{\prime\prime}},{R^{\prime\prime}},V^{\prime\prime}),f{\models}B$ if and only if $(W^{\prime},{\leq^{\prime}},{R^{\prime}},V^{\prime}),f{\models}B$.
\end{itemize}
\end{claim}
\begin{proofclaim}
By induction on $B$.
\end{proofclaim}
Since $(W^{\prime},{\leq^{\prime}},{R^{\prime}},V^{\prime}),s{\not\models}A$, then $(W^{\prime\prime},{\leq^{\prime\prime}},{R^{\prime\prime}},V^{\prime\prime}),s{\not\models}A$.
Hence, $(W^{\prime\prime},{\leq^{\prime\prime}},
$\linebreak$
{R^{\prime\prime}},V^{\prime\prime}){\not\models}A$.
Thus, $(W^{\prime\prime},{\leq^{\prime\prime}},{R^{\prime\prime}}){\not\models}A$: a contradiction.
\medskip
\end{proof}
However,
\begin{lemma}\label{lemma:about:modal:definability:quasi:fbd:frames}\footnote{The formulas ${\lozenge}(p{\rightarrow}q){\rightarrow}({\square}p{\rightarrow}{\lozenge}q)$, $({\lozenge}p{\rightarrow}{\square}q){\rightarrow}{\square}(p{\rightarrow}q)$ and ${\square}(p{\vee}q){\rightarrow}{\lozenge}p{\vee}{\square}q$ are used in Section~\ref{section:axiomatization} to axiomatize $\Log({\mathcal C}_{\fcfra})$, $\Log({\mathcal C}_{\bcfra})$ and $\Log({\mathcal C}_{\dcfra})$.}
\begin{enumerate}
\item ${\lozenge}(p{\rightarrow}q){\rightarrow}({\square}p{\rightarrow}{\lozenge}q)$ modally defines ${\mathcal C}_{\qfcfra}$,
\item $({\lozenge}p{\rightarrow}{\square}q){\rightarrow}{\square}(p{\rightarrow}q)$ modally defines ${\mathcal C}_{\qbcfra}$,
\item ${\square}(p{\vee}q){\rightarrow}{\lozenge}p{\vee}{\square}q$ modally defines ${\mathcal C}_{\qdcfra}$.
\end{enumerate}
\end{lemma}
\begin{proof}
$\mathbf{(1)}$~For the sake of the contradiction, suppose ${\lozenge}(p{\rightarrow}q){\rightarrow}({\square}p{\rightarrow}{\lozenge}q)$ does not modally define ${\mathcal C}_{\qfcfra}$.
Hence, there exists a frame $(W,{\leq},{R})$ such that either $(W,{\leq},{R})$ is in ${\mathcal C}_{\qfcfra}$ and $(W,{\leq},{R}){\not\models}{\lozenge}(p{\rightarrow}q){\rightarrow}({\square}p{\rightarrow}{\lozenge}q)$, or $(W,{\leq},{R})$ is not in ${\mathcal C}_{\qfcfra}$ and $(W,{\leq},
$\linebreak$
{R}){\models}{\lozenge}(p{\rightarrow}q){\rightarrow}({\square}p{\rightarrow}{\lozenge}q)$.
The reader may easily verify that the former case leads to a contradiction.
In the latter case, there exists $s,t{\in}W$ such that $s{\geq}{\circ}{R}t$ and not $s(({\leq}{\circ}{R}{\circ}{\leq}){\cap}({\geq}{\circ}{R}{\circ}{\geq})){\circ}{\geq}t$.
Let $V\ :\ \At{\longrightarrow}\wp(W)$ be a valuation on $(W,{\leq},{R})$ such that $V(p){=}\{u{\in}W\ :\ s{\leq}{\circ}{R}{\circ}{\leq}u\}$ and $V(q){=}\{u{\in}W\ :\ s{\leq}{\circ}{R}{\circ}{\leq}u$ and $t{\leq}u\}$.
Suppose for a while that $s{\not\models}{\lozenge}(p{\rightarrow}q)$.
Since $s{\geq}{\circ}{R}t$, then $t{\not\models}p{\rightarrow}q$.
Thus, there exists $v{\in}W$ such that $t{\leq}v$, $v{\models}p$ and $v{\not\models}q$.
Consequently, $s{\leq}{\circ}{R}{\circ}{\leq}v$ and either not $s{\leq}{\circ}{R}{\circ}{\leq}v$, or not $t{\leq}v$.
Hence, not $t{\leq}v$: a contradiction.
Thus, $s{\models}{\lozenge}(p{\rightarrow}q)$.
Suppose for a while that $s{\not\models}{\square}p$.
Consequently, there exists $v{\in}W$ such that $s{\leq}{\circ}{R}v$ and $v{\not\models}p$.
Hence, not $s{\leq}{\circ}{R}{\circ}{\leq}v$: a contradiction.
Thus, $s{\models}{\square}p$.
Since $(W,{\leq},{R}){\models}{\lozenge}(p{\rightarrow}q){\rightarrow}({\square}p{\rightarrow}{\lozenge}q)$ and $s{\models}{\lozenge}(p{\rightarrow}q)$, then $s{\models}{\lozenge}q$.
Consequently, there exists $v{\in}W$ such that $s{\geq}{\circ}{R}v$ and $v{\models}q$.
Hence, $s{\leq}{\circ}{R}{\circ}{\leq}v$ and $t{\leq}v$.
Since $s{\geq}{\circ}{R}v$, then $s{(({\leq}{\circ}}{R}{{\circ}{\leq}){\cap}({\geq}{\circ}{R}{\circ}{\geq}))}{\circ}{\geq}t$: a contradiction.
$\mathbf{(2)}$~For the sake of the contradiction, suppose $({\lozenge}p{\rightarrow}{\square}q){\rightarrow}{\square}(p{\rightarrow}q)$ does not modally define ${\mathcal C}_{\qbcfra}$.
Thus, there exists a frame $(W,{\leq},{R})$ such that either $(W,{\leq},{R})$ is in ${\mathcal C}_{\qbcfra}$ and $(W,{\leq},{R}){\not\models}({\lozenge}p{\rightarrow}{\square}q){\rightarrow}{\square}(p{\rightarrow}q)$, or $(W,{\leq},{R})$ is not in ${\mathcal C}_{\qbcfra}$ and $(W,{\leq},{R}){\models}({\lozenge}p{\rightarrow}{\square}q){\rightarrow}{\square}(p{\rightarrow}q)$.
The reader may easily verify that the former case leads to a contradiction.
In the latter case, there exists $s,t{\in}W$ such that $s{R}{\circ}{\leq}t$ and not $s{\leq}{\circ}(({\leq}{\circ}{R}{\circ}{\leq}){\cap}({\geq}{\circ}{R}{\circ}{\geq}))t$.
Let $V\ :\ \At{\longrightarrow}\wp(W)$ be a valuation on $(W,{\leq},{R})$ such that $V(p){=}\{u{\in}W\ :\ t{\leq}u\}$ and $V(q){=}\{u{\in}W\ :$ there exists $v{\in}W$ such that $s{\leq}v$, $v{\leq}{\circ}{R}{\circ}{\leq}u$ and $v{\geq}{\circ}{R}{\circ}{\geq}t\}$.
Suppose for a while that $s{\not\models}{\lozenge}p{\rightarrow}{\square}q$.
Consequently, there exists $w{\in}W$ such that $s{\leq}w$, $w{\models}{\lozenge}p$ and $w{\not\models}{\square}q$.
Hence, there exists $x{\in}W$ such that $w{\geq}{\circ}{R}x$ and $x{\models}p$.
Moreover, there exists $y{\in}W$ such that $w{\leq}{\circ}{R}y$ and $y{\not\models}q$.
Thus, $t{\leq}x$.
Since $w{\geq}{\circ}{R}x$, then $w{\geq}{\circ}{R}{\circ}{\geq}t$.
Since $s{\leq}w$ and $w{\leq}{\circ}{R}y$, then $y{\in}V(q)$.
Consequently, $y{\models}q$: a contradiction.
Hence, $s{\models}{\lozenge}p{\rightarrow}{\square}q$.
Since $(W,{\leq},{R}){\models}({\lozenge}p{\rightarrow}{\square}q){\rightarrow}{\square}(p{\rightarrow}q)$, then $s{\models}{\square}(p{\rightarrow}q)$.
Since $s{R}{\circ}{\leq}t$, then there exists $z{\in}W$ such that $s{R}z$ and $z{\leq}t$.
Since $s{\models}{\square}(p{\rightarrow}q)$, then $z{\models}p{\rightarrow}q$.
Since $t{\models}p$ and $z{\leq}t$, then $t{\models}q$.
Thus, there exists $v{\in}W$ such that $s{\leq}v$, $v{\leq}{\circ}{R}{\circ}{\leq}t$ and $v{\geq}{\circ}{R}{\circ}{\geq}t$.
Consequently, $s{\leq}{\circ}{(({\leq}{\circ}{R}{\circ}{\leq}){\cap}({\geq}{\circ}{R}{\circ}{\geq}))}t$: a contradiction.
$\mathbf{(3)}$~For the sake of the contradiction, suppose ${\square}(p{\vee}q){\rightarrow}{\lozenge}p{\vee}{\square}q$ does not modally define ${\mathcal C}_{\qdcfra}$.
Hence, there exists a frame $(W,{\leq},{R})$ such that either $(W,{\leq},{R})$ is in ${\mathcal C}_{\qdcfra}$ and $(W,{\leq},{R}){\not\models}{\square}(p{\vee}q){\rightarrow}{\lozenge}p{\vee}{\square}q$, or $(W,{\leq},{R})$ is not in ${\mathcal C}_{\qbcfra}$ and $(W,{\leq},{R}){\models}
$\linebreak$
{\square}(p{\vee}q){\rightarrow}{\lozenge}p{\vee}{\square}q$.
The reader may easily verify that the former case leads to a contradiction.
In the latter case, there exists $s,t{\in}W$ such that $s{\leq}{\circ}{R}t$ and not $s(({\leq}{\circ}{R}{\circ}{\leq}){\cap}
$\linebreak$
({\geq}{\circ}{R}{\circ}{\geq})){\circ}{\leq}t$.
Let $V\ :\ \At{\longrightarrow}\wp(W)$ be a valuation on $(W,{\leq},{R})$ such that $V(p){=}
$\linebreak$
\{u{\in}W\ :$ not $s{\geq}{\circ}{R}{\circ}{\geq}u\}$ and $V(q){=}\{u{\in}W\ :$ not $t{\geq}u\}$.
Consequently, $t{\not\models}q$.
Suppose for a while that $s{\not\models}{\square}(p{\vee}q)$.
Thus, there exists $v{\in}W$ such that $s{\leq}{\circ}{R}v$ and $v{\not\models}p{\vee}q$.
Consequently, $v{\not\models}p$ and $v{\not\models}q$.
Hence, $s{\geq}{\circ}{R}{\circ}{\geq}v$ and $t{\geq}v$.
Since $s{\leq}{\circ}{R}v$, then $s(({\leq}{\circ}{R}{\circ}{\leq}){\cap}({\geq}{\circ}{R}{\circ}{\geq})){\circ}{\leq}t$: a contradiction.
Thus, $s{\models}{\square}(p{\vee}q)$.
Since $(W,{\leq},
$\linebreak$
{R}){\models}{\square}(p{\vee}q){\rightarrow}{\lozenge}p{\vee}{\square}q$, then $s{\models}{\lozenge}p{\vee}{\square}q$.
Consequently, either $s{\models}{\lozenge}p$, or $s{\models}{\square}q$.
In the former case, there exists $w{\in}W$ such that $s{\geq}{\circ}{R}w$ and $w{\models}p$.
Hence, not $s{\geq}{\circ}{R}{\circ}{\geq}w$: a contradiction.
In the latter case, since $s{\leq}{\circ}{R}t$, then $t{\models}q$: a contradiction.
\medskip
\end{proof}
Thanks to Lemma~\ref{lemma:about:modal:definability:quasi:fbd:frames}, we obtain the following results.
\begin{lemma}\label{lemma:about:axiom:fc:soundness}
Let $(W,{\leq},{R})$ be a frame.
\begin{itemize}
\item If $(W,{\leq},{R})$ is forward confluent then $(W,{\leq},{R}){\models}{\lozenge}(p{\rightarrow}q){\rightarrow}({\square}p{\rightarrow}{\lozenge}q)$.
\item if $(W,{\leq},{R})$ is backward confluent then $(W,{\leq},{R}){\models}({\lozenge}p{\rightarrow}{\square}q){\rightarrow}{\square}(p{\rightarrow}q)$,
\item if $(W,{\leq},{R})$ is downward confluent then $(W,{\leq},{R}){\models}{\square}(p{\vee}q){\rightarrow}{\lozenge}p{\vee}{\square}q$.
\end{itemize}
\end{lemma}
\begin{proof}
By Lemma~\ref{lemma:about:modal:definability:quasi:fbd:frames}, using the fact that for all frames $(W,{\leq},{R})$, if $(W,{\leq},{R})$ is forward confluent then $(W,{\leq},{R})$ is quasi-forward confluent, if $(W,{\leq},{R})$ is backward confluent then $(W,{\leq},{R})$ is quasi-backward confluent and if $(W,{\leq},{R})$ is backward confluent then $(W,{\leq},{R})$ is quasi-backward confluent.
%For the sake of the contradiction, suppose $(W,{\leq},{R})$ is downward confluent and $(W,{\leq},{R}){\not\models}{\square}(p{\vee}q){\rightarrow}{\lozenge}p{\vee}{\square}q$.
%Hence, there exists a model $(W,{\leq},{R},V)$ based on $(W,{\leq},{R})$ such that $(W,{\leq},{R},V){\not\models}{\square}(p{\vee}q){\rightarrow}{\lozenge}p{\vee}{\square}q$.
%Thus, there exists $s{\in}W$ such that $s{\not\models}{\square}(p{\vee}q){\rightarrow}{\lozenge}p{\vee}{\square}q$.
%Consequently, there exists $t{\in}W$ such that $s{\leq}t$, $t{\models}{\square}(p{\vee}q)$ and $t{\not\models}{\lozenge}p{\vee}{\square}q$.
%Hence, $t{\not\models}{\lozenge}p$ and $t{\not\models}{\square}q$.
%Thus, there exists $u{\in}W$ such that $t{\leq}{\circ}{R}u$ and $u{\not\models}q$.
%Since $(W,{\leq},{R})$ is downward confluent, then $t{R}{\circ}{\leq}u$.
%Consequently, there exists $v{\in}W$ such that $t{R}v$ and $v{\leq}u$.
%Since $t{\not\models}{\lozenge}p$ and $u{\not\models}q$, then $v{\not\models}p$ and $v{\not\models}q$.
%Hence, $v{\not\models}p{\vee}q$.
%Since $t{\models}{\square}(p{\vee}q)$ and $t{R}v$, then $v{\models}p{\vee}q$: a contradiction.
%
%
\medskip
\end{proof}
Leaving unexplored the task of determining whether ${\mathcal C}_{\qucfra}$ is modally definable, let us consider the questions of the modal definability of the following classes of frames: ${\mathcal C}_{\reflexive}$, ${\mathcal C}_{\symmetric}$ and ${\mathcal C}_{\transitive}$.
\begin{lemma}\label{lemma:about:modal:definability:of:ref:sym:tra:par}
The following classes of frames are not modally definable: $\mathbf{(1)}$~${\mathcal C}_{\reflexive}$,
\linebreak$
\mathbf{(2)}$~${\mathcal C}_{\symmetric}$, $\mathbf{(3)}$~${\mathcal C}_{\transitive}$, $\mathbf{(4)}$~${\mathcal C}_{\reflexive}{\cap}{\mathcal C}_{\symmetric}$, $\mathbf{(5)}$~${\mathcal C}_{\reflexive}{\cap}{\mathcal C}_{\transitive}$, $\mathbf{(6)}$~${\mathcal C}_{\symmetric}{\cap}{\mathcal C}_{\transitive}$ and $\mathbf{(7)}$~${\mathcal C}_{\partition}$.
\end{lemma}
\begin{proof}
$\mathbf{(1)}$~For the sake of the contradiction, suppose ${\mathcal C}_{\reflexive}$ is modally definable.
Hence, there exists a formula $A$ such that $A$ modally defines ${\mathcal C}_{\reflexive}$.
Let $(W^{\prime},{\leq^{\prime}},{R^{\prime}})$ be the frame defined by $W^{\prime}{=}\{a,b\}$, $a{\leq^{\prime}}b$ and $b{\leq^{\prime}}a$ and $a{R^{\prime}}b$ and $b{R^{\prime}}a$ and $(W^{\prime\prime},{\leq^{\prime\prime}},{R^{\prime\prime}})$ be the frame defined by $W^{\prime\prime}{=}\{a,b\}$ and $a{\leq^{\prime\prime}}b$ and $b{\leq^{\prime\prime}}a$ and $a{R^{\prime\prime}}a$, $a{R^{\prime\prime}}b$, $b{R^{\prime\prime}}a$ and $b{R^{\prime\prime}}b$.
Obviously, $(W^{\prime},{\leq^{\prime}},{R^{\prime}})$ is not reflexive and $(W^{\prime\prime},{\leq^{\prime\prime}},{R^{\prime\prime}})$ is reflexive.
Since $A$ modally defines ${\mathcal C}_{\reflexive}$, then $(W^{\prime},{\leq^{\prime}},{R^{\prime}}){\not\models}A$ and $(W^{\prime\prime},{\leq^{\prime\prime}},{R^{\prime\prime}}){\models}A$.
Thus, there exists a valuation $V^{\prime}\ :\ \At{\longrightarrow}\wp(W^{\prime})$ on $(W^{\prime},{\leq^{\prime}},{R^{\prime}})$ such that $(W^{\prime},{\leq^{\prime}},{R^{\prime}},V^{\prime}){\not\models}A$.
Consequently, there exists $s{\in}W^{\prime}$ such that $(W^{\prime},{\leq^{\prime}},{R^{\prime}},V^{\prime}),s{\not\models}A$.
Let $V^{\prime\prime}\ :\ \At{\longrightarrow}
$\linebreak$
\wp(W^{\prime\prime})$ be the valuation on $(W^{\prime\prime},{\leq^{\prime\prime}},{R^{\prime\prime}})$ such that for all atoms $p$, $V^{\prime\prime}(p){=}V^{\prime}(p)$.
\begin{claim}
For all formulas $B$,
\begin{itemize}
\item $(W^{\prime\prime},{\leq^{\prime\prime}},{R^{\prime\prime}},V^{\prime\prime}),a{\models}B$ if and only if $(W^{\prime},{\leq^{\prime}},{R^{\prime}},V^{\prime}),a{\models}B$,
\item $(W^{\prime\prime},{\leq^{\prime\prime}},{R^{\prime\prime}},V^{\prime\prime}),b{\models}B$ if and only if $(W^{\prime},{\leq^{\prime}},{R^{\prime}},V^{\prime}),b{\models}B$.
\end{itemize}
\end{claim}
\begin{proofclaim}
By induction on $B$.
\end{proofclaim}
Since $(W^{\prime},{\leq^{\prime}},{R^{\prime}},V^{\prime}),s{\not\models}A$, then $(W^{\prime\prime},{\leq^{\prime\prime}},{R^{\prime\prime}},V^{\prime}),s{\not\models}A$.
Hence, $(W^{\prime\prime},{\leq^{\prime\prime}},
$\linebreak$
{R^{\prime\prime}},V^{\prime\prime}){\not\models}A$.
Thus, $(W^{\prime\prime},{\leq^{\prime\prime}},{R^{\prime\prime}}){\not\models}A$: a contradiction.
$\mathbf{(2)}$~For the sake of the contradiction, suppose ${\mathcal C}_{\symmetric}$ is modally definable.
Hence, there exists a formula $A$ such that $A$ modally defines ${\mathcal C}_{\symmetric}$.
Let $(W^{\prime},{\leq^{\prime}},{R^{\prime}})$ be the frame defined by $W^{\prime}{=}\{a,b,c,d\}$, $a{\leq^{\prime}}b$ and $b{\leq^{\prime}}c$ and $a{R^{\prime}}d$, $c{R^{\prime}}d$, $d{R^{\prime}}a$, $d{R^{\prime}}b$ and $d{R^{\prime}}c$ and $(W^{\prime\prime},{\leq^{\prime\prime}},{R^{\prime\prime}})$ be the frame defined by $W^{\prime\prime}{=}\{a,b,c,d\}$ and $a{\leq^{\prime\prime}}b$ and $b{\leq^{\prime\prime}}c$ and $a{R^{\prime\prime}}d$, $b{R^{\prime\prime}}d$, $c{R^{\prime\prime}}d$, $d{R^{\prime\prime}}a$, $d{R^{\prime\prime}}b$ and $d{R^{\prime\prime}}c$.
Obviously, $(W^{\prime},{\leq^{\prime}},{R^{\prime}})$ is not symmetric and $(W^{\prime\prime},{\leq^{\prime\prime}},{R^{\prime\prime}})$ is symmetric.
Since $A$ modally defines ${\mathcal C}_{\symmetric}$, then $(W^{\prime},{\leq^{\prime}},{R^{\prime}}){\not\models}A$ and $(W^{\prime\prime},{\leq^{\prime\prime}},{R^{\prime\prime}}){\models}A$.
Thus, there exists a valuation $V^{\prime}\ :\ \At{\longrightarrow}
$\linebreak$
\wp(W^{\prime})$ on $(W^{\prime},{\leq^{\prime}},{R^{\prime}})$ such that $(W^{\prime},{\leq^{\prime}},{R^{\prime}},V^{\prime}){\not\models}A$.
Consequently, there exists $s{\in}W^{\prime}$ such that $(W^{\prime},{\leq^{\prime}},{R^{\prime}},V^{\prime}),s{\not\models}A$.
Let $V^{\prime\prime}\ :\ \At{\longrightarrow}\wp(W^{\prime\prime})$ be the valuation on $(W^{\prime\prime},{\leq^{\prime\prime}},{R^{\prime\prime}})$ such that for all atoms $p$, $V^{\prime\prime}(p){=}V^{\prime}(p)$.
\begin{claim}
For all formulas $B$,
\begin{itemize}
\item $(W^{\prime\prime},{\leq^{\prime\prime}},{R^{\prime\prime}},V^{\prime\prime}),a{\models}B$ if and only if $(W^{\prime},{\leq^{\prime}},{R^{\prime}},V^{\prime}),a{\models}B$,
\item $(W^{\prime\prime},{\leq^{\prime\prime}},{R^{\prime\prime}},V^{\prime\prime}),b{\models}B$ if and only if $(W^{\prime},{\leq^{\prime}},{R^{\prime}},V^{\prime}),b{\models}B$,
\item $(W^{\prime\prime},{\leq^{\prime\prime}},{R^{\prime\prime}},V^{\prime\prime}),c{\models}B$ if and only if $(W^{\prime},{\leq^{\prime}},{R^{\prime}},V^{\prime}),c{\models}B$,
\item $(W^{\prime\prime},{\leq^{\prime\prime}},{R^{\prime\prime}},V^{\prime\prime}),d{\models}B$ if and only if $(W^{\prime},{\leq^{\prime}},{R^{\prime}},V^{\prime}),d{\models}B$.
\end{itemize}
\end{claim}
\begin{proofclaim}
By induction on $B$.
\end{proofclaim}
Since $(W^{\prime},{\leq^{\prime}},{R^{\prime}},V^{\prime}),s{\not\models}A$, then $(W^{\prime\prime},{\leq^{\prime\prime}},{R^{\prime\prime}},V^{\prime}),s{\not\models}A$.
Hence, $(W^{\prime\prime},{\leq^{\prime\prime}},
$\linebreak$
{R^{\prime\prime}},V^{\prime\prime}){\not\models}A$.
Thus, $(W^{\prime\prime},{\leq^{\prime\prime}},{R^{\prime\prime}}){\not\models}A$: a contradiction.
$\mathbf{(3)}$~Similar to the proof of Item~$\mathbf{(1)}$.
$\mathbf{(4)}$~Similar to the proof of Item~$\mathbf{(1)}$.
$\mathbf{(5)}$~Similar to the proof of Item~$\mathbf{(1)}$.
$\mathbf{(6)}$~Similar to the proof of Item~$\mathbf{(1)}$.
$\mathbf{(7)}$~Similar to the proof of Item~$\mathbf{(1)}$.
\medskip
\end{proof}
Therefore, one may ask whether the following formulas are modal definitions of elementary conditions on frames: ${\square}p{\rightarrow}p$, $p{\rightarrow}{\lozenge}p$, ${\lozenge}{\square}p{\rightarrow}p$, $p{\rightarrow}{\square}{\lozenge}p$, ${\square}p{\rightarrow}{\square}{\square}p$ and ${\lozenge}{\lozenge}p{\rightarrow}{\lozenge}p$.\footnote{Here, the reader should remind that as far as classical modal validity is concerned, the formulas ${\square}p{\rightarrow}p$ and $p{\rightarrow}{\lozenge}p$ correspond to reflexivity, the formulas ${\lozenge}{\square}p{\rightarrow}p$ and $p{\rightarrow}{\square}{\lozenge}p$ correspond to symmetry and the formulas ${\square}p{\rightarrow}{\square}{\square}p$ and ${\lozenge}{\lozenge}p{\rightarrow}{\lozenge}p$ correspond to transitivity.
See~\cite[Chapter~$4$]{Blackburn:et:al:2001} and~\cite[Chapter~$3$]{Chagrov:Zakharyaschev:1997}.}
\begin{lemma}\label{lemma:correspondence:formulas:T:B:4:elementary:conditions}
\begin{enumerate}
\item ${\square}p{\rightarrow}p$ modally defines ${\mathcal C}_{\ureflexive}$,
\item $p{\rightarrow}{\lozenge}p$ modally defines ${\mathcal C}_{\dreflexive}$,
\item ${\lozenge}{\square}p{\rightarrow}p$ modally defines ${\mathcal C}_{\usymmetric}$,
\item $p{\rightarrow}{\square}{\lozenge}p$ modally defines ${\mathcal C}_{\dsymmetric}$,
\item ${\square}p{\rightarrow}{\square}{\square}p$ modally defines ${\mathcal C}_{\utransitive}$,
\item ${\lozenge}{\lozenge}p{\rightarrow}{\lozenge}p$ modally defines ${\mathcal C}_{\dtransitive}$.
\end{enumerate}
\end{lemma}
\begin{proof}
$\mathbf{(1)}$~For the sake of the contradiction, suppose ${\square}p{\rightarrow}p$ does not modally define ${\mathcal C}_{\ureflexive}$.
Hence, there exists a frame $(W,{\leq},{R})$ such that either for all $s{\in}W$, $s{\leq}{\circ}{R}{\circ}{\leq}s$ and $(W,{\leq},{R}){\not\models}{\square}p{\rightarrow}p$, or there exists $s{\in}W$ such that not $s{\leq}{\circ}{R}{\circ}{\leq}s$ and $(W,{\leq},{R}){\models}
$\linebreak$
{\square}p{\rightarrow}p$.
The reader may easily verify that the former case leads to a contradiction.
In the latter case, let $V\ :\ \At{\longrightarrow}\wp(W)$ be a valuation on $(W,{\leq},{R})$ such that $V(p){=}\{t{\in}W\ :\ s{\leq}{\circ}{R}{\circ}{\leq}t\}$.
Since not $s{\leq}{\circ}{R}{\circ}{\leq}s$, then $s{\not\in}V(p)$.
Thus, $s{\not\models}p$.
Since for all $u{\in}W$, if $s{\leq}{\circ}{R}u$ then $u{\in}V(p)$, then for all $u{\in}W$, if $s{\leq}{\circ}{R}u$ then $u{\models}p$.
Consequently, $s{\models}{\square}p$.
Since $s{\not\models}p$, then $s{\not\models}{\square}p{\rightarrow}p$.
Hence, $(W,{\leq},{R},V){\not\models}{\square}p{\rightarrow}p$.
Thus, $(W,{\leq},{R}){\not\models}{\square}p{\rightarrow}p$: a contradiction.
$\mathbf{(2)}$~For the sake of the contradiction, suppose $p{\rightarrow}{\lozenge}p$ does not modally define ${\mathcal C}_{\dreflexive}$.
Consequently, there exists a frame $(W,{\leq},{R})$ such that either for all $s{\in}W$, $s{\geq}{\circ}{R}{\circ}{\geq}s$ and $(W,{\leq},{R}){\not\models}p{\rightarrow}{\lozenge}p$, or there exists $s{\in}W$ such that not $s{\geq}{\circ}{R}{\circ}{\geq}s$ and $(W,{\leq},{R}){\models}
$\linebreak$
p{\rightarrow}{\lozenge}p$.
The reader may easily verify that the former case leads to a contradiction.
In the latter case, let $V\ :\ \At{\longrightarrow}\wp(W)$ be a valuation on $(W,{\leq},{R})$ such that $V(p){=}\{t{\in}W\ :\ s{\leq}t\}$.
Since not $s{\geq}{\circ}{R}{\circ}{\geq}s$, then for all $u{\in}W$, if $s{\geq}{\circ}{R}u$ then $u{\not\in}V(p)$.
Hence, for all $u{\in}W$, if $s{\geq}{\circ}{R}u$ then $u{\not\models}p$.
Thus, $s{\not\models}{\lozenge}p$.
Since $s{\in}V(p)$, then $s{\models}p$.
Since $s{\not\models}{\lozenge}p$, then $s{\not\models}p{\rightarrow}{\lozenge}p$.
Consequently, $(W,{\leq},{R},V){\not\models}p{\rightarrow}{\lozenge}p$.
Hence, $(W,{\leq},{R}){\not\models}p{\rightarrow}{\lozenge}p$: a contradiction.
$\mathbf{(3)}$~For the sake of the contradiction, suppose ${\lozenge}{\square}p{\rightarrow}p$ does not modally define ${\mathcal C}_{\usymmetric}$.
Thus, there exists a frame $(W,{\leq},{R})$ such that either for all $s,t{\in}W$, if $s{R}t$ then $t{\leq}{\circ}{R}{\circ}{\leq}s$ and $(W,{\leq},{R}){\not\models}{\lozenge}{\square}p{\rightarrow}p$, or there exists $s,t{\in}W$ such that $s{R}t$ and not $t{\leq}{\circ}{R}{\circ}{\leq}s$ and $(W,{\leq},{R}){\models}{\lozenge}{\square}p{\rightarrow}p$.
The reader may easily verify that the former case leads to a contradiction.
In the latter case, let $V\ :\ \At{\longrightarrow}\wp(W)$ be a valuation on $(W,{\leq},{R})$ such that $V(p){=}\{u{\in}W\ :\ t{\leq}{\circ}{R}{\circ}{\leq}u\}$.
Since not $t{\leq}{\circ}{R}{\circ}{\leq}s$, then $s{\not\in}V(p)$.
Consequently, $s{\not\models}p$.
Since for all $v{\in}W$, if $t{\leq}{\circ}{R}v$ then $v{\in}V(p)$, then for all $v{\in}W$, if $t{\leq}{\circ}{R}v$ then $v{\models}p$.
Hence, $t{\models}{\square}p$.
Since $s{R}t$, then $s{\models}{\lozenge}{\square}p$.
Since $s{\not\models}p$, then $s{\not\models}{\lozenge}{\square}p{\rightarrow}p$.
Thus, $(W,{\leq},{R},V){\not\models}{\lozenge}{\square}p{\rightarrow}p$.
Consequently, $(W,{\leq},{R}){\not\models}{\lozenge}{\square}p{\rightarrow}p$: a contradiction.
$\mathbf{(4)}$~For the sake of the contradiction, suppose $p{\rightarrow}{\square}{\lozenge}p$ does not modally define ${\mathcal C}_{\dsymmetric}$.
Hence, there exists a frame $(W,{\leq},{R})$ such that either for all $s,t{\in}W$, if $s{R}t$ then $t{\geq}{\circ}{R}{\circ}{\geq}s$ and $(W,{\leq},{R}){\not\models}p{\rightarrow}{\square}{\lozenge}p$, or there exists $s,t{\in}W$ such that $s{R}t$ and not $t{\geq}{\circ}{R}{\circ}{\geq}s$ and $(W,{\leq},{R}){\models}p{\rightarrow}{\square}{\lozenge}p$.
The reader may easily verify that the former case leads to a contradiction.
In the latter case, let $V\ :\ \At{\longrightarrow}\wp(W)$ be a valuation on $(W,{\leq},{R})$ such that $V(p){=}\{u{\in}W\ :\ s{\leq}u\}$.
Thus, $s{\in}V(p)$.
Consequently, $s{\models}p$.
Since not $t{\geq}{\circ}{R}{\circ}{\geq}s$, then for all $v{\in}W$, if $t{\geq}{\circ}{R}v$ then $v{\not\in}V(p)$.
Hence, for all $v{\in}W$, if $t{\geq}{\circ}{R}v$ then $v{\not\models}p$.
Thus, $t{\not\models}{\lozenge}p$.
Since $s{R}t$, then $s{\not\models}{\square}{\lozenge}p$.
Since $s{\models}p$, then $s{\not\models}p{\rightarrow}{\square}{\lozenge}p$.
Consequently, $(W,{\leq},{R},V){\not\models}p{\rightarrow}{\square}{\lozenge}p$.
Hence, $(W,{\leq},{R}){\not\models}p{\rightarrow}{\square}{\lozenge}p$: a contradiction.
$\mathbf{(5)}$~For the sake of the contradiction, suppose ${\square}p{\rightarrow}{\square}{\square}p$ does not modally define ${\mathcal C}_{\utransitive}$.
Thus, there exists a frame $(W,{\leq},{R})$ such that either for all $s,t,u,v{\in}W$, if $s{R}t$, $t{\leq}u$ and $u{R}v$ then $s{\leq}{\circ}{R}{\circ}{\leq}v$ and $(W,{\leq},{R}){\not\models}{\square}p{\rightarrow}{\square}{\square}p$, or there exists $s,t,u,v{\in}W$ such that $s{R}t$, $t{\leq}u$, $u{R}v$ and not $s{\leq}{\circ}{R}{\circ}{\leq}v$ and $(W,{\leq},{R}){\models}{\square}p{\rightarrow}{\square}{\square}p$.
The reader may easily verify that the former case leads to a contradiction.
In the latter case, let $V\ :\ \At{\longrightarrow}\wp(W)$ be a valuation on $(W,{\leq},{R})$ such that $V(p){=}\{w{\in}W\ :\ s{\leq}{\circ}{R}{\circ}{\leq}w\}$.
Since not $s{\leq}{\circ}{R}{\circ}{\leq}v$, then $v{\not\in}V(p)$.
Consequently, $v{\not\models}p$.
Since $u{R}v$, then $u{\not\models}{\square}p$.
Since $t{\leq}u$, then $t{\not\models}{\square}p$.
Since $s{R}t$, then $s{\not\models}{\square}{\square}p$.
Since for all $x{\in}W$, if $s{\leq}{\circ}{R}x$ then $x{\in}V(p)$,then for all $x{\in}W$, if $s{\leq}{\circ}{R}x$ then $x{\models}p$.
Hence, $s{\models}{\square}p$.
Since $s{\not\models}{\square}{\square}p$, then $s{\not\models}{\square}p{\rightarrow}{\square}{\square}p$.
Thus, $(W,{\leq},{R},V){\not\models}{\square}p{\rightarrow}{\square}{\square}p$.
Consequently, $(W,{\leq},{R}){\not\models}{\square}p{\rightarrow}{\square}{\square}p$: a contradiction.
$\mathbf{(6)}$~For the sake of the contradiction, suppose ${\lozenge}{\lozenge}p{\rightarrow}{\lozenge}p$ does not modally define ${\mathcal C}_{\dtransitive}$.
Hence, there exists a frame $(W,{\leq},{R})$ such that either for all $s,t,u,v
$\linebreak$
{\in}W$, if $s{R}t$, $t{\geq}u$ and $u{R}v$ then $s{\geq}{\circ}{R}{\circ}{\geq}v$ and $(W,{\leq},{R}){\not\models}{\lozenge}{\lozenge}p{\rightarrow}{\lozenge}p$, or there exists $s,t,u,v{\in}W$ such that $s{R}t$, $t{\geq}u$, $u{R}v$ and not $s{\geq}{\circ}{R}{\circ}{\geq}v$ and $(W,{\leq},{R}){\models}{\lozenge}{\lozenge}p{\rightarrow}{\lozenge}p$.
The reader may easily verify that the former case leads to a contradiction.
In the latter case, let $V\ :\ \At{\longrightarrow}\wp(W)$ be a valuation on $(W,{\leq},{R})$ such that $V(p){=}\{w{\in}W\ :\ v{\leq}w\}$.
Thus, $v{\in}V(p)$.
Consequently, $v{\models}p$.
Since $u{R}v$, then $u{\models}{\lozenge}p$.
Since $t{\geq}u$, then $t{\models}{\lozenge}p$.
Since $s{R}t$, then $s{\models}{\lozenge}{\lozenge}p$.
Since not $s{\geq}{\circ}{R}{\circ}{\geq}v$, then for all $x{\in}W$, if $s{\geq}{\circ}{R}x$ then $x{\not\in}V(p)$.
Hence, for all $x{\in}W$, if $s{\geq}{\circ}{R}x$ then $x{\not\models}p$.
Thus, $s{\not\models}{\lozenge}p$.
Since $s{\models}{\lozenge}{\lozenge}p$, then $s{\not\models}{\lozenge}{\lozenge}p{\rightarrow}{\lozenge}p$.
Consequently, $(W,{\leq},{R},V){\not\models}{\lozenge}{\lozenge}p{\rightarrow}{\lozenge}p$.
Hence, $(W,{\leq},{R}){\not\models}{\lozenge}{\lozenge}p{\rightarrow}{\lozenge}p$: a contradiction.
\medskip
\end{proof}
\section{Axiomatization}\label{section:axiomatization}
In this section, we axiomatically present different intuitionistic modal logics.
\begin{definition}[Intuitionistic modal logics]
In our language, an {\em intuitionistic modal logic}\/ is a set of formulas closed under uniform substitution, containing the standard axioms of $\IPL$, closed under the standard inference rules of $\IPL$, containing the axioms
\begin{description}
\item[$(\Axiom1)$] ${\square}(p{\rightarrow}q){\rightarrow}({\square}p{\rightarrow}{\square}q)$,
\item[$(\Axiom2)$] ${\square}(p{\vee}q){\rightarrow}(({\lozenge}p{\rightarrow}{\square}q){\rightarrow}{\square}q)$,
\item[$(\Axiom3)$] ${\lozenge}(p{\vee}q){\rightarrow}{\lozenge}p{\vee}{\lozenge}q$,
\item[$(\Axiom4)$] $\neg{\lozenge}{\bot}$,
\end{description}
and closed under the inference rules
\begin{description}
\item[$(\Rule1)$] $\frac{p}{{\square}p}$,
\item[$(\Rule2)$] $\frac{p{\rightarrow}q}{{\lozenge}p{\rightarrow}{\lozenge}q}$,
\item[$(\Rule3)$] $\frac{{\lozenge}p{\rightarrow}q{\vee}{\square}(p{\rightarrow}r)}{{\lozenge}p{\rightarrow}q{\vee}{\lozenge}r}$.
\end{description}
\end{definition}
The reader may easily verify that axiom $(\Axiom1)$ and inference rule $(\Rule1)$ can be replaced by the axioms
\begin{description}
\item[$(\Axiom5)$] ${\square}p{\wedge}{\square}q{\rightarrow}{\square}(p{\wedge}q)$,
\item[$(\Axiom6)$] ${\square}{\top}$,
\end{description}
and the inference rule
\begin{description}
\item[$(\Rule4)$] $\frac{p{\rightarrow}q}{{\square}p{\rightarrow}{\square}q}$,
\end{description}
without affecting the above definition of intuitionistic modal logics.\footnote{As mentioned in Section~\ref{section:introduction}, the axioms ${\square}p{\wedge}{\square}q{\rightarrow}{\square}(p{\wedge}q)$, ${\lozenge}(p{\vee}q){\rightarrow}{\lozenge}p{\vee}{\lozenge}q$, ${\square}{\top}$ and $\neg{\lozenge}{\bot}$ and the inference rules $\frac{p{\rightarrow}q}{{\square}p{\rightarrow}{\square}q}$ and $\frac{p{\rightarrow}q}{{\lozenge}p{\rightarrow}{\lozenge}q}$ are usually associated to the concept of normality in modal logics.}
In other respect, the reader may easily see that axiom $(\Axiom2)$ has similarities with P\v{r}enosil's equation ${\lozenge}a{\rightarrow}{\square}b{\leq}{\square}(a{\vee}b){\rightarrow}{\square}b$ and inference rule $(\Rule3)$ has similarities with P\v{r}enosil's {\em positive modal law}\/ ${\lozenge}b{\leq}{\square}a{\vee}c{\Rightarrow}{\lozenge}b{\leq}{\lozenge}(a{\wedge}b){\vee}c$.
It is worth mentioning here that our axiomatization does not require an inference rule having similarities with P\v{r}enosil's {\em negative modal law}\/ ${\lozenge}a{\wedge}c{\leq}{\square}b{\Rightarrow}{\square}(a{\vee}b){\wedge}c{\leq}{\square}b$.
See~\cite{Prenosil:2014} for details about the above-mentioned equation and the above-mentioned modal laws.
\begin{definition}[Additional axioms]
We also consider the axioms
\begin{description}
\item[$(\Axiom\mathbf{f})$] ${\lozenge}(p{\rightarrow}q){\rightarrow}({\square}p{\rightarrow}{\lozenge}q)$,
\item[$(\Axiom\mathbf{b})$] $({\lozenge}p{\rightarrow}{\square}q){\rightarrow}{\square}(p{\rightarrow}q)$,
\item[$(\Axiom\mathbf{d})$] ${\square}(p{\vee}q){\rightarrow}{\lozenge}p{\vee}{\square}q$,
\item[$(\Axiom\mathbf{uref})$] ${\square}p{\rightarrow}p$,
\item[$(\Axiom\mathbf{dref})$] $p{\rightarrow}{\lozenge}p$,
\item[$(\Axiom\mathbf{usym})$] ${\lozenge}{\square}p{\rightarrow}p$,
\item[$(\Axiom\mathbf{dsym})$] $p{\rightarrow}{\square}{\lozenge}p$,
\item[$(\Axiom\mathbf{utra})$] ${\square}p{\rightarrow}{\square}{\square}p$,
\item[$(\Axiom\mathbf{dtra})$] ${\lozenge}{\lozenge}p{\rightarrow}{\lozenge}p$.
\end{description}
\end{definition}
All the above axioms but axioms $(\Axiom2)$ and $(\Axiom\mathbf{d})$ and all the above inference rules but inference rule $(\Rule3)$ have been already considered in the above-mentioned literature about intuitionistic modal logics~\cite{FischerServi:1984,Wijesekera:1990}.
Notice that axiom $(\Axiom\mathbf{d})$ has been already considered in~\cite{Balbiani:et:al:2021}.
Obviously, axioms $(\Axiom1)$, $(\Axiom3)$ and $(\Axiom4)$ are valid in any frame.
Concerning axiom $(\Axiom2)$,
\begin{lemma}\label{lemma:about:axiom:2}
Let $A$, $B$ be formulas.
For all frames $(W,{\leq},{R})$, $(W,{\leq},{R}){\models}{\square}(A{\vee}B){\rightarrow}
$\linebreak$
(({\lozenge}A{\rightarrow}{\square}B){\rightarrow}{\square}B)$.
\end{lemma}
\begin{proof}
If not, there exists a frame $(W,{\leq},{R})$ such that $(W,{\leq},{R}){\not\models}{\square}(A{\vee}B){\rightarrow}(({\lozenge}A{\rightarrow}
$\linebreak$
{\square}B){\rightarrow}{\square}B)$.
Hence, there exists a model $(W,{\leq},{R},V)$ based on $(W,{\leq},{R})$ such that $(W,{\leq},{R},V){\not\models}{\square}(A{\vee}B){\rightarrow}(({\lozenge}A{\rightarrow}{\square}B){\rightarrow}{\square}B)$.
Thus, there exists $s{\in}W$ such that $s{\not\models}{\square}(A{\vee}B){\rightarrow}(({\lozenge}A{\rightarrow}{\square}B){\rightarrow}{\square}B)$.
Consequently, there exists $t{\in}W$ such that $s{\leq}t$, $t{\models}{\square}(A{\vee}B)$ and $t\not\models({\lozenge}A{\rightarrow}{\square}B){\rightarrow}{\square}B$.
Hence, there exists $u{\in}W$ such that $t{\leq}u$, $u{\models}{\lozenge}A{\rightarrow}{\square}B$ and $u{\not\models}{\square}B$.
Thus, there exists $v{\in}W$ such that $u{\leq}{\circ}Rv$ and $v{\not\models}B$.
Since $t{\leq}u$, then $t{\leq}{\circ}{R}v$.
Since $t{\models}{\square}(A{\vee}B)$, then $v{\models}A{\vee}B$.
Consequently, either $v{\models}A$, or $v{\models}B$.
Since $v{\not\models}B$, then $v{\models}A$.
Since $u{\leq}{\circ}Rv$, then there exists $w{\in}W$ such that $u{\leq}w$ and $wRv$.
Since $v{\models}A$, then $w{\models}{\lozenge}A$.
Since $u{\models}{\lozenge}A{\rightarrow}{\square}B$ and $u{\leq}w$, then $w{\models}{\square}B$.
Since $wRv$, then $v{\models}B$: a contradiction.
\medskip
\end{proof}
Therefore, axiom $(\Axiom2)$ is valid in any frame.
Obviously, inference rules $(\Rule1)$ and $(\Rule2)$ preserve validity in any frame.
Concerning inference rule $(\Rule3)$,
\begin{lemma}\label{lemma:about:rule:3}
Let $A$, $B$, $C$ be formulas.
For all frames $(W,{\leq},{R})$, if $(W,{\leq},{R}){\models}{\lozenge}A{\rightarrow}B
$\linebreak$
{\vee}{\square}(A{\rightarrow}C)$ then $(W,{\leq},{R}){\models}{\lozenge}A{\rightarrow}B{\vee}{\lozenge}C$.
\end{lemma}
\begin{proof}
If not, there exists a frame $(W,{\leq},{R})$ such that $(W,{\leq},{R}){\models}{\lozenge}A{\rightarrow}B{\vee}{\square}(A{\rightarrow}C)$ and $(W,{\leq},{R}){\not\models}{\lozenge}A{\rightarrow}B{\vee}{\lozenge}C$.
Hence, there exists a model $(W,{\leq},{R},V)$ based on $(W,{\leq},{R})$ such that $(W,{\leq},{R},V){\not\models}{\lozenge}A{\rightarrow}B{\vee}{\lozenge}C$.
Thus, there exists $s{\in}W$ such that $s{\not\models}{\lozenge}A{\rightarrow}B{\vee}{\lozenge}C$.
Consequently, there exists $t{\in}W$ such that $s{\leq}t$, $t{\models}{\lozenge}A$ and $t{\not\models}B{\vee}
$\linebreak$
{\lozenge}C$.
Hence, there exists $u{\in}W$ such that $t{\geq}{\circ}Ru$ and $u{\models}A$.
Thus, there exists $v{\in}W$ such that $t{\geq}v$ and $vRu$.
Since $u{\models}A$, then $v{\models}{\lozenge}A$.
Since $(W,{\leq},{R}){\models}{\lozenge}A{\rightarrow}B{\vee}{\square}(A{\rightarrow}
$\linebreak$
C)$, then $(W,{\leq},{R},V){\models}{\lozenge}A{\rightarrow}B{\vee}{\square}(A{\rightarrow}C)$.
Consequently, $v{\models}{\lozenge}A{\rightarrow}B{\vee}{\square}(A{\rightarrow}C)$.
Since $v{\models}{\lozenge}A$, then $v{\models}B{\vee}{\square}(A{\rightarrow}C)$.
Hence, either $v{\models}B$, or $v{\models}{\square}(A{\rightarrow}C)$.
In the former case, since $t{\geq}v$, then $t{\models}B$.
Thus, $t{\models}B{\vee}{\lozenge}C$: a contradiction.
In the latter case, since $vRu$, then $u{\models}A{\rightarrow}C$.
Since $u{\models}A$, then $u{\models}C$.
Since $t{\geq}{\circ}Ru$, then $t{\models}{\lozenge}C$.
Consequently, $t{\models}B{\vee}{\lozenge}C$: a contradiction.
\medskip
\end{proof}
Therefore, inference rule $(\Rule3)$ preserves validity in any frame.
\begin{definition}[Consistency]
An intuitionistic modal logic $\L$ is {\em consistent}\/ if ${\bot}{{\not\in}}\L$.
\end{definition}
Obviously, for all intuitionistic modal logics $\L$, using axioms and inference rules of $\IPL$, $\L$ is consistent if and only if $\L{\not=}\Fo$.
Moreover, for all consistent intuitionistic modal logics $\L$, using axiom $(\Axiom4)$, ${\lozenge}{\bot}{\not\in}\L$.
Obviously, for all indexed families $(\L_{i})_{i{\in}I}$ of intuitionistic modal logics, $\bigcap\{\L_{i}\ :\ i{\in}I\}$ is an intuitionistic modal logic and for all nonempty chains $(\L_{i})_{i{\in}I}$ of intuitionistic modal logics, $\bigcup\{\L_{i}\ :\ i{\in}I\}$ is an intuitionistic modal logic.
As a result, there exists a least intuitionistic modal logic.
In other respect, obviously, $\Fo$ is the greatest intuitionistic modal logic.
\begin{definition}[Least intuitionistic modal logic]
Let $\L_{\min}$ be the least intuitionistic modal logic.
\end{definition}
Obviously, for all intuitionistic modal logics $\L$ and for all sets $\Sigma$ of formulas, there exists a least intuitionistic modal logic containing $\L$ and $\Sigma$.
\begin{definition}[Finite axiomatizability]
For all intuitionistic modal logics $\L$ and for all sets $\Sigma$ of formulas, let $\L{\oplus}\Sigma$ be the least intuitionistic modal logic containing $\L$ and $\Sigma$.
An intuitionistic modal logic $\L$ is {\em finitely axiomatizable}\/ if there exists a finite set $\Sigma$ of formulas such that $\L{=}\L_{\min}{\oplus}\Sigma$.
For all intuitionistic modal logics $\L$ and for all formulas $A$, we write $\L{\oplus}A$ instead of $\L{\oplus}\{A\}$.
Let $\L_{\fcfra}{=}\L_{\min}{\oplus}(\Axiom\mathbf{f})$, $\L_{\bcfra}{=}\L_{\min}{\oplus}(\Axiom\mathbf{b})$ and $\L_{\dcfra}{=}\L_{\min}{\oplus}(\Axiom\mathbf{d})$.
We also write $\L_{\fbcfra}$ to denote $\L_{\min}{\oplus}(\Axiom\mathbf{f}){\oplus}(\Axiom\mathbf{b})$, $\L_{\fdcfra}$ to denote $\L_{\min}{\oplus}(\Axiom\mathbf{f}){\oplus}(\Axiom\mathbf{d})$, etc.
Let $\L_{\reflexive}{=}\L_{\min}{\oplus}(\Axiom\mathbf{uref}){\oplus}(\Axiom\mathbf{dref})$ and $\L_{\symmetric}{=}\L_{\min}
$\linebreak$
{\oplus}(\Axiom\mathbf{usym}){\oplus}(\Axiom\mathbf{usym})$.
Let $\L_{\ureflexive}{=}\L_{\min}{\oplus}(\Axiom\mathbf{uref})$, $\L_{\dreflexive}{=}\L_{\min}{\oplus}(\Axiom\mathbf{dref})$,
\linebreak$
\L_{\usymmetric}{=}\L_{\min}{\oplus}(\Axiom\mathbf{usym})$, $\L_{\dsymmetric}{=}\L_{\min}{\oplus}(\Axiom\mathbf{dsym})$, $\L_{\utransitive}{=}\L_{\min}{\oplus}(\Axiom\mathbf{utra})$ and $\L_{\dtransitive}{=}\L_{\min}{\oplus}(\Axiom\mathbf{dtra})$.
\end{definition}
In the proof of Lemma~\ref{lemma:about:a:useful:application:of:rule:R3}, inference rule $(\Rule3)$ is used in a crucial way.\footnote{Lemma~\ref{lemma:about:a:useful:application:of:rule:R3} is used in the proof of Lemma~\ref{first:lemma:about:the:use:of:the:special:inference:rule}.}
\begin{lemma}\label{lemma:about:a:useful:application:of:rule:R3}
Let $A,B,C$ be formulas.
For all intuitionistic modal logics $\L$, if ${\lozenge}A{\rightarrow}{\square}B
$\linebreak$
{\vee}C{\in}\L$ then ${\lozenge}A{\rightarrow}{\lozenge}(B{\wedge}A){\vee}C{\in}\L$.
\end{lemma}
\begin{proof}
Let $\L$ be an intuitionistic modal logic.
By considering the following sequence of formulas, the reader may easily construct a proof that if ${\lozenge}A{\rightarrow}{\square}B{\vee}C{\in}\L$ then ${\lozenge}A{\rightarrow}{\lozenge}(B{\wedge}A){\vee}C{\in}\L$:
\begin{enumerate}
\item ${\lozenge}A{\rightarrow}{\square}B{\vee}C$ (hypothesis),
\item $B{\rightarrow}(A{\rightarrow}B{\wedge}A)$ ($\IPL$-reasoning),
\item ${\square}B{\rightarrow}{\square}(A{\rightarrow}B{\wedge}A)$ (rule $(\Rule1)$ on~$\mathbf{(2)}$),
\item ${\lozenge}A{\rightarrow}C{\vee}{\square}(A{\rightarrow}B{\wedge}A)$ ($\IPL$-reasoning on~$\mathbf{(1)}$ and~$\mathbf{(3)}$),
\item ${\lozenge}A{\rightarrow}C{\vee}{\lozenge}(B{\wedge}A)$ (rule $(\Rule3)$ on~$\mathbf{(4)}$).
\end{enumerate}
\medskip
\end{proof}
Lemma~\ref{first:lemma:about:the:use:of:the:special:inference:rule} is used in the proofs of Lemmas~\ref{lemma:about:what:happens:if:T:lozenge:is:in:the:logic}, \ref{lemma:about:what:happens:if:B:lozenge:is:in:the:logic} and~\ref{lemma:about:what:happens:if:4:lozenge:is:in:the:logic}.
\begin{lemma}\label{first:lemma:about:the:use:of:the:special:inference:rule}
Let $A,B,C,D,E,F,G,H,I,J$ be formulas.
For all intuitionistic modal logics $\L$, if $A{\wedge}{\lozenge}G{\rightarrow}{\square}E{\vee}F{\in}\L$, ${\lozenge}H{\rightarrow}{\square}B{\vee}D{\in}\L$ and ${\lozenge}I{\wedge}{\lozenge}C{\rightarrow}A{\in}\L$, then ${\lozenge}(G{\wedge}
$\linebreak$
(B{\rightarrow}C){\wedge}H{\wedge}I{\wedge}(E{\rightarrow}J)){\rightarrow}F{\vee}D{\vee}{\lozenge}J{\in}\L$.
\end{lemma}
\begin{proof}
Let $\L$ be an intuitionistic modal logic.
By considering the following sequence of formulas, the reader may easily construct a proof that if $A{\wedge}{\lozenge}G{\rightarrow}{\square}E{\vee}F{\in}\L$, ${\lozenge}H{\rightarrow}
$\linebreak$
{\square}B{\vee}D{\in}\L$ and ${\lozenge}I{\wedge}{\lozenge}C{\rightarrow}A{\in}\L$ then ${\lozenge}(G{\wedge}(B{\rightarrow}C){\wedge}H{\wedge}I{\wedge}(E{\rightarrow}J)){\rightarrow}F{\vee}D{\vee}{\lozenge}J{\in}
$\linebreak$
\L$:
\begin{enumerate}
\item $A{\wedge}{\lozenge}G{\rightarrow}{\square}E{\vee}F$ (hypothesis),
\item ${\lozenge}H{\rightarrow}{\square}B{\vee}D$ (hypothesis),
\item ${\lozenge}I{\wedge}{\lozenge}C{\rightarrow}A$ (hypothesis),
\item $G{\wedge}(B{\rightarrow}C){\wedge}H{\wedge}I{\wedge}(E{\rightarrow}J){\rightarrow}H$ ($\IPL$-reasoning),
\item ${\lozenge}(G{\wedge}(B{\rightarrow}C){\wedge}H{\wedge}I{\wedge}(E{\rightarrow}J)){\rightarrow}{\lozenge}H$ (rule $(\Rule2)$ on~$\mathbf{(4)}$),
\item ${\lozenge}(G{\wedge}(B{\rightarrow}C){\wedge}H{\wedge}I{\wedge}(E{\rightarrow}J)){\rightarrow}{\square}B{\vee}D$ ($\IPL$-reasoning on~$\mathbf{(2)}$ and~$\mathbf{(5)}$),
\item ${\lozenge}(G{\wedge}(B{\rightarrow}C){\wedge}H{\wedge}I{\wedge}(E{\rightarrow}J)){\rightarrow}{\lozenge}(B{\wedge}G{\wedge}(B{\rightarrow}C){\wedge}H{\wedge}I{\wedge}(E{\rightarrow}J)){\vee}D
$\linebreak
(Lemma~\ref{lemma:about:a:useful:application:of:rule:R3} on~$\mathbf{(6)}$),
\item $B{\wedge}G{\wedge}(B{\rightarrow}C){\wedge}H{\wedge}I{\wedge}(E{\rightarrow}J){\rightarrow}G$ ($\IPL$-reasoning),
\item $B{\wedge}G{\wedge}(B{\rightarrow}C){\wedge}H{\wedge}I{\wedge}(E{\rightarrow}J){\rightarrow}I$ ($\IPL$-reasoning),
\item $B{\wedge}G{\wedge}(B{\rightarrow}C){\wedge}H{\wedge}I{\wedge}(E{\rightarrow}J){\rightarrow}C$ ($\IPL$-reasoning),
\item ${\lozenge}(B{\wedge}G{\wedge}(B{\rightarrow}C){\wedge}H{\wedge}I{\wedge}(E{\rightarrow}J)){\rightarrow}{\lozenge}G$ (rule $(\Rule2)$ on~$\mathbf{(8)}$),
\item ${\lozenge}(B{\wedge}G{\wedge}(B{\rightarrow}C){\wedge}H{\wedge}I{\wedge}(E{\rightarrow}J)){\rightarrow}{\lozenge}I$ (rule $(\Rule2)$ on~$\mathbf{(9)}$),
\item ${\lozenge}(B{\wedge}G{\wedge}(B{\rightarrow}C){\wedge}H{\wedge}I{\wedge}(E{\rightarrow}J)){\rightarrow}{\lozenge}C$ (rule $(\Rule2)$ on~$\mathbf{(10)}$),
\item ${\lozenge}(B{\wedge}G{\wedge}(B{\rightarrow}C){\wedge}H{\wedge}I{\wedge}(E{\rightarrow}J)){\rightarrow}A$ ($\IPL$-reasoning on~$\mathbf{(3)}$, $\mathbf{(12)}
$\linebreak
and~$\mathbf{(13)}$),
\item ${\lozenge}(B{\wedge}G{\wedge}(B{\rightarrow}C){\wedge}H{\wedge}I{\wedge}(E{\rightarrow}J)){\rightarrow}{\square}E{\vee}F$ ($\IPL$-reasoning on~$\mathbf{(1)}$, $\mathbf{(11)}
$\linebreak
and~$\mathbf{(14)}$),
\item ${\lozenge}(B{\wedge}G{\wedge}(B{\rightarrow}C){\wedge}H{\wedge}I{\wedge}(E{\rightarrow}J)){\rightarrow}{\lozenge}(E{\wedge}B{\wedge}G{\wedge}(B{\rightarrow}C){\wedge}H{\wedge}I{\wedge}(E{\rightarrow}J)){\vee}
$\linebreak$
F$ (Lemma~\ref{lemma:about:a:useful:application:of:rule:R3} on~$\mathbf{(15)}$),
\item $E{\wedge}B{\wedge}G{\wedge}(B{\rightarrow}C){\wedge}H{\wedge}I{\wedge}(E{\rightarrow}J){\rightarrow}J$ ($\IPL$-reasoning),
\item ${\lozenge}(E{\wedge}B{\wedge}G{\wedge}(B{\rightarrow}C){\wedge}H{\wedge}I{\wedge}(E{\rightarrow}J)){\rightarrow}{\lozenge}J$ (rule $(\Rule2)$ on~$\mathbf{(17)}$),
\item ${\lozenge}(G{\wedge}(B{\rightarrow}C){\wedge}H{\wedge}I{\wedge}(E{\rightarrow}J)){\rightarrow}F{\vee}D{\vee}{\lozenge}J$ ($\IPL$-reasoning on~$\mathbf{(7)}$, $\mathbf{(16)}
$\linebreak
and~$\mathbf{(18)}$).
\end{enumerate}
\medskip
\end{proof}
Lemma~\ref{lemma:inference:rule:R3:can:be:eliminated:if:AxiomFORWARDCONFLUENCE:is:used} says that inference rule $(\Rule3)$ can be eliminated when axiom $(\Axiom\mathbf{f})$ is used.
\begin{lemma}\label{lemma:inference:rule:R3:can:be:eliminated:if:AxiomFORWARDCONFLUENCE:is:used}
For all intuitionistic modal logics $\L$, if $(\Axiom\mathbf{f}){\in}\L$ then $\L$ is equal to the least set of formulas closed under uniform substitution, containing the standard axioms of $\IPL$, closed under the standard inference rules of $\IPL$, containing the axioms $(\Axiom1)$, $(\Axiom2)$, $(\Axiom3)$ and $(\Axiom4)$, closed under the inference rules $(\Rule1)$ and $(\Rule2)$ and containing $\L$.
\end{lemma}
\begin{proof}
Let $\L$ be an intuitionistic modal logic.
Suppose $(\Axiom\mathbf{f}){\in}\L$.
Let $\L^{\prime}$ be the least set of formulas closed under uniform substitution, containing the standard axioms of $\IPL$, closed under the standard inference rules of $\IPL$, containing the axioms $(\Axiom1)$, $(\Axiom2)$, $(\Axiom3)$ and $(\Axiom4)$, closed under the inference rules $(\Rule1)$ and $(\Rule2)$ and containing $\L$.
Obviously, it suffices to demonstrate that for all formulas $A,B,C$, if ${\lozenge}A{\rightarrow}B{\vee}{\square}(A{\rightarrow}C){\in}\L^{\prime}$ then ${\lozenge}A{\rightarrow}B{\vee}{\lozenge}C{\in}\L^{\prime}$.\footnote{That is to say, it suffices to demonstrate that inference rule $(\Rule3)$ is $\L^{\prime}$-admissible.}
Let $A,B,C$ be formulas.
By considering the following sequence of formulas, the reader may easily construct a proof that if ${\lozenge}A{\rightarrow}B{\vee}{\square}(A{\rightarrow}C){\in}\L^{\prime}$ then ${\lozenge}A{\rightarrow}B{\vee}{\lozenge}C{\in}\L^{\prime}$:
\begin{enumerate}
\item ${\lozenge}A{\rightarrow}B{\vee}{\square}(A{\rightarrow}C)$ (hypothesis),
\item $A{\rightarrow}((A{\rightarrow}C){\rightarrow}C)$ ($\IPL$-reasoning),
\item ${\lozenge}A{\rightarrow}{\lozenge}((A{\rightarrow}C){\rightarrow}C)$ (rule $(\Rule2)$ on~$\mathbf{(2)}$),
\item ${\lozenge}((A{\rightarrow}C){\rightarrow}C){\rightarrow}({\square}(A{\rightarrow}C){\rightarrow}{\lozenge}C)$ (axiom $(\Axiom\mathbf{f})$),
\item ${\lozenge}A{\rightarrow}B{\vee}{\lozenge}C$ ($\IPL$-reasoning on~$\mathbf{(1)}$, $\mathbf{(3)}$ and~$\mathbf{(4)}$).
\end{enumerate}
\medskip
\end{proof}
%
%
%In Section~\ref{section:soundness:and:completeness}, we prove the completeness results described in Tables~\ref{table:completeness:results:a}, \ref{table:completeness:results:b}, \ref{table:completeness:results:c} and~\ref{table:completeness:results:d}.
In Section~\ref{section:soundness:and:completeness}, we prove the completeness results described in Tables~\ref{table:completeness:results:a}, \ref{table:completeness:results:b}, \ref{table:completeness:results:c} and~\ref{table:completeness:results:d}.
Our proofs are based on the canonical model construction which needs the results presented in Sections~\ref{section:existence:lemmas:and:lindenbaum:lemma} and~\ref{section:canonical:frame:and:canonical:model}.
\begin{table}[ht]
\begin{center}
\begin{tabular}{|c|c|}
\hline
class of frames&intuitionistic modal logic
\\
\hline
${\mathcal C}_{\fcfra}$&$\L_{\fcfra}$
\\
\hline
${\mathcal C}_{\bcfra}$&$\L_{\bcfra}$
\\
\hline
${\mathcal C}_{\dcfra}$&$\L_{\dcfra}$
\\
\hline
${\mathcal C}_{\ucfra}$&?
\\
\hline
${\mathcal C}_{\fbcfra}$&$\L_{\fbcfra}$
\\
\hline
${\mathcal C}_{\fdcfra}$&$\L_{\fdcfra}$
\\
\hline
${\mathcal C}_{\fucfra}$&?
\\
\hline
${\mathcal C}_{\bdcfra}$&$\L_{\bdcfra}$
\\
\hline
${\mathcal C}_{\bucfra}$&?
\\
\hline
${\mathcal C}_{\ducfra}$&?
\\
\hline
${\mathcal C}_{\fbdcfra}$&$\L_{\fbdcfra}$
\\
\hline
${\mathcal C}_{\fbucfra}$&?
\\
\hline
${\mathcal C}_{\fducfra}$&?
\\
\hline
${\mathcal C}_{\bducfra}$&?
\\
\hline
${\mathcal C}_{\fbducfra}$&?
\\
\hline
\end{tabular}
\caption{Completeness results proved in Proposition~\ref{proposition:soundness:completeness}}\label{table:completeness:results:a}
\end{center}
\end{table}
\begin{table}[ht]
\begin{center}
\begin{tabular}{|c|c|}
\hline
class of frames&intuitionistic modal logic
\\
\hline
${\mathcal C}_{\qfcfra}$&$\L_{\fcfra}$
\\
\hline
${\mathcal C}_{\qbcfra}$&$\L_{\bcfra}$
\\
\hline
${\mathcal C}_{\qdcfra}$&$\L_{\dcfra}$
\\
\hline
${\mathcal C}_{\qucfra}$&?
\\
\hline
\end{tabular}
\caption{Completeness results proved in Proposition~\ref{proposition:soundness:completeness}}\label{table:completeness:results:b}
\end{center}
\end{table}
\begin{table}[ht]
\begin{center}
\begin{tabular}{|c|c|}
\hline
class of frames&intuitionistic modal logic
\\
\hline
${\mathcal C}_{\reflexive}$&$\L_{\reflexive}$
\\
\hline
${\mathcal C}_{\symmetric}$&$\L_{\symmetric}$
\\
\hline
${\mathcal C}_{\transitive}$&$\L_{\min}$
\\
\hline
${\mathcal C}_{\reflexive}{\cap}{\mathcal C}_{\symmetric}$&$\L_{\reflexive}{\oplus}\L_{\symmetric}$
\\
\hline
${\mathcal C}_{\reflexive}{\cap}{\mathcal C}_{\transitive}$&$\L_{\reflexive}$
\\
\hline
${\mathcal C}_{\symmetric}{\cap}{\mathcal C}_{\transitive}$&?
\\
\hline
${\mathcal C}_{\partition}$&$\L_{\reflexive}{\oplus}\L_{\symmetric}$
\\
\hline
\end{tabular}
\caption{Completeness results proved in Proposition~\ref{lemma:table:7:lemma}}\label{table:completeness:results:c}
\end{center}
\end{table}
\begin{table}[ht]
\begin{center}
\begin{tabular}{|c|c|}
\hline
class of frames&intuitionistic modal logic
\\
\hline
${\mathcal C}_{\ureflexive}$&$\L_{\ureflexive}$
\\
\hline
${\mathcal C}_{\dreflexive}$&$\L_{\dreflexive}$
\\
\hline
${\mathcal C}_{\usymmetric}$&$\L_{\usymmetric}$
\\
\hline
${\mathcal C}_{\dsymmetric}$&$\L_{\dsymmetric}$
\\
\hline
${\mathcal C}_{\utransitive}$&$\L_{\utransitive}$
\\
\hline
${\mathcal C}_{\dtransitive}$&$\L_{\dtransitive}$
\\
\hline
\end{tabular}
\caption{Completeness results proved in Proposition~\ref{proposition:completeness:uref:dref:etc}}\label{table:completeness:results:d}
\end{center}
\end{table}
\section{Theories}\label{section:theories}
Let $\L$ be an intuitionistic modal logic.
\begin{definition}[Theories]
A {\em $\L$-theory}\/ is a set of formulas containing $\L$ and closed under modus ponens.
\end{definition}
Obviously, for all indexed families $(\Gamma_{i})_{i{\in}I}$ of $\L$-theories, $\bigcap\{\Gamma_{i}\ :\ i{\in}I\}$ is a $\L$-theory and for all nonempty chains $(\Gamma_{i})_{i{\in}I}$ of $\L$-theories, $\bigcup\{\Gamma_{i}\ :\ i{\in}I\}$ is a $\L$-theory.
As a result, there exists a least $\L$-theory (which is nothing but $\L$).
In other respect, obviously, $\Fo$ is the greatest $\L$-theory.
\begin{definition}[Proper theories]
A $\L$-theory $\Gamma$ is {\em proper}\/ if ${\bot}{\not\in}\Gamma$.
\end{definition}
Obviously, for all $\L$-theories $\Gamma$, using axioms and inference rules of $\IPL$, $\Gamma$ is proper if and only if $\Gamma{\not=}\Fo$.
Moreover, for all proper $\L$-theories $\Gamma$, using axiom $(\Axiom4)$, ${\lozenge}{\bot}{\not\in}\Gamma$.
\begin{definition}[Prime theories]
A proper $\L$-theory $\Gamma$ is {\em prime}\/ if for all formulas $A,B$, if $A{\vee}B{\in}\Gamma$ then either $A{\in}\Gamma$, or $B{\in}\Gamma$.
\end{definition}
\begin{definition}[Operations on theories]
For all $\L$-theories $\Gamma$ and for all sets $\Delta$ of formulas, let $\Gamma{+}\Delta{=}\{B{\in}\Fo\ :$ there exists $m{\in}\N$ and there exists $A_{1},\ldots,A_{m}{\in}\Delta$ such that $A_{1}{\wedge}\ldots{\wedge}A_{m}{\rightarrow}B{\in}\Gamma\}$.
For all $\L$-theories $\Gamma$ and and for all formulas $A$, we write $\Gamma{+}A$ instead of $\Gamma{+}\{A\}$.
For all $\L$-theories $\Gamma$, let ${\square}\Gamma{=}\{A{\in}\Fo\ :\ {\square}A{\in}\Gamma\}$.
\end{definition}
\begin{lemma}\label{lemma:theory:gamma:plus:A}
For all $\L$-theories $\Gamma$ and for all sets $\Delta$ of formulas,
\begin{enumerate}
\item $\Gamma{+}\Delta$ is a $\L$-theory,
\item $\Gamma{\subseteq}\Gamma{+}\Delta$,
\item $\Delta{\subseteq}\Gamma{+}\Delta$,
\item for all $\L$-theories $\Lambda$, if $\Gamma{\subseteq}\Lambda$ and $\Delta{\subseteq}\Lambda$ then $\Gamma{+}\Delta{\subseteq}\Lambda$,
\item $\Gamma{+}\Delta$ is proper if and only if for all $m{\in}\N$ and for all $A_{1},\ldots,A_{m}{\in}\Delta$, ${\neg}(A_{1}{\wedge}\ldots
$\linebreak$
{\wedge}A_{m}){\not\in}\Gamma$.
\end{enumerate}
\end{lemma}
\begin{proof}
Use axioms and inference rules of $\IPL$.
\medskip
\end{proof}
\begin{lemma}\label{lemma:theory:square:gamma}
For all $\L$-theories $\Gamma$, ${\square}\Gamma$ is a $\L$-theory.
\end{lemma}
\begin{proof}
Use axiom $(\Axiom1)$ and inference rule $(\Rule1)$.
\medskip
\end{proof}
From now on, we use Lemmas~\ref{lemma:theory:gamma:plus:A} and~\ref{lemma:theory:square:gamma} without explicitly mentioning them.
\section{Existence Lemmas}\label{section:existence:lemmas:and:lindenbaum:lemma}
Let $\L$ be an intuitionistic modal logic.
\begin{lemma}\label{lemma:Zorn:lemma}
Let ${\mathcal S}$ be a set of $\L$-theories.
If ${\mathcal S}$ is nonempty and for all nonempty chains $(\Gamma_{i})_{i{\in}I}$ of elements of ${\mathcal S}$, $\bigcup\{\Gamma_{i}\ :\ i{\in}I\}$ is an element of ${\mathcal S}$ then ${\mathcal S}$ possesses a maximal element.
\end{lemma}
\begin{proof}
This is a standard application of Zorn's Lemma.
See~\cite[Chapter~$1$]{Wechler:1992}.
\medskip
\end{proof}
\begin{lemma}[Existence Lemma for ${\rightarrow}$]\label{lemma:prime:proper:for:implication}
Let $\Gamma$ be a prime $\L$-theory.
Let $B,C$ be formulas.
If $B{\rightarrow}C{\not\in}\Gamma$ then there exists a prime $\L$-theory $\Delta$ such that $\Gamma{\subseteq}\Delta$, $B{\in}\Delta$ and $C{\not\in}\Delta$.
\end{lemma}
\begin{proof}
Suppose $B{\rightarrow}C{\not\in}\Gamma$.
Let ${\mathcal S}{=}\{\Delta\ :\ \Delta$ is a $\L$-theory such that {\bf (1)}~$\Gamma{\subseteq}\Delta$, {\bf (2)}~$B{\in}\Delta$ and {\bf (3)}~$C{\not\in}\Delta\}$.
Since $B{\rightarrow}C{\not\in}\Gamma$, then $C{\not\in}\Gamma{+}B$.
Hence, $\Gamma{+}B\in{\mathcal S}$.
Thus, ${\mathcal S}$ is nonempty.
Moreover, for all nonempty chains $(\Delta_{i})_{i{\in}I}$ of elements of ${\mathcal S}$, $\bigcup\{\Delta_{i}\ :\ i{\in}I\}$ is an element of ${\mathcal S}$.
Consequently, by Lemma~\ref{lemma:Zorn:lemma}, ${\mathcal S}$ possesses a maximal element $\Delta$.
Hence, $\Delta$ is a $\L$-theory such that $\Gamma{\subseteq}\Delta$, $B{\in}\Delta$ and $C{\not\in}\Delta$.
Thus, it only remains to be proved that $\Delta$ is proper and prime.
We claim that $\Delta$ is proper.
If not, $\Delta{=}\Fo$.
Consequently, $C{\in}\Delta$: a contradiction.
We claim that $\Delta$ is prime.
If not, there exists formulas $D,E$ such that $D{\vee}E{\in}\Delta$, $D{\not\in}\Delta$ and $E{\not\in}\Delta$.
Thus, by the maximality of $\Delta$ in ${\mathcal S}$, $\Delta{+}D{\not\in}{\mathcal S}$ and $\Delta{+}E{\not\in}{\mathcal S}$.
Consequently, $C{\in}\Delta{+}D$ and $C{\in}\Delta{+}E$.
Hence, $D{\rightarrow}C{\in}\Delta$ and $E{\rightarrow}C{\in}\Delta$.
Thus, using axioms and inference rules of $\IPL$, $D{\vee}E{\rightarrow}C{\in}\Delta$.
Since $D{\vee}E{\in}\Delta$, then $C{\in}\Delta$: a contradiction.
\medskip
\end{proof}
\begin{lemma}[Existence Lemma for ${\square}$]\label{lemma:prime:proper:for:square}
Let $\Gamma$ be a prime $\L$-theory.
Let $B$ be a formula.
If ${\square}B{\not\in}\Gamma$ then there exists prime $\L$-theories $\Delta,\Lambda$ such that $\Gamma{\subseteq}\Delta$, $\Delta{\bowtie}\Lambda$ and $B{\not\in}\Lambda$.
\end{lemma}
\begin{proof}
Suppose ${\square}B{\not\in}\Gamma$.
Let ${\mathcal S}{=}\{\Delta\ :\ \Delta$ is a $\L$-theory such that {\bf (1)}~$\Gamma{\subseteq}\Delta$ and {\bf (2)}~${\square}B{\not\in}\Delta\}$.
Since ${\square}B{\not\in}\Gamma$, then $\Gamma\in{\mathcal S}$.
Hence, ${\mathcal S}$ is nonempty.
Moreover, for all nonempty chains $(\Delta_{i})_{i{\in}I}$ of elements of ${\mathcal S}$, $\bigcup\{\Delta_{i}\ :\ i{\in}I\}$ is an element of ${\mathcal S}$.
Thus, by Lemma~\ref{lemma:Zorn:lemma}, ${\mathcal S}$ possesses a maximal element $\Delta$.
Consequently, $\Delta$ is a $\L$-theory such that $\Gamma{\subseteq}\Delta$ and ${\square}B{\not\in}\Delta$.
We claim that $\Delta$ is proper.
If not, $\Delta{=}\Fo$.
Hence, ${\square}B{\in}\Delta$: a contradiction.
We claim that $\Delta$ is prime.
If not, there exists formulas $C,D$ such that $C{\vee}D{\in}\Delta$, $C{\not\in}\Delta$ and $D{\not\in}\Delta$.
Consequently, by the maximality of $\Delta$ in ${\mathcal S}$, $\Delta{+}C{\not\in}{\mathcal S}$ and $\Delta{+}D{\not\in}{\mathcal S}$.
Hence, ${\square}B{\in}\Delta{+}C$ and ${\square}B{\in}\Delta{+}D$.
Thus, $C{\rightarrow}{\square}B{\in}\Delta$ and $D{\rightarrow}{\square}B{\in}\Delta$.
Consequently, using axioms and inference rules of $\IPL$, $C{\vee}D{\rightarrow}{\square}B{\in}\Delta$.
Since $C{\vee}D{\in}\Delta$, then ${\square}B{\in}\Delta$: a contradiction.
We claim that $(\ast)$~for all formulas $C$, if $C{\vee}B{\in}{\square}\Delta$ then ${\lozenge}C{\in}\Delta$.
If not, there exists a formula $C$ such that $C{\vee}B{\in}{\square}\Delta$ and ${\lozenge}C{\not\in}\Delta$.
Thus, by the maximality of $\Delta$ in ${\mathcal S}$, $\Delta{+}{\lozenge}C{\not\in}{\mathcal S}$.
Consequently, ${\square}B{\in}\Delta{+}{\lozenge}C$.
Hence, ${\lozenge}C{\rightarrow}{\square}B{\in}\Delta$.
Since $C{\vee}B{\in}{\square}\Delta$, then ${\square}(C{\vee}B){\in}\Delta$.
Since ${\lozenge}C{\rightarrow}{\square}B{\in}\Delta$, then using axiom $(\Axiom2)$, ${\square}B{\in}\Delta$: a contradiction.\footnote{This is our only use of axiom $(\Axiom2)$ in this section.}
Let ${\mathcal T}{=}\{\Lambda\ :\ \Lambda$ is a $\L$-theory such that {\bf (1)}~${\square}\Delta{\subseteq}\Lambda$, {\bf (2)}~for all formulas $C$, if $C{\vee}B{\in}\Lambda$ then ${\lozenge}C{\in}\Delta$ and {\bf (3)}~$B{\not\in}\Lambda\}$.
Since ${\square}B{\not\in}\Delta$, then $B{\not\in}{\square}\Delta$.
Consequently, by~$(\ast)$, ${\square}\Delta\in{\mathcal T}$.
Hence, ${\mathcal T}$ is nonempty.
Moreover, for all nonempty chains $(\Lambda_{i})_{i{\in}I}$ of elements of ${\mathcal T}$, $\bigcup\{\Lambda_{i}\ :\ i{\in}I\}$ is an element of ${\mathcal T}$.
Thus, by Lemma~\ref{lemma:Zorn:lemma}, ${\mathcal T}$ possesses a maximal element $\Lambda$.
Consequently, $\Lambda$ is a $\L$-theory such that ${\square}\Delta{\subseteq}\Lambda$, for all formulas $C$, if $C{\vee}B{\in}\Lambda$ then ${\lozenge}C{\in}\Delta$ and $B{\not\in}\Lambda$.
Hence, it only remains to be proved that $\Lambda$ is proper and prime and $\Delta{\bowtie}\Lambda$.
We claim that $\Lambda$ is proper.
If not, $\Lambda{=}\Fo$.
Thus, $B{\in}\Lambda$: a contradiction.
We claim that $\Lambda$ is prime.
If not, there exists formulas $C,D$ such that $C{\vee}D{\in}\Lambda$, $C{\not\in}\Lambda$ and $D{\not\in}\Lambda$.
Hence, by the maximality of $\Lambda$ in ${\mathcal T}$, $\Lambda{+}C{\not\in}{\mathcal T}$ and $\Lambda{+}D{\not\in}{\mathcal T}$.
Thus, either there exists a formula $E$ such that $E{\vee}B{\in}\Lambda{+}C$ and ${\lozenge}E{\not\in}\Delta$, or $B{\in}\Lambda{+}C$ and either there exists a formula $F$ such that $F{\vee}B{\in}\Lambda{+}D$ and ${\lozenge}F{\not\in}\Delta$, or $B{\in}\Lambda{+}D$.
Consequently, we have to consider the following $4$ cases.
$\mathbf{(1)}$ Case ``there exists a formula $E$ such that $E{\vee}B{\in}\Lambda{+}C$ and ${\lozenge}E{\not\in}\Delta$ and there exists a formula $F$ such that $F{\vee}B{\in}\Lambda{+}D$ and ${\lozenge}F{\not\in}\Delta$'':
Hence, $C{\rightarrow}E{\vee}B{\in}\Lambda$ and $D{\rightarrow}F{\vee}B{\in}\Lambda$.
Thus, using axioms and inference rules of $\IPL$, $C{\vee}D{\rightarrow}E{\vee}F{\vee}B{\in}\Lambda$.
Since $C{\vee}D{\in}\Lambda$, then $E{\vee}F{\vee}B{\in}\Lambda$.
Consequently, ${\lozenge}(E{\vee}F){\in}\Delta$.
Hen\-ce, using axiom $(\Axiom3)$, either ${\lozenge}E{\in}\Delta$, or ${\lozenge}F{\in}\Delta$: a contradiction.
$\mathbf{(2)}$ Case ``there exists a formula $E$ such that $E{\vee}B{\in}\Lambda{+}C$ and ${\lozenge}E{\not\in}\Delta$ and $B{\in}\Lambda{+}
$\linebreak$
D$'':
Thus, $C{\rightarrow}E{\vee}B{\in}\Lambda$ and $D{\rightarrow}B{\in}\Lambda$.
Consequently, using axioms and inference rules of $\IPL$, $C{\vee}D{\rightarrow}E{\vee}B{\in}\Lambda$.
Since $C{\vee}D{\in}\Lambda$, then $E{\vee}B{\in}\Lambda$.
Hence, ${\lozenge}E{\in}\Delta$: a contradiction.
$\mathbf{(3)}$ Case ``$B{\in}\Lambda{+}C$ and there exists a formula $F$ such that $F{\vee}B{\in}\Lambda{+}D$ and ${\lozenge}F{\not\in}
$\linebreak$
\Delta$'':
Thus, $C{\rightarrow}B{\in}\Lambda$ and $D{\rightarrow}F{\vee}B{\in}\Lambda$.
Consequently, using axioms and inference rules of $\IPL$, $C{\vee}D{\rightarrow}F{\vee}B{\in}\Lambda$.
Since $C{\vee}D{\in}\Lambda$, then $F{\vee}B{\in}\Lambda$.
Hence, ${\lozenge}F{\in}\Delta$: a contradiction.
$\mathbf{(4)}$ Case ``$B{\in}\Lambda{+}C$ and $B{\in}\Lambda{+}D$'':
Thus, $C{\rightarrow}B{\in}\Lambda$ and $D{\rightarrow}B{\in}\Lambda$.
Consequently, using axioms and inference rules of $\IPL$, $C{\vee}D{\rightarrow}B{\in}\Lambda$.
Since $C{\vee}D{\in}\Lambda$, then $B{\in}\Lambda$: a contradiction.
We claim that $\Delta{\bowtie}\Lambda$.
If not, there exists a formula $C$ such that $C{\in}\Lambda$ and ${\lozenge}C{\not\in}\Delta$.
Thus, using axioms and inference rules of $\IPL$, $C{\vee}B{\in}\Lambda$.
Consequently ${\lozenge}C{\in}\Delta$: a contradiction.
\medskip
\end{proof}
\begin{lemma}[Existence Lemma for ${\lozenge}$]\label{lemma:prime:proper:for:lozenge}
Let $\Gamma$ be a prime $\L$-theory.
Let $B$ be a formula.
If ${\lozenge}B{\in}\Gamma$ then there exists prime $\L$-theories $\Delta,\Lambda$ such that $\Gamma{{\supseteq}}\Delta$, $\Delta{\bowtie}\Lambda$ and $B{\in}\Lambda$.
\end{lemma}
\begin{proof}
Suppose ${\lozenge}B{\in}\Gamma$.
Let ${\mathcal S}{=}\{\Lambda\ :\ \Lambda$ is a $\L$-theory such that {\bf (1)}~for all formulas $C$, if $C{\in}\Lambda$ then ${\lozenge}C{\in}\Gamma$ and {\bf (2)}~$B{\in}\Lambda\}$.
We claim that $\L{+}B\in{\mathcal S}$.
If not, there exists a formula $C$ such that $C{\in}\L{+}B$ and ${\lozenge}C{\not\in}\Gamma$.
Hence, $B{\rightarrow}C{\in}\L$.
Thus, using inference rule $(\Rule2)$, ${\lozenge}B{\rightarrow}{\lozenge}C{\in}\L$.
Since ${\lozenge}B{\in}\Gamma$, then ${\lozenge}C{\in}\Gamma$: a contradiction.
Hence, ${\mathcal S}$ is nonempty.
Moreover, for all nonempty chains $(\Lambda_{i})_{i{\in}I}$ of elements of ${\mathcal S}$, $\bigcup\{\Lambda_{i}\ :\ i{\in}I\}$ is an element of ${\mathcal S}$.
Thus, by Lemma~\ref{lemma:Zorn:lemma}, ${\mathcal S}$ possesses a maximal element $\Lambda$.
Consequently, $\Lambda$ is a $\L$-theory such that for all formulas $C$, if $C{\in}\Lambda$ then ${\lozenge}C{\in}\Gamma$ and $B{\in}\Lambda$.
We claim that $\Lambda$ is proper.
If not, ${\bot}{\in}\Lambda$.
Hence, ${\lozenge}{\bot}{\in}\Gamma$: a contradiction.
We claim that $\Lambda$ is prime.
If not, there exists formulas $C,D$ such that $C{\vee}D{\in}\Lambda$, $C{\not\in}\Lambda$ and $D{\not\in}\Lambda$.
Consequently, by the maximality of $\Lambda$ in ${\mathcal S}$, $\Lambda{+}C{\not\in}{\mathcal S}$ and $\Lambda{+}D{\not\in}{\mathcal S}$.
Hence, there exists a formula $E$ such that $E{\in}\Lambda{+}C$ and ${\lozenge}E{\not\in}\Gamma$ and there exists a formula $F$ such that $F{\in}\Lambda{+}D$ and ${\lozenge}F{\not\in}\Gamma$.
Thus, $C{\rightarrow}E{\in}\Lambda$ and $D{\rightarrow}F{\in}\Lambda$.
Consequently, using axioms and inference rules of $\IPL$, $C{\vee}D{\rightarrow}E{\vee}F{\in}\Lambda$.
Since $C{\vee}D{\in}\Lambda$, then $E{\vee}F{\in}\Lambda$.
Hence, ${\lozenge}(E{\vee}F){\in}\Gamma$.
Thus, using axiom $(\Axiom3)$, either ${\lozenge}E{\in}\Gamma$, or ${\lozenge}F{\in}\Gamma$: a contradiction.
Let $C_{1},C_{2},\ldots$ be an enumeration of $\Lambda$.
Obviously, using axioms and inference rules of $\IPL$, for all $n{\in}\N$, $C_{1}{\wedge}\ldots{\wedge}C_{n}{\in}\Lambda$.
For all $n{\in}\N$, let $\Delta_{n}{=}\L{+}{\lozenge}(C_{1}{\wedge}\ldots{\wedge}
$\linebreak$
C_{n})$.
Obviously, using inference rule $(\Rule2)$, $(\Delta_{n})_{n{\in}\N}$ is a nonempty chain of $\L$-theories.
We claim that $(\ast)$~$\Gamma{\supseteq}\bigcup\{\Delta_{n}\ :\ n{\in}\N\}$.
If not, there exists $n{\in}\N$ and there exists a formula $D$ such that $D{\not\in}\Gamma$ and $D{\in}\Delta_{n}$.
Hence, ${\lozenge}(C_{1}{\wedge}\ldots{\wedge}C_{n}){\rightarrow}D{\in}\L$.
Since $D{\not\in}\Gamma$, then ${\lozenge}(C_{1}{\wedge}\ldots{\wedge}C_{n}){\not\in}\Gamma$.
Thus, $C_{1}{\wedge}\ldots{\wedge}C_{n}{\not\in}\Lambda$: a contradiction.
We claim that $(\ast\ast)$~for all formulas $D,E$, if $D{\not\in}\Gamma$ and $D{\vee}{\square}E{\in}\bigcup\{\Delta_{n}\ :\ n{\in}\N\}$ then $E{\in}\Lambda$.
If not, there exists $n{\in}\N$ and there exists formulas $D,E$ such that $D{\not\in}\Gamma$, $D{\vee}{\square}E{\in}\Delta_{n}$ and $E{\not\in}\Lambda$.
Hence, by the maximality of $\Lambda$ in ${\mathcal S}$, $\Lambda{+}E{\not\in}{\mathcal S}$.
Thus, there exists a formula $F$ such that $F{\in}\Lambda{+}E$ and ${\lozenge}F{\not\in}\Gamma$.
Consequently, $E{\rightarrow}F{\in}\Lambda$.
Hence, let $m{\in}\N$ be such that $C_{m}{=}E{\rightarrow}F$.
Let $o{=}\max\{m,n\}$.
Since $D{\vee}{\square}E{\in}\Delta_{n}$, then $D{\vee}{\square}E{\in}\Delta_{o}$.
Thus, ${\lozenge}(C_{1}{\wedge}\ldots{\wedge}C_{o}){\rightarrow}D{\vee}{\square}E{\in}\L$.
Since $C_{m}{=}E{\rightarrow}F$, then using axioms and inference rules of $\IPL$, $E{\rightarrow}(C_{1}{\wedge}\ldots{\wedge}C_{o}{\rightarrow}F){\in}\L$.
Since ${\lozenge}(C_{1}{\wedge}\ldots{\wedge}C_{o})
$\linebreak$
{\rightarrow}D{\vee}{\square}E{\in}\L$, then using axiom $(\Axiom1)$ and inference rule $(\Rule1)$, ${\lozenge}(C_{1}{\wedge}\ldots{\wedge}C_{o}){\rightarrow}D{\vee}
$\linebreak$
{\square}(C_{1}{\wedge}\ldots{\wedge}C_{o}{\rightarrow}F){\in}\L$.
Consequently, using inference rule $(\Rule3)$, ${\lozenge}(C_{1}{\wedge}\ldots{\wedge}C_{o}){\rightarrow}
$\linebreak$
D{\vee}{\lozenge}F{\in}\L$.\footnote{This is our only use of inference rule $(\Rule3)$ in this section.}
Sin\-ce $C_{1}{\wedge}\ldots{\wedge}C_{o}{\in}\Lambda$, then ${\lozenge}(C_{1}{\wedge}\ldots{\wedge}C_{o}){\in}\Gamma$.
Sin\-ce ${\lozenge}(C_{1}{\wedge}\ldots{\wedge}C_{o})
$\linebreak$
{\rightarrow}D{\vee}{\lozenge}F{\in}\L$, then $D{\vee}{\lozenge}F{\in}\Gamma$.
Sin\-ce $D{\not\in}\Gamma$ and ${\lozenge}F{\not\in}\Gamma$, then $D{\vee}{\lozenge}F{\not\in}\Gamma$: a contradiction.
We claim that $(\ast\ast\ast)$~for all formulas $D$, if $D{\in}\Lambda$ then ${\lozenge}D{\in}\bigcup\{\Delta_{n}\ :\ n{\in}\N\}$.
If not, there exists a formula $D$ such that $D{\in}\Lambda$ and ${\lozenge}D{\not\in}\bigcup\{\Delta_{n}\ :\ n{\in}\N\}$.
Hence, let $n{\in}\N$ be such that $C_{n}{=}D$.
Thus, using axioms and inference rules of $\IPL$, $C_{1}{\wedge}\ldots{\wedge}C_{n}{\rightarrow}D{\in}
$\linebreak$
\L$.
Consequently, using inference rule $(\Rule2)$, ${\lozenge}(C_{1}{\wedge}\ldots{\wedge}C_{n}){\rightarrow}{\lozenge}D{\in}\L$.
Hence, ${\lozenge}D{\in}
$\linebreak$
\Delta_{n}$.
Thus, ${\lozenge}D{\in}\bigcup\{\Delta_{n}\ :\ n{\in}\N\}$: a contradiction.
Let ${\mathcal T}{=}\{\Delta\ :\ \Delta$ is a $\L$-theory such that {\bf (1)}~$\Gamma{\supseteq}\Delta$ and {\bf (2)}~$\Delta{\bowtie^{\Gamma}}\Lambda\}$.
Obviously, by~$(\ast)$, $(\ast\ast)$ and~$(\ast\ast\ast)$, $\bigcup\{\Delta_{n}\ :\ n{\in}\N\}\in{\mathcal T}$.
Consequently, ${\mathcal T}$ is nonempty.
Moreover, for all nonempty chains $(\Theta_{i})_{i{\in}I}$ of elements of ${\mathcal T}$, $\bigcup\{\Theta_{i}\ :\ i{\in}I\}$ is an element of ${\mathcal T}$.
Hence, by Lemma~\ref{lemma:Zorn:lemma}, ${\mathcal T}$ possesses a maximal element $\Delta$.
Thus, $\Delta$ is a $\L$-theory such that $\Gamma{\supseteq}\Delta$ and $\Delta{\bowtie^{\Gamma}}\Lambda$.
Consequently, it only remains to be proved that $\Delta$ is proper and prime and $\Delta{\bowtie}\Lambda$.
We claim that $\Delta$ is proper.
If not, ${\bot}{\in}\Delta$.
Since, $\Gamma{\supseteq}\Delta$, then ${\bot}{\in}\Gamma$: a contradiction.
We claim that $\Delta$ is prime.
If not, there exists formulas $D,E$ such that $D{\vee}E{\in}\Delta$, $D{\not\in}\Delta$ and $E{\not\in}\Delta$.
Thus, by the maximality of $\Delta$ in ${\mathcal T}$, $\Delta{+}D{\not\in}{\mathcal T}$ and $\Delta{+}E{\not\in}{\mathcal T}$.
Consequently, either there exists a formula $F$ such that $F{\not\in}\Gamma$ and $F{\in}\Delta{+}D$, or there exists formulas $F,G$ such that $F{\not\in}\Gamma$, $F{\vee}{\square}G{\in}\Delta{+}D$ and $G{\not\in}\Lambda$ and either there exists a formula $H$ such that $H{\not\in}\Gamma$ and $H{\in}\Delta{+}E$, or there exists formulas $H,I$ such that $H{\not\in}\Gamma$, $H{\vee}{\square}I{\in}\Delta{+}E$ and $I{\not\in}\Lambda$.
Hence, we have to consider the following $4$ cases.
$\mathbf{(1)}$ Case ``there exists a formula $F$ such that $F{\not\in}\Gamma$ and $F{\in}\Delta{+}D$ and there exists a formula $H$ such that $H{\not\in}\Gamma$ and $H{\in}\Delta{+}E$'':
Thus, $D{\rightarrow}F{\in}\Delta$ and $E{\rightarrow}H{\in}\Delta$.
Consequently, using axioms and inference rules of $\IPL$, $D{\vee}E{\rightarrow}F{\vee}H{\in}\Delta$.
Since $D{\vee}E{\in}\Delta$, then $F{\vee}H{\in}\Delta$.
Since $F{\not\in}\Gamma$ and $H{\not\in}\Gamma$, then $F{\vee}H{\not\in}\Gamma$.
Since $\Gamma{\supseteq}\Delta$, then $F{\vee}H{\not\in}\Delta$: a contradiction.
$\mathbf{(2)}$ Case ``there exists a formula $F$ such that $F{\not\in}\Gamma$ and $F{\in}\Delta{+}D$ and there exists formulas $H,I$ such that $H{\not\in}\Gamma$, $H{\vee}{\square}I{\in}\Delta{+}E$ and $I{\not\in}\Lambda$'':
Hence, $D{\rightarrow}F{\in}\Delta$ and $E{\rightarrow}H{\vee}{\square}I{\in}\Delta$.
Thus, using axioms and inference rules of $\IPL$, $D{\vee}E{\rightarrow}F{\vee}H{\vee}{\square}I{\in}
$\linebreak$
\Delta$.
Since $D{\vee}E{\in}\Delta$, then $F{\vee}H{\vee}{\square}I{\in}\Delta$.
Since $F{\not\in}\Gamma$ and $H{\not\in}\Gamma$, then $F{\vee}H{\not\in}\Gamma$.
Since $F{\vee}H{\vee}{\square}I{\in}\Delta$, then $I{\in}\Lambda$: a contradiction.
$\mathbf{(3)}$ Case ``there exists formulas $F,G$ such that $F{\not\in}\Gamma$, $F{\vee}{\square}G{\in}\Delta{+}D$ and $G{\not\in}\Lambda$ and there exists a formula $H$ such that $H{\not\in}\Gamma$ and $H{\in}\Delta{+}E$'':
Consequently, $D{\rightarrow}F{\vee}{\square}G{\in}
$\linebreak$
\Delta$ and $E{\rightarrow}H{\in}\Delta$.
Hence, using axioms and inference rules of $\IPL$, $D{\vee}E{\rightarrow}F{\vee}H{\vee}
$\linebreak$
{\square}G{\in}\Delta$.
Since $D{\vee}E{\in}\Delta$, then $F{\vee}H{\vee}{\square}G{\in}\Delta$.
Since $F{\not\in}\Gamma$ and $H{\not\in}\Gamma$, then $F{\vee}H{\not\in}\Gamma$.
Since $F{\vee}H{\vee}{\square}G{\in}\Delta$, then $G{\in}\Lambda$: a contradiction.
$\mathbf{(4)}$ Case ``there exists formulas $F,G$ such that $F{\not\in}\Gamma$, $F{\vee}{\square}G{\in}\Delta{+}D$ and $G{\not\in}\Lambda$ and there exists formulas $H,I$ such that $H{\not\in}\Gamma$, $H{\vee}{\square}I{\in}\Delta{+}E$ and $I{\not\in}\Lambda$'':
Thus, $D{\rightarrow}F{\vee}
$\linebreak$
{\square}G{\in}\Delta$ and $E{\rightarrow}H{\vee}{\square}I{\in}\Delta$.
Moreover, $G{\vee}I{\not\in}\Lambda$.
Consequently, using axioms and inference rules of $\IPL$, $D{\vee}E{\rightarrow}F{\vee}H{\vee}{\square}G{\vee}{\square}I{\in}\Delta$.
Since $D{\vee}E{\in}\Delta$, then $F{\vee}H{\vee}{\square}G{\vee}
$\linebreak$
{\square}I{\in}\Delta$.
Obviously, using axiom $(\Axiom1)$ and inference rule $(\Rule1)$, $F{\vee}H{\vee}{\square}G{\vee}{\square}I{\rightarrow}F{\vee}
$\linebreak$
H{\vee}{\square}(G{\vee}I){\in}\L$.
Since $F{\vee}H{\vee}{\square}G{\vee}{\square}I{\in}\Delta$, then $F{\vee}H{\vee}{\square}(G{\vee}I){\in}\Delta$.
Since $F{\not\in}\Gamma$ and $H{\not\in}\Gamma$, then $F{\vee}H{\not\in}\Gamma$.
Since $F{\vee}H{\vee}{\square}(G{\vee}I){\in}\Delta$, then $G{\vee}I{\in}\Lambda$: a contradiction.
We claim that $\Delta{\bowtie}\Lambda$.
If not, there exists a formula $D$ such that ${\square}D{\in}\Delta$ and $D{\not\in}\Lambda$.
Thus, using axioms and inference rules of $\IPL$, ${\bot}{\vee}{\square}D{\in}\Delta$.
Since ${\bot}{\not\in}\Gamma$ and $\Delta{\bowtie^{\Gamma}}\Lambda$, then $D{\in}\Lambda$: a contradiction.
\medskip
\end{proof}
\section{Canonical frame and canonical model}\label{section:canonical:frame:and:canonical:model}
Let $\L$ be an intuitionistic modal logic.
\begin{lemma}[Lindenbaum Lemma]\label{lemma:almost:completeness}
Let $A$ be a formula.
If $A{\not\in}\L$ then there exists a prime $\L$-theory $\Gamma$ such that $A{\not\in}\Gamma$.
\end{lemma}
\begin{proof}
Suppose $A{\not\in}\L$.
Hence, using axioms and inference rules of $\IPL$, ${\top}{\rightarrow}A{\not\in}\L$.
Thus, by Existence Lemma for ${\rightarrow}$, there exists a prime $\L$-theory $\Gamma$ such that $A{\not\in}\Gamma$.
\medskip
\end{proof}
\begin{lemma}\label{lemma:about:consistent:logic:and:proper:theories}
If $\L$ is consistent then there exists a prime $\L$-theory.
\end{lemma}
\begin{proof}
Suppose $\L$ is consistent.
Hence, ${\bot}{\not\in}\L$.
Thus, by Lindenbaum Lemma, there exists a prime $\L$-theory.
\medskip
\end{proof}
From now on in this section, we assume that $\L$ is consistent.
\begin{definition}[Canonical frame]
Let $(W_{\L},{\leq_{\L}},{R_{\L}})$ be the frame such that
\begin{itemize}
\item $W_{\L}$ is the nonempty set of all prime $\L$-theories,
\item $\leq_{\L}$ is the preorder on $W_{\L}$ such that for all $\Gamma,\Delta{\in}W_{\L}$, $\Gamma{\leq_{\L}}\Delta$ if and only if $\Gamma{\subseteq}\Delta$,
\item $R_{\L}$ is the binary relation on $W_{\L}$ such that for all $\Gamma,\Delta{\in}W_{\L}$, $\Gamma{R_{\L}}\Delta$ if and only if $\Gamma{\bowtie}\Delta$.
\end{itemize}
%
%
%Notice that, indeed, $\leq_{\L}$ is a partial order on $W_{\L}$.
The frame $(W_{\L},{\leq_{\L}},{R_{\L}})$ is called {\em canonical frame of $\L$.}
\end{definition}
%
%
%Notice that, indeed, for all atoms $p$, $\{\Gamma{\in}W_{\L}\ :\ p{\in}\Gamma\}$ is $\leq_{L}$-closed.
%
%
\begin{definition}[Canonical model]
Let $V_{\L}\ :\ \At{\longrightarrow}\wp(W_{\L})$ be the valuation on $(W_{\L},
$\linebreak$
{\leq_{\L}},{R_{\L}})$ such that for all atoms $p$, $V_{\L}(p){=}\{\Gamma{\in}W_{\L}\ :\ p{\in}\Gamma\}$.
The valuation $V_{\L}\ :\ \At{\longrightarrow}\wp(W_{\L})$ on $(W_{\L},{\leq_{\L}},{R_{\L}})$ is called {\em canonical valuation of $\L$.}
The model $(W_{\L},
$\linebreak$
{\leq_{\L}},{R_{\L}},V_{\L})$ is called {\em canonical model of $\L$.}
\end{definition}
\begin{lemma}[Truth Lemma]\label{lemma:truth:lemma}
For all formulas $A$ and for all $\Gamma{\in}W_{\L}$, $A{\in}\Gamma$ if and only if $\Gamma{\models}A$.
\end{lemma}
\begin{proof}
By induction on $A$.
We only consider the following $3$ cases.
$\mathbf{(1)}$ Case ``there exists formulas $B,C$ such that $A{=}B{\rightarrow}C$'':
Let $\Gamma{\in}W_{\L}$.
From left to right, suppose $B{\rightarrow}C{\in}\Gamma$ and $\Gamma{\not\models}B{\rightarrow}C$.
Hence, there exists $\Delta{\in}W_{\L}$ such that $\Gamma\leq_{\L}\Delta$, $\Delta{\models}B$ and $\Delta{\not\models}C$.
Thus, $\Gamma{\subseteq}\Delta$.
Moreover, by induction hypothesis, $B{\in}\Delta$ and $C{\not\in}\Delta$.
Since $B{\rightarrow}C{\in}\Gamma$, then $B{\rightarrow}C{\in}\Delta$.
Since $B{\in}\Delta$, then $C{\in}\Delta$: a contradiction.
From right to left, suppose $\Gamma{\models}B{\rightarrow}C$ and $B{\rightarrow}C{\not\in}\Gamma$.
Consequently, by Existence Lemma for ${\rightarrow}$, there exists a prime $\L$-theory $\Delta$ such that $\Gamma{\subseteq}\Delta$, $B{\in}\Delta$ and $C{\not\in}\Delta$.
Hence, $\Gamma\leq_{\L}\Delta$.
Moreover, by induction hypothesis, $\Delta{\models}B$ and $\Delta{\not\models}C$.
Thus, $\Gamma{\not\models}B{\rightarrow}C$: a contradiction.
$\mathbf{(2)}$ Case ``there exists a formula $B$ such that $A{=}{\square}B$'':
Let $\Gamma{\in}W_{\L}$.
From left to right, suppose ${\square}B{\in}\Gamma$ and $\Gamma{\not\models}{\square}B$.
Thus, there exists $\Delta,\Lambda{\in}W_{\L}$ such that $\Gamma\leq_{\L}\Delta$, $\Delta R_{\L}\Lambda$ and $\Lambda{\not\models}B$.
Consequently, $\Gamma{\subseteq}\Delta$ and $\Delta{\bowtie}\Lambda$.
Moreover, by induction hypothesis, $B{\not\in}\Lambda$.
Since ${\square}B{\in}\Gamma$, then $B{\in}\Lambda$: a contradiction.
From right to left, suppose $\Gamma{\models}{\square}B$ and ${\square}B{\not\in}\Gamma$.
Hence, by Existence Lemma for ${\square}$, there exists prime $\L$-theories $\Delta,\Lambda$ such that $\Gamma{\subseteq}\Delta$, $\Delta{\bowtie}\Lambda$ and $B{\not\in}\Lambda$.
Thus, $\Gamma\leq_{\L}\Delta$ and $\Delta R_{\L}\Lambda$.
Moreover, by induction hypothesis, $\Lambda{\not\models}B$.
Consequently, $\Gamma{\not\models}{\square}B$: a contradiction.
$\mathbf{(3)}$ Case ``there exists a formula $B$ such that $A{=}{\lozenge}B$'':
Let $\Gamma{\in}W_{\L}$.
From left to right, suppose ${\lozenge}B{\in}\Gamma$ and $\Gamma{\not\models}{\lozenge}B$.
Consequently, by Existence Lemma for ${\lozenge}$, there exists prime $\L$-theories $\Delta,\Lambda$ such that $\Gamma{\supseteq}\Delta$, $\Delta{\bowtie}\Lambda$ and $B{\in}\Lambda$.
Hence, $\Gamma\geq_{\L}\Delta$ and $\Delta R_{\L}\Lambda$.
Moreover, by induction hypothesis, $\Lambda{\models}B$.
Thus, $\Gamma{\models}{\lozenge}B$: a contradiction.
From right to left, suppose $\Gamma{\models}{\lozenge}B$ and ${\lozenge}B{\not\in}\Gamma$.
Consequently, there exists $\Delta,\Lambda{\in}W_{\L}$ such that $\Gamma\geq_{\L}\Delta$, $\Delta R_{\L}\Lambda$ and $\Lambda{\models}B$.
Hence, $\Gamma{\supseteq}\Delta$ and $\Delta{\bowtie}\Lambda$.
Moreover, by induction hypothesis, $B{\in}\Lambda$.
Since ${\lozenge}B{\not\in}\Gamma$, then $B{\not\in}\Lambda$: a contradiction.
\medskip
\end{proof}
\begin{lemma}\label{fc:canonical:frame:is:forward:confluent}
If $(\Axiom\mathbf{f}){\in}\L$ then $(W_{\L},{\leq_{\L}},{R_{\L}})$ is forward confluent.
\end{lemma}
\begin{proof}
Suppose $(\Axiom\mathbf{f}){\in}\L$.
Let $\Gamma,\Delta,\Lambda{\in}W_{\L}$ be such that $\Gamma\geq_{\L}\Delta$ and $\Delta R_{\L}\Lambda$.
Hence, $\Gamma{\supseteq}\Delta$ and $\Delta{\bowtie}\Lambda$.
Let $A_{1},A_{2},\ldots$ be an enumeration of ${\square}\Gamma$ and $B_{1},B_{2},\ldots$ be an enumeration of $\Lambda$.
Obviously, using axioms and inference rules of $\IPL$, for all $n{\in}\N$, ${\square}(A_{1}{\wedge}\ldots{\wedge}A_{n}){\in}\Gamma$ and $B_{1}{\wedge}\ldots{\wedge}B_{n}{\in}\Lambda$.
Since $\Delta{\bowtie}\Lambda$, then for all $n{\in}\N$, ${\lozenge}(B_{1}{\wedge}\ldots{\wedge}
$\linebreak$
B_{n}){\in}\Delta$.
For all $n{\in}\N$, let $\Theta_{n}{=}\L{+}A_{1}{\wedge}\ldots{\wedge}A_{n}{\wedge}B_{1}{\wedge}\ldots{\wedge}B_{n}$.
Obviously, using axioms and inference rules of $\IPL$, $(\Theta_{n})_{n{\in}\N}$ is a nonempty chain of $\L$-theories such that $\bigcup\{\Theta_{n}\ :\ n{\in}\N\}{\supseteq}\Lambda$.
We claim that $(\ast)$~for all formulas $C$, if ${\square}C{\in}\Gamma$ then $C{\in}\bigcup\{\Theta_{n}\ :\ n{\in}\N\}$.
If not, there exists a formula $C$ such that ${\square}C{\in}\Gamma$ and $C{\not\in}\bigcup\{\Theta_{n}\ :\ n{\in}\N\}$.
Thus, $C{\in}{\square}\Gamma$.
Consequently, let $n{\in}\N$ be such that $A_{n}{=}C$.
Hence, using axioms and inference rules of $\IPL$, $A_{1}{\wedge}\ldots{\wedge}A_{n}{\wedge}B_{1}{\wedge}\ldots{\wedge}B_{n}{\rightarrow}C{\in}\L$.
Thus, $C{\in}\Theta_{n}$.
Consequently, $C{\in}\bigcup\{\Theta_{n}\ :\ n{\in}\N\}$: a contradiction.
We claim that $(\ast\ast)$~for all formulas $C$, if $C{\in}\bigcup\{\Theta_{n}\ :\ n{\in}\N\}$ then ${\lozenge}C{\in}\Gamma$.
If not, there exists $n{\in}\N$ and there exists a formula $C$ such that $C{\in}\Theta_{n}$ and ${\lozenge}C{\not\in}\Gamma$.
Thus, $A_{1}{\wedge}\ldots{\wedge}A_{n}{\wedge}B_{1}{\wedge}\ldots{\wedge}B_{n}{\rightarrow}C{\in}\L$.
Consequently, using axioms and inference rules of $\IPL$, $B_{1}{\wedge}\ldots{\wedge}B_{n}{\rightarrow}(A_{1}{\wedge}\ldots{\wedge}A_{n}{\rightarrow}C){\in}\L$.
Hence, using inference rule $(\Rule2)$, ${\lozenge}(B_{1}{\wedge}\ldots{\wedge}B_{n}){\rightarrow}{\lozenge}(A_{1}{\wedge}\ldots{\wedge}A_{n}{\rightarrow}C){\in}\L$.
Since ${\lozenge}(B_{1}{\wedge}\ldots{\wedge}B_{n}){\in}\Delta$, then ${\lozenge}(A_{1}{\wedge}
$\linebreak$
\ldots{\wedge}A_{n}{\rightarrow}C){\in}\Delta$.
Since $\Gamma{\supseteq}\Delta$, then ${\lozenge}(A_{1}{\wedge}\ldots{\wedge}A_{n}{\rightarrow}C){\in}\Gamma$.
Thus, using the fact that $(\Axiom\mathbf{f}){\in}\L$, ${\square}(A_{1}{\wedge}\ldots{\wedge}A_{n}){\rightarrow}{\lozenge}C{\in}\Gamma$.
Since ${\square}(A_{1}{\wedge}\ldots{\wedge}A_{n}){\in}\Gamma$, then ${\lozenge}C{\in}\Gamma$: a contradiction.
Let ${\mathcal S}{=}\{\Theta\ :\ \Theta$ is a $\L$-theory such that {\bf (1)}~$\Gamma{\bowtie}\Theta$ and {\bf (2)}~$\Theta{\supseteq}\Lambda\}$.
Obviously, by~$(\ast)$ and~$(\ast\ast)$, $\bigcup\{\Theta_{n}\ :\ n{\in}\N\}\in{\mathcal S}$.
Hence, ${\mathcal S}$ is nonempty.
Moreover, for all nonempty chains $(\Pi_{i})_{i{\in}I}$ of elements of ${\mathcal S}$, $\bigcup\{\Pi_{i}\ :\ i{\in}I\}$ is an element of ${\mathcal S}$.
Thus, by Lemma~\ref{lemma:Zorn:lemma}, ${\mathcal S}$ possesses a maximal element $\Theta$.
Consequently, $\Theta$ is a $\L$-theory such that $\Gamma{\bowtie}\Theta$ and $\Theta{\supseteq}\Lambda$.
Hence, it only remains to be proved that $\Theta$ is proper and prime.
We claim that $\Theta$ is proper.
If not, ${\bot}{\in}\Theta$.
Since $\Gamma{\bowtie}\Theta$, then ${\lozenge}{\bot}{\in}\Gamma$: a contradiction.
We claim that $\Theta$ is prime.
If not, there exists formulas $C,D$ such that $C{\vee}D{\in}\Theta$, $C{\not\in}\Theta$ and $D{\not\in}\Theta$.
Consequently, by the maximality of $\Theta$ in ${\mathcal S}$, $\Theta{+}C{\not\in}{\mathcal S}$ and $\Theta{+}D{\not\in}{\mathcal S}$.
Hence, there exists a formula $E$ such that $E{\in}\Theta{+}C$ and ${\lozenge}E{\not\in}\Gamma$ and there exists a formula $F$ such that $F{\in}\Theta{+}D$ and ${\lozenge}F{\not\in}\Gamma$.
Thus, $C{\rightarrow}E{\in}\Theta$ and $D{\rightarrow}F{\in}\Theta$.
Consequently, using axioms and inference rules of $\IPL$, $C{\vee}D{\rightarrow}E{\vee}F{\in}\Theta$.
Since $C{\vee}D{\in}\Theta$, then $E{\vee}F{\in}\Theta$.
Since $\Gamma{\bowtie}\Theta$, then ${\lozenge}(E{\vee}F){\in}\Gamma$.
Hence, using axiom $(\Axiom3)$, either ${\lozenge}E{\in}\Gamma$, or ${\lozenge}F{\in}\Gamma$: a contradiction.
%
%
%\\
%\\
%
%
%All in all, we have proved that $\Theta{\in}W_{\L}$ is such that $\Gamma R_{\L}\Theta$ and $\Theta\geq_{\L}\Lambda$.
%
%
\medskip
\end{proof}
\begin{lemma}\label{bc:canonical:frame:is:backward:confluent}
If $(\Axiom\mathbf{b}){\in}\L$ then $(W_{\L},{\leq_{\L}},{R_{\L}})$ is backward confluent.
\end{lemma}
\begin{proof}
Suppose $(\Axiom\mathbf{b}){\in}\L$.
Let $\Gamma,\Delta,\Lambda{\in}W_{\L}$ be such that $\Gamma R_{\L}\Delta$ and $\Delta\leq_{\L}\Lambda$.
Hence, $\Gamma{\bowtie}\Delta$ and $\Delta{\subseteq}\Lambda$.
Let $A_{1},A_{2},\ldots$ be an enumeration of $\Lambda$.
Obviously, using axioms and inference rules of $\IPL$, for all $n{\in}\N$, $A_{1}{\wedge}\ldots{\wedge}A_{n}{\in}\Lambda$.
For all $n{\in}\N$, let $\Theta_{n}{=}\Gamma{+}{\lozenge}(A_{1}{\wedge}\ldots{\wedge}A_{n})$.
Obviously, using inference rule $(\Rule2)$, $(\Theta_{n})_{n{\in}\N}$ is a nonempty chain of $\L$-theories such that $\Gamma{\subseteq}\bigcup\{\Theta_{n}\ :\ n{\in}\N\}$.
We claim that $(\ast)$~for all formulas $B$, if ${\square}B{\in}\bigcup\{\Theta_{n}\ :\ n{\in}\N\}$ then $B{\in}\Lambda$.
If not, there exists $n{\in}\N$ and there exists a formula $B$ such that ${\square}B{\in}\Theta_{n}$ and $B{\not\in}\Lambda$.
Thus, ${\lozenge}(A_{1}{\wedge}\ldots{\wedge}A_{n}){\rightarrow}{\square}B{\in}\Gamma$.
Consequently, using the fact that $(\Axiom\mathbf{b}){\in}\L$, ${\square}(A_{1}{\wedge}\ldots{\wedge}
$\linebreak$
A_{n}{\rightarrow}B){\in}\Gamma$.
Since $\Gamma{\bowtie}\Delta$, then $A_{1}{\wedge}\ldots{\wedge}A_{n}{\rightarrow}B{\in}\Delta$.
Since $\Delta{\subseteq}\Lambda$, then $A_{1}{\wedge}\ldots{\wedge}A_{n}
$\linebreak$
{\rightarrow}B{\in}\Lambda$.
Sin\-ce $A_{1}{\wedge}\ldots{\wedge}A_{n}{\in}\Lambda$, then $B{\in}\Lambda$: a contradiction.
We claim that $(\ast\ast)$~for all formulas $B$, if $B{\in}\Lambda$ then ${\lozenge}B{\in}\bigcup\{\Theta_{n}\ :\ n{\in}\N\}$.
If not, there exists a formula $B$ such that $B{\in}\Lambda$ and ${\lozenge}B{\not\in}\bigcup\{\Theta_{n}\ :\ n{\in}\N\}$.
Thus, let $n{\in}\N$ be such that $A_{n}{=}B$.
Consequently, using axioms and inference rules of $\IPL$, $A_{1}{\wedge}\ldots{\wedge}A_{n}{\rightarrow}B{\in}\L$.
Hence, using inference rule $(\Rule2)$, ${\lozenge}(A_{1}{\wedge}\ldots{\wedge}A_{n}){\rightarrow}{\lozenge}B{\in}\L$.
Thus, ${\lozenge}B{\in}\Theta_{n}$.
Consequently, ${\lozenge}B{\in}\bigcup\{\Theta_{n}\ :\ n{\in}\N\}$: a contradiction.
Let ${\mathcal S}{=}\{\Theta\ :\ \Theta$ is a $\L$-theory such that {\bf (1)}~$\Gamma{\subseteq}\Theta$ and {\bf (2)}~$\Theta{\bowtie}\Lambda\}$.
Obviously, by~$(\ast)$ and~$(\ast\ast)$, $\bigcup\{\Theta_{n}\ :\ n{\in}\N\}\in{\mathcal S}$.
Thus, ${\mathcal S}$ is nonempty.
Moreover, for all nonempty chains $(\Pi_{i})_{i{\in}I}$ of elements of ${\mathcal S}$, $\bigcup\{\Pi_{i}\ :\ i{\in}I\}$ is an element of ${\mathcal S}$.
Consequently, by Lemma~\ref{lemma:Zorn:lemma}, ${\mathcal S}$ possesses a maximal element $\Theta$.
Hence, $\Theta$ is a $\L$-theory such that $\Gamma{\subseteq}\Theta$ and $\Theta{\bowtie}\Lambda$.
Thus, it only remains to be proved that $\Theta$ is proper and prime.
We claim that $\Theta$ is proper.
If not, $\Theta{=}\Fo$.
Consequently ${\square}{\bot}{\in}\Theta$.
Since $\Theta{\bowtie}\Lambda$, then ${\bot}{\in}\Lambda$: a contradiction.
We claim that $\Theta$ is prime.
If not, there exists formulas $B,C$ such that $B{\vee}C{\in}\Theta$, $B{\not\in}\Theta$ and $C{\not\in}\Theta$.
Thus, by the maximality of $\Theta$ in ${\mathcal S}$, $\Theta{+}B{\not\in}{\mathcal S}$ and $\Theta{+}C{\not\in}{\mathcal S}$.
Consequently, there exists a formula $D$ such that ${\square}D{\in}\Theta{+}B$ and $D{\not\in}\Lambda$ and there exists a formula $E$ such that ${\square}E{\in}\Theta{+}C$ and $E{\not\in}\Lambda$.
Hence, $B{\rightarrow}{\square}D{\in}\Theta$ and $C{\rightarrow}{\square}E{\in}\Theta$.
Thus, using axioms and inference rules of $\IPL$, $B{\vee}C{\rightarrow}{\square}D{\vee}{\square}E{\in}\Theta$.
Since $B{\vee}C{\in}\Theta$, then ${\square}D{\vee}{\square}E{\in}\Theta$.
Obviously, using axiom $(\Axiom1)$ and inference rule $(\Rule1)$, ${\square}D{\vee}{\square}E{\rightarrow}
$\linebreak$
{\square}(D{\vee}E){\in}\L$.
Since ${\square}D{\vee}{\square}E{\in}\Theta$, then ${\square}(D{\vee}E){\in}\Theta$.
Since $\Theta{\bowtie}\Lambda$, then $D{\vee}E{\in}\Lambda$.
Since $D{\not\in}\Lambda$ and $E{\not\in}\Lambda$, then $D{\vee}E{\not\in}\Lambda$: a contradiction.
%
%
%\\
%\\
%
%
%All in all, we have proved that $\Theta{\in}W_{\L}$ is such that $\Gamma\leq_{\L}\Theta$ and $\Theta R_{\L}\Lambda$.
%
%
\medskip
\end{proof}
\begin{lemma}\label{dc:canonical:frame:is:downward:confluent}
If $(\Axiom\mathbf{d}){\in}\L$ then $(W_{\L},{\leq_{\L}},{R_{\L}})$ is downward confluent.
\end{lemma}
\begin{proof}
Suppose $(\Axiom\mathbf{d}){\in}\L$.
Let $\Gamma,\Delta,\Lambda{\in}W_{\L}$ be such that $\Gamma\leq_{\L}\Delta$ and $\Delta R_{\L}\Lambda$.
Hence, $\Gamma{\subseteq}\Delta$ and $\Delta{\bowtie}\Lambda$.
We claim that $(\ast)$~${\square}\Gamma{\subseteq}\Lambda$.
If not, there exists a formula $A$ such that $A{\in}{\square}\Gamma$ and $A{\not\in}\Lambda$.
Thus, ${\square}A{\in}\Gamma$.
Since $\Gamma{\subseteq}\Delta$, then ${\square}A{\in}\Delta$.
Since $\Delta{\bowtie}\Lambda$, then $A{\in}\Lambda$: a contradiction.
We claim that $(\ast\ast)$~for all formulas $A,B$, if ${\lozenge}A{\not\in}\Gamma$ and $A{\vee}B{\in}{\square}\Gamma$ then $B{\in}\Lambda$.
If not, there exists formulas $A,B$ such that ${\lozenge}A{\not\in}\Gamma$, $A{\vee}B{\in}{\square}\Gamma$ and $B{\not\in}\Lambda$.
Hence, ${\square}(A{\vee}B){\in}\Gamma$.
Thus, using the fact that $(\Axiom\mathbf{d}){\in}\L$, ${\lozenge}A{\vee}{\square}B{\in}\Gamma$.
Consequently, either ${\lozenge}A{\in}\Gamma$, or ${\square}B{\in}\Gamma$.
Since ${\lozenge}A{\not\in}\Gamma$, then ${\square}B{\in}\Gamma$.
Since $\Gamma{\subseteq}\Delta$, then ${\square}B{\in}\Delta$.
Since $\Delta{\bowtie}\Lambda$, then $B{\in}\Lambda$: a contradiction.

Let ${\mathcal S}{=}\{\Theta\ :\ \Theta$ is a $\L$-theory such that {\bf (1)}~${\square}\Gamma{\subseteq}\Theta$, {\bf (2)}~$\Theta{\subseteq}\Lambda$ and {\bf (3)}~for all formulas $A,B$, if ${\lozenge}A{\not\in}\Gamma$ and $A{\vee}B{\in}\Theta$ then $B{\in}\Lambda\}$.
Obviously, by~$(\ast)$ and~$(\ast\ast)$, ${\square}\Gamma\in{\mathcal S}$.
Thus, ${\mathcal S}$ is nonempty.
Moreover, for all nonempty chains $(\Theta_{i})_{i{\in}I}$ of elements of ${\mathcal S}$, $\bigcup\{\Theta_{i}\ :\ i{\in}I\}$ is an element of ${\mathcal S}$.
Consequently, by Lemma~\ref{lemma:Zorn:lemma}, ${\mathcal S}$ possesses a maximal element $\Theta$.
Hence, $\Theta$ is a $\L$-theory such that ${\square}\Gamma{\subseteq}\Theta$, $\Theta{\subseteq}\Lambda$ and for all formulas $A,B$, if ${\lozenge}A{\not\in}\Gamma$ and $A{\vee}B{\in}\Theta$ then $B{\in}\Lambda$.
Thus, it only remains to be proved that $\Theta$ is proper and prime and $\Gamma{\bowtie}\Theta$.
We claim that $\Theta$ is proper.
If not, ${\bot}{\in}\Theta$.
Since $\Theta{\subseteq}\Lambda$, then ${\bot}{\in}\Lambda$: a contradiction.
We claim that $\Theta$ is prime.
If not, there exists formulas $A,B$ such that $A{\vee}B{\in}\Theta$, $A{\not\in}\Theta$ and $B{\not\in}\Theta$.
Hence, by the maximality of $\Theta$ in ${\mathcal S}$, $\Theta{+}A{\not\in}{\mathcal S}$ and $\Theta{+}B{\not\in}{\mathcal S}$.
Thus, either there exists a formula $C$ such that $C{\in}\Theta{+}A$ and $C{\not\in}\Lambda$, or there exists formulas $C,D$ such that ${\lozenge}C{\not\in}\Gamma$, $C{\vee}D{\in}\Theta{+}A$ and $D{\not\in}\Lambda$ and either there exists a formula $E$ such that $E{\in}\Theta{+}B$ and $E{\not\in}\Lambda$, or there exists formulas $E,F$ such that ${\lozenge}E{\not\in}\Gamma$, $E{\vee}F{\in}\Theta{+}B$ and $F{\not\in}\Lambda$.
Consequently, we have to consider the following $4$ cases.
$\mathbf{(1)}$ Case ``there exists a formula $C$ such that $C{\in}\Theta{+}A$ and $C{\not\in}\Lambda$ and there exists a formula $E$ such that $E{\in}\Theta{+}B$ and $E{\not\in}\Lambda$'':
Hence, $A{\rightarrow}C{\in}\Theta$ and $B{\rightarrow}E{\in}\Theta$.
Thus, using axioms and inference rules of $\IPL$, $A{\vee}B{\rightarrow}C{\vee}E{\in}\Theta$.
Since $A{\vee}B{\in}\Theta$, then $C{\vee}E{\in}\Theta$.
Since $\Theta{\subseteq}\Lambda$, then $C{\vee}E{\in}\Lambda$.
Since $C{\not\in}\Lambda$ and $E{\not\in}\Lambda$, then $C{\vee}E{\not\in}\Lambda$: a contradiction.
$\mathbf{(2)}$ Case ``there exists a formula $C$ such that $C{\in}\Theta{+}A$ and $C{\not\in}\Lambda$ and there exists formulas $E,F$ such that ${\lozenge}E{\not\in}\Gamma$, $E{\vee}F{\in}\Theta{+}B$ and $F{\not\in}\Lambda$'':
Consequently, $A{\rightarrow}C{\in}\Theta$ and $B{\rightarrow}E{\vee}F{\in}\Theta$.
Hence, using axioms and inference rules of $\IPL$, $A{\vee}B{\rightarrow}E{\vee}C{\vee}F
$\linebreak$
{\in}\Theta$.
Since $A{\vee}B{\in}\Theta$, then $E{\vee}C{\vee}F{\in}\Theta$.
Since ${\lozenge}E{\not\in}\Gamma$, then $C{\vee}F{\in}\Lambda$.
Since $C{\not\in}\Lambda$ and $F{\not\in}\Lambda$, then $C{\vee}F{\not\in}\Lambda$: a contradiction.
$\mathbf{(3)}$ Case ``there exists formulas $C,D$ such that ${\lozenge}C{\not\in}\Gamma$, $C{\vee}D{\in}\Theta{+}A$ and $D{\not\in}\Lambda$ and there exists a formula $E$ such that $E{\in}\Theta{+}B$ and $E{\not\in}\Lambda$'':
Thus, $A{\rightarrow}C{\vee}D{\in}\Theta$ and $B{\rightarrow}E{\in}\Theta$.
Consequently, using axioms and inference rules of $\IPL$, $A{\vee}B{\rightarrow}C{\vee}D{\vee}E
$\linebreak$
{\in}\Theta$.
Since $A{\vee}B{\in}\Theta$, then $C{\vee}D{\vee}E{\in}\Theta$.
Since ${\lozenge}C{\not\in}\Gamma$, then $D{\vee}E{\in}\Lambda$.
Since $D{\not\in}\Lambda$ and $E{\not\in}\Lambda$, then $D{\vee}E{\not\in}\Lambda$: a contradiction.
$\mathbf{(4)}$ Case ``there exists formulas $C,D$ such that ${\lozenge}C{\not\in}\Gamma$, $C{\vee}D{\in}\Theta{+}A$ and $D{\not\in}\Lambda$ and there exists formulas $E,F$ such that ${\lozenge}E{\not\in}\Gamma$, $E{\vee}F{\in}\Theta{+}B$ and $F{\not\in}\Lambda$'':
Hence, $A{\rightarrow}C{\vee}D{\in}\Theta$ and $B{\rightarrow}E{\vee}F{\in}\Theta$.
Thus, using axioms and inference rules of $\IPL$, $A{\vee}B{\rightarrow}C{\vee}E{\vee}D{\vee}F{\in}\Theta$.
Since $A{\vee}B{\in}\Theta$, then $C{\vee}E{\vee}D{\vee}F{\in}\Theta$.
Since ${\lozenge}C{\not\in}\Gamma$ and ${\lozenge}E{\not\in}\Gamma$, then using axiom $(\Axiom3)$, ${\lozenge}(C{\vee}E){\not\in}\Gamma$.
Since $C{\vee}E{\vee}D{\vee}F{\in}\Theta$, then $D{\vee}F{\in}\Lambda$.
Since $D{\not\in}\Lambda$ and $F{\not\in}\Lambda$, then $D{\vee}F{\not\in}\Lambda$: a contradiction.
We claim that $\Gamma{\bowtie}\Theta$.
If not, there exists a formula $A$ such that $A{\in}\Theta$ and ${\lozenge}A{\not\in}\Gamma$.
Hence, using axioms and inference rules of $\IPL$, $A{\vee}{\bot}{\in}\Theta$.
Since ${\lozenge}A{\not\in}\Gamma$, then ${\bot}{\in}\Lambda$: a contradiction.
%
%
%\\
%\\
%
%
%All in all, we have proved that $\Theta{\in}W_{\L}$ is such that $\Gamma R_{\L}\Theta$ and $\Theta\leq_{\L}\Lambda$.
%
%
\medskip
\end{proof}
%
%
%Lemma~\ref{lemma:about:what:happens:if:T:square:is:in:the:logic} is used in Section~\ref{section:soundness:and:completeness} when we show that $\L_{\min}{\oplus}{\square}p{\rightarrow}p$ is the intuitionistic modal logic determined by the class of all frames $(W,{\leq},{R})$ such that for all $s{\in}W$, $s{\leq}{\circ}{R}{\circ}{\leq}s$.
%
%
\begin{lemma}\label{lemma:about:what:happens:if:T:square:is:in:the:logic}
If ${\square}p{\rightarrow}p{\in}\L$ then $(W_{\L},{\leq_{\L}},{R_{\L}})$ is up-reflexive.
\end{lemma}
\begin{proof}
Suppose ${\square}p{\rightarrow}p{\in}\L$.
Let $\Gamma{\in}W_{\L}$.
We claim that $(\ast)$~${\square}\Gamma{\subseteq}\Gamma$.
If not, there exists a formula $A$ such that ${\square}A{\in}\Gamma$ and $A{\not\in}\Gamma$.
Since ${\square}p{\rightarrow}p{\in}\L$, then $A{\in}\Gamma$: a contradiction.
Let ${\mathcal S}{=}\{\Delta:\ \Delta$ is a $\L$-theory such that {\bf (1)}~$\Gamma{\subseteq}\Delta$ and {\bf (2)}~${\square}\Delta{\subseteq}\Gamma\}$.
Obviously, by~$(\ast)$, $\Gamma{\in}{\mathcal S}$.
Thus, ${\mathcal S}$ is nonempty.
Moreover, for all nonempty chains $(\Delta_{i})_{i{\in}I}$ of elements of ${\mathcal S}$, $\bigcup\{\Delta_{i}\ :\ i{\in}I\}$ is an element of ${\mathcal S}$.
Consequently, by Lemma~\ref{lemma:Zorn:lemma}, ${\mathcal S}$ possesses a maximal element $\Delta$.
Hence, $\Delta$ is a $\L$-theory such that $\Gamma{\subseteq}\Delta$ and ${\square}\Delta{\subseteq}\Gamma$.
We claim that $\Delta$ is proper.
If not, $\Delta{=}\Fo$.
Thus, ${\square}{\bot}{\in}\Delta$.
Since ${\square}\Delta{\subseteq}\Gamma$, then ${\bot}{\in}\Gamma$: a contradiction.
We claim that $\Delta$ is prime.
If not, there exists formulas $A,B$ such that $A{\vee}B{\in}\Delta$, $A{\not\in}\Delta$ and $B{\not\in}\Delta$.
Hence, by the maximality of $\Delta$ in ${\mathcal S}$, $\Delta{+}A{\not\in}{\mathcal S}$ and $\Delta{+}B{\not\in}{\mathcal S}$.
Thus, there exists formulas $C,D$ such that $C{\in}{\square}(\Delta{+}A)$, $C{\not\in}\Gamma$, $D{\in}{\square}(\Delta{+}B)$ and $D{\not\in}\Gamma$.
Hence, $A{\rightarrow}{\square}C{\in}\Delta$ and $B{\rightarrow}{\square}D{\in}\Delta$.
Consequently, using axioms and inference rules of $\IPL$, $A{\vee}B{\rightarrow}{\square}C{\vee}{\square}D{\in}\Delta$.
Since $A{\vee}B{\in}\Delta$, then ${\square}C{\vee}{\square}D{\in}\Delta$.
Hence, using axiom $(\Axiom1)$ and inference rule $(\Rule1)$, ${\square}(C{\vee}D){\in}\Delta$.
Since ${\square}\Delta{\subseteq}\Gamma$, then $C{\vee}D{\in}\Gamma$.
Consequently, either $C{\in}\Gamma$, or $D{\in}\Gamma$: a contradiction.
We claim that $(\ast\ast)$~for all formulas $C,D$, if $C{\vee}D{\in}{\square}\Delta$ then either ${\lozenge}C{\in}\Delta$, or $D{\in}\Gamma$.
If not, there exists formulas $C,D$ such that $C{\vee}D{\in}{\square}\Delta$, ${\lozenge}C{\not\in}\Delta$ and $D{\not\in}\Gamma$.
Hence, by the maximality of $\Delta$ in ${\mathcal S}$, $\Delta{+}{\lozenge}C{\not\in}{\mathcal S}$.
Thus, ${\square}(\Delta{+}{\lozenge}C){\not\subseteq}\Gamma$.
Consequently, there exists a formula $E$ such that $E{\in}{\square}(\Delta{+}{\lozenge}C)$ and $E{\not\in}\Gamma$.
Hence, ${\lozenge}C{\rightarrow}{\square}E{\in}\Delta$.
Consequently, using axiom $(\Axiom1)$ and inference rule $(\Rule1)$, ${\lozenge}C{\rightarrow}{\square}(D{\vee}E){\in}\Delta$.
Since $C{\vee}D{\in}{\square}\Delta$, then ${\square}(C{\vee}D){\in}\Delta$.
Hence, using axiom $(\Axiom1)$ and inference rule $(\Rule1)$, ${\square}(C{\vee}D{\vee}E){\in}\Delta$.
Since ${\lozenge}C{\rightarrow}{\square}(D{\vee}E){\in}\Delta$, then using axiom $(\Axiom2)$, ${\square}(D{\vee}E){\in}\Delta$.
Since ${\square}\Delta{\subseteq}\Gamma$, then $D{\vee}E{\in}\Gamma$.
Consequently, either $D{\in}\Gamma$, or $E{\in}\Gamma$: a contradiction.
Let ${\mathcal T}{=}\{\Lambda\ :\ \Lambda$ is a $\L$-theory such that {\bf (1)}~${\square}\Delta{\subseteq}\Lambda$ and {\bf (2)}~for all formulas $C,D$, if $C{\vee}D{\in}\Lambda$ then either ${\lozenge}C{\in}\Delta$, or $D{\in}\Gamma\}$.
Obviously, by~$(\ast\ast)$, ${\square}\Delta\in{\mathcal T}$.
Hence, ${\mathcal T}$ is nonempty.
Moreover, for all nonempty chains $(\Lambda_{i})_{i{\in}I}$ of elements of ${\mathcal T}$, $\bigcup\{\Lambda_{i}\ :\ i{\in}I\}$ is an element of ${\mathcal T}$.
Thus, by Lemma~\ref{lemma:Zorn:lemma}, ${\mathcal T}$ possesses a maximal element $\Lambda$.
Consequently, $\Lambda$ is a $\L$-theory such that ${\square}\Delta{\subseteq}\Lambda$ and for all formulas $C,D$, if $C{\vee}D{\in}\Lambda$ then either ${\lozenge}C{\in}\Delta$, or $D{\in}\Gamma$.
Hence, it only remains to be proved that $\Lambda$ is proper and prime, $\Delta{\bowtie}\Lambda$ and $\Lambda{\subseteq}\Gamma$.
We claim that $\Lambda$ is proper.
If not, $\Lambda{=}\Fo$.
Thus, ${\bot}{\vee}{\bot}{\in}\Lambda$.
Consequently, either ${\lozenge}{\bot}{\in}\Delta$, or ${\bot}{\in}\Gamma$: a contradiction.
We claim that $\Lambda$ is prime.
If not, there exists formulas $E,F$ such that $E{\vee}F{\in}\Lambda$, $E{\not\in}\Lambda$ and $F{\not\in}\Lambda$.
Hence, by the maximality of $\Lambda$ in ${\mathcal T}$, $\Lambda{+}E{\not\in}{\mathcal T}$ and $\Lambda{+}F{\not\in}{\mathcal T}$.
Thus, there exists formulas $G,H$ such that $G{\vee}H{\in}\Lambda{+}E$, ${\lozenge}G{\not\in}\Delta$ and $H{\not\in}\Gamma$ and there exists formulas $I,J$ such that $I{\vee}J{\in}\Lambda{+}F$, ${\lozenge}I{\not\in}\Delta$ and $J{\not\in}\Gamma$.
Consequently, $E{\rightarrow}G{\vee}H{\in}\Lambda$ and $F{\rightarrow}I{\vee}J{\in}\Lambda$.
Hence, using axioms and inference rules of $\IPL$, $E{\vee}F{\rightarrow}G{\vee}I{\vee}H{\vee}
$\linebreak$
J{\in}\Lambda$.
Since $E{\vee}F{\in}\Lambda$, then $G{\vee}I{\vee}H{\vee}J{\in}\Lambda$.
Thus, either ${\lozenge}(G{\vee}I){\in}\Delta$, or $H{\vee}J{\in}\Gamma$.
In the former case, using axiom $(\Axiom3)$, ${\lozenge}G{\vee}{\lozenge}I{\in}\Delta$.
Consequently, either ${\lozenge}G{\in}\Delta$, or ${\lozenge}I{\in}\Delta$: a contradiction.
In the latter case, either $H{\in}\Gamma$, or $J{\in}\Gamma$: a contradiction.
We claim that $\Delta{\bowtie}\Lambda$.
If not, there exists a formula $E$ such that $E{\in}\Lambda$ and ${\lozenge}E{\not\in}\Delta$.
Hence, $E{\vee}{\bot}{\in}\Lambda$.
Thus, either ${\lozenge}E{\in}\Delta$, or ${\bot}{\in}\Gamma$.
Since ${\lozenge}E{\not\in}\Delta$, then ${\bot}{\in}\Gamma$: a contradiction.
We claim that $\Lambda{\subseteq}\Gamma$.
If not, there exists a formula $E$ such that $E{\in}\Lambda$ and $E{\not\in}\Gamma$.
Consequently, ${\bot}{\vee}E{\in}\Lambda$.
Hence, either ${\lozenge}{\bot}{\in}\Delta$, or $E{\in}\Gamma$.
Since $E{\not\in}\Gamma$, then ${\lozenge}{\bot}{\in}\Delta$: a contradiction.
\medskip
\end{proof}
\begin{lemma}\label{lemma:about:what:happens:if:T:lozenge:is:in:the:logic}
If $p{\rightarrow}{\lozenge}p{\in}\L$ then $(W_{\L},{\leq_{\L}},{R_{\L}})$ is down-reflexive.
\end{lemma}
\begin{proof}
Suppose $p{\rightarrow}{\lozenge}p{\in}\L$.
Let $\Gamma{\in}W_{\L}$.
We claim that $(\ast)$~${\lozenge}\Gamma{\subseteq}\Gamma$.
If not, there exists a formula $A$ such that $A{\in}\Gamma$ and ${\lozenge}A{\not\in}\Gamma$.
Since $p{\rightarrow}{\lozenge}p{\in}\L$, then ${\lozenge}A{\in}\Gamma$: a contradiction.
Let ${\mathcal S}{=}\{\Lambda:\ \Lambda$ is a $\L$-theory such that {\bf (1)}~$\Gamma{\subseteq}\Lambda$ and {\bf (2)}~${\lozenge}\Lambda{\subseteq}\Gamma\}$.
Obviously, by~$(\ast)$, $\Gamma{\in}{\mathcal S}$.
Thus, ${\mathcal S}$ is nonempty.
Moreover, for all nonempty chains $(\Lambda_{i})_{i{\in}I}$ of elements of ${\mathcal S}$, $\bigcup\{\Lambda_{i}\ :\ i{\in}I\}$ is an element of ${\mathcal S}$.
Consequently, by Lemma~\ref{lemma:Zorn:lemma}, ${\mathcal S}$ possesses a maximal element $\Lambda$.
Hence, $\Lambda$ is a $\L$-theory such that $\Gamma{\subseteq}\Lambda$ and ${\lozenge}\Lambda{\subseteq}\Gamma$.
We claim that $\Lambda$ is proper.
If not, ${\bot}{\in}\Lambda$.
Since ${\lozenge}\Lambda{\subseteq}\Gamma$, then ${\lozenge}{\bot}{\in}\Gamma$: a contradiction.
We claim that $\Lambda$ is prime.
If not, there exists formulas $A,B$ such that $A{\vee}B{\in}\Lambda$, $A{\not\in}\Lambda$ and $B{\not\in}\Lambda$.
Consequently, by the maximality of $\Lambda$ in ${\mathcal S}$, $\Lambda{+}A{\not\in}{\mathcal S}$ and $\Lambda{+}B{\not\in}{\mathcal S}$.
Hence, ${\lozenge}(\Lambda{+}A){\not\subseteq}\Gamma$ and ${\lozenge}(\Lambda{+}B){\not\subseteq}\Gamma$.
Thus, there exists formulas $C,D$ such that $C{\in}\Lambda{+}A$, ${\lozenge}C{\not\in}\Gamma$, $D{\in}\Lambda{+}B$ and ${\lozenge}D{\not\in}\Gamma$.
Consequently, $A{\rightarrow}C{\in}\Lambda$ and $B{\rightarrow}D{\in}\Lambda$.
Hence, using axioms and inference rules of $\IPL$, $A{\vee}B{\rightarrow}C{\vee}D{\in}\Lambda$.
Since $A{\vee}B{\in}\Lambda$, then $C{\vee}D{\in}\Lambda$.
Since ${\lozenge}\Lambda{\subseteq}\Gamma$, then ${\lozenge}(C{\vee}D){\in}\Gamma$.
Thus, using axiom $(\Axiom3)$, ${\lozenge}C{\vee}{\lozenge}D{\in}
$\linebreak$
\Gamma$.
Consequently, either ${\lozenge}C{\in}\Gamma$, or ${\lozenge}D{\in}\Gamma$: a contradiction.
Let $\Sigma{=}\{A{\in}\Fo:$ there exists $B,C,D{\in}\Fo$ such that $B{\rightarrow}C{\in}\Lambda$, ${\square}B{\vee}D{\in}\L{+}{\lozenge}\Lambda$, ${\lozenge}C{\rightarrow}A{\in}\L{+}{\lozenge}\Lambda$ and $D{\not\in}\Gamma\}$.
We claim that $(\ast\ast)$~for all $E,F{\in}\Fo$, if ${\square}E{\vee}F{\in}\L{+}(\Sigma{\cup}{\lozenge}\Lambda)$ then either $E{\in}\Lambda$, or $F{\in}\Gamma$.
If not, there exists $E,F{\in}\Fo$ such that ${\square}E{\vee}F{\in}\L{+}(\Sigma{\cup}{\lozenge}\Lambda)$, $E{\not\in}\Lambda$ and $F{\not\in}\Gamma$.
Hence, there exists $m,n{\in}\N$, there exists $A_{1},\ldots,A_{m}{\in}\Sigma$ and there exists $G_{1},\ldots,G_{n}{\in}
$\linebreak$
\Lambda$ such that $A_{1}{\wedge}\ldots{\wedge}A_{m}{\wedge}{\lozenge}G_{1}{\wedge}\ldots{\wedge}{\lozenge}G_{n}{\rightarrow}{\square}E{\vee}F{\in}\L$.
Thus, for all $i{\in}\{1,\ldots,m\}$, there exists $B_{i},C_{i},D_{i}{\in}\Fo$ such that $B_{i}{\rightarrow}C_{i}{\in}\Lambda$, ${\square}B_{i}{\vee}D_{i}{\in}\L{+}{\lozenge}\Lambda$, ${\lozenge}C_{i}{\rightarrow}A_{i}{\in}\L{+}{\lozenge}\Lambda$ and $D_{i}{\not\in}\Gamma$.
Let $\bar{A}{=}A_{1}{\wedge}\ldots{\wedge}A_{m}$ and $\bar{G}{=}G_{1}{\wedge}\ldots{\wedge}G_{m}$.
Since $G_{1},\ldots,G_{n}{\in}\Lambda$, then using axioms and inference rules of $\IPL$, $\bar{G}{\in}\Lambda$.
Since $A_{1}{\wedge}\ldots{\wedge}A_{m}{\wedge}{\lozenge}G_{1}{\wedge}\ldots{\wedge}
$\linebreak$
{\lozenge}G_{n}{\rightarrow}{\square}E{\vee}F{\in}\L$, then using inference rule $(\Rule2)$, $\bar{A}{\wedge}{\lozenge}\bar{G}{\rightarrow}{\square}E{\vee}F{\in}\L$.
Let $\bar{B}{=}B_{1}{\wedge}
$\linebreak$
\ldots{\wedge}B_{m}$, $\bar{C}{=}C_{1}{\wedge}\ldots{\wedge}C_{m}$ and $\bar{D}{=}D_{1}{\vee}\ldots{\vee}D_{m}$.
Since for all $i{\in}\{1,\ldots,m\}$, $B_{i}{\rightarrow}
$\linebreak$
C_{i}{\in}\Lambda$, then $\bar{B}{\rightarrow}\bar{C}{\in}\Lambda$.
Since for all $i{\in}\{1,\ldots,m\}$, ${\square}B_{i}{\vee}D_{i}{\in}\L{+}{\lozenge}\Lambda$, then using axiom $(\Axiom1)$ and inference rule $(\Rule1)$, ${\square}\bar{B}{\vee}\bar{D}{\in}\L{+}{\lozenge}\Lambda$.
Since for all $i{\in}\{1,\ldots,m\}$, ${\lozenge}C_{i}{\rightarrow}A_{i}{\in}\L{+}{\lozenge}\Lambda$, then using inference rule $(\Rule2)$, ${\lozenge}\bar{C}{\rightarrow}\bar{A}{\in}\L{+}{\lozenge}\Lambda$.
Since for all $i{\in}\{1,\ldots,m\}$, $D_{i}{\not\in}\Gamma$, then using axioms and inference rules of $\IPL$, $\bar{D}{\not\in}\Gamma$.
Since ${\square}\bar{B}{\vee}\bar{D}{\in}\L{+}{\lozenge}\Lambda$ and ${\lozenge}\bar{C}{\rightarrow}\bar{A}{\in}\L{+}{\lozenge}\Lambda$, then there exists $o,p{\in}\N$, there exists $H_{1},\ldots,
$\linebreak$
H_{o}{\in}\Lambda$ and there exists $I_{1},\ldots,I_{p}{\in}\Lambda$ such that ${\lozenge}H_{1}{\wedge}\ldots{\wedge}{\lozenge}H_{o}{\rightarrow}{\square}\bar{B}{\vee}\bar{D}{\in}\L$ and ${\lozenge}I_{1}
$\linebreak$
{\wedge}\ldots{\wedge}{\lozenge}I_{p}{\wedge}{\lozenge}\bar{C}{\rightarrow}\bar{A}{\in}\L$.
Let $\bar{H}{=}H_{1}{\wedge}\ldots{\wedge}H_{o}$ and $\bar{I}{=}I_{1}{\wedge}\ldots{\wedge}I_{p}$.
Since $H_{1},\ldots,H_{o}
$\linebreak$
{\in}\Lambda$ and $I_{1},\ldots,I_{p}{\in}\Lambda$, then using axioms and inference rules of $\IPL$, $\bar{H}{\in}\Lambda$ and $\bar{I}{\in}\Lambda$.
Since ${\lozenge}H_{1}{\wedge}\ldots{\wedge}{\lozenge}H_{o}{\rightarrow}{\square}\bar{B}{\vee}\bar{D}{\in}\L$ and ${\lozenge}I_{1}{\wedge}\ldots{\wedge}{\lozenge}I_{p}{\wedge}{\lozenge}\bar{C}{\rightarrow}\bar{A}{\in}\L$, then using inference rule $(\Rule2)$, ${\lozenge}\bar{H}{\rightarrow}{\square}\bar{B}{\vee}\bar{D}{\in}\L$ and ${\lozenge}\bar{I}{\wedge}{\lozenge}\bar{C}{\rightarrow}\bar{A}{\in}\L$.
Since $E{\not\in}\Lambda$, then by the maximality of $\Lambda$ in ${\mathcal S}$, $\Lambda{+}E{\not\in}{\mathcal S}$.
Consequently, ${\lozenge}(\Lambda{+}E){\not\subseteq}\Gamma$.
Hence, there exists $J{\in}\Fo$ such that $J{\in}\Lambda{+}E$ and ${\lozenge}J{\not\in}\Gamma$.
Thus, $E{\rightarrow}J{\in}\Lambda$.
Since $\bar{G}{\in}\Lambda$, $\bar{B}{\rightarrow}\bar{C}{\in}\Lambda$, $\bar{H}{\in}\Lambda$ and $\bar{I}{\in}\Lambda$, then $\bar{G}{\wedge}(\bar{B}{\rightarrow}\bar{C}){\wedge}\bar{H}{\wedge}\bar{I}{\wedge}(E{\rightarrow}J){\in}\Lambda$.
Since ${\lozenge}\Lambda{\subseteq}\Gamma$, then ${\lozenge}(\bar{G}{\wedge}(\bar{B}{\rightarrow}\bar{C}){\wedge}\bar{H}{\wedge}\bar{I}{\wedge}
$\linebreak$
(E{\rightarrow}J)){\in}\Gamma$.
Since $\bar{A}{\wedge}{\lozenge}\bar{G}{\rightarrow}{\square}E{\vee}F{\in}\L$, ${\lozenge}\bar{H}{\rightarrow}{\square}\bar{B}{\vee}\bar{D}{\in}\L$ and ${\lozenge}\bar{I}{\wedge}{\lozenge}\bar{C}{\rightarrow}\bar{A}{\in}\L$, then by Lemma~\ref{first:lemma:about:the:use:of:the:special:inference:rule}, ${\lozenge}(\bar{G}{\wedge}(\bar{B}{\rightarrow}\bar{C}){\wedge}\bar{H}{\wedge}\bar{I}{\wedge}(E{\rightarrow}J)){\rightarrow}F{\vee}\bar{D}{\vee}{\lozenge}J{\in}\L$.
Since ${\lozenge}(\bar{G}{\wedge}(\bar{B}{\rightarrow}\bar{C})
$\linebreak$
{\wedge}\bar{H}{\wedge}\bar{I}{\wedge}(E{\rightarrow}J)){\in}\Gamma$, then $F{\vee}\bar{D}{\vee}{\lozenge}J{\in}\Gamma$.
Consequently, either $F{\in}\Gamma$, or $\bar{D}{\in}\Gamma$, or ${\lozenge}J{\in}\Gamma$: a contradiction.
Let ${\mathcal T}{=}\{\Delta\ :\ \Delta$ is a $\L$-theory such that {\bf (1)}~${\lozenge}\Lambda{\subseteq}\Delta$ and {\bf (2)}~for all formulas $E,F$, if ${\square}E{\vee}F{\in}\Delta$ then either $E{\in}\Lambda$, or $F{\in}\Gamma\}$.
Obviously, by~$(\ast\ast)$, $\L{+}(\Sigma{\cup}{\lozenge}\Lambda){\in}{\mathcal T}$.
Hence, ${\mathcal T}$ is nonempty.
Moreover, for all nonempty chains $(\Delta_{i})_{i{\in}I}$ of elements of ${\mathcal T}$, $\bigcup\{\Delta_{i}\ :\ i{\in}I\}$ is an element of ${\mathcal T}$.
Consequently, by Lemma~\ref{lemma:Zorn:lemma}, ${\mathcal T}$ possesses a maximal element $\Delta$.
Hence, $\Delta$ is a $\L$-theory such that ${\lozenge}\Lambda{\subseteq}\Delta$ and for all formulas $E,F$, if ${\square}E{\vee}F{\in}\Delta$ then either $E{\in}\Lambda$, or $F{\in}\Gamma$.
Thus, it only remains to be proved that $\Delta$ is proper and prime, $\Delta{\subseteq}\Gamma$ and $\Delta{\bowtie}\Lambda$.
We claim that $\Delta$ is proper.
If not, $\Delta{=}\Fo$.
Consequently, ${\square}{\bot}{\vee}{\bot}{\in}\Delta$.
Hence, either ${\bot}{\in}\Lambda$, or ${\bot}{\in}\Gamma$: a contradiction.
We claim that $\Delta$ is prime.
If not, there exists formulas $A,B$ such that $A{\vee}B{\in}\Delta$, $A{\not\in}\Delta$ and $B{\not\in}\Delta$.
Thus, by the maximality of $\Delta$ in ${\mathcal T}$, $\Delta{+}A{\not\in}{\mathcal T}$ and $\Delta{+}B{\not\in}{\mathcal T}$.
Consequently, there exists formulas $E,F$ such that ${\square}E{\vee}F{\in}\Delta{+}A$, $E{\not\in}\Lambda$ and $F{\not\in}\Gamma$ and there exists formulas $G,H$ such that ${\square}G{\vee}H{\in}\Delta{+}B$, $G{\not\in}\Lambda$ and $H{\not\in}\Gamma$.
Hence, $A{\rightarrow}{\square}E{\vee}F{\in}
$\linebreak$
\Delta$ and $B{\rightarrow}{\square}G{\vee}H{\in}\Delta$.
Thus, using axioms and inference rules of $\IPL$, $A{\vee}B{\rightarrow}{\square}E{\vee}
$\linebreak$
{\square}G{\vee}F{\vee}H{\in}\Delta$.
Since $A{\vee}B{\in}\Theta$, then ${\square}E{\vee}{\square}G{\vee}F{\vee}H{\in}\Delta$.
Consequently, using axiom $(\Axiom1)$ and inference rule $(\Rule1)$, ${\square}(E{\vee}G){\vee}F{\vee}H{\in}\Delta$.
Hence, either $E{\vee}G{\in}\Lambda$, or $F{\vee}H{\in}\Gamma$.
In the former case, either $E{\in}\Lambda$, or $G{\in}\Lambda$: a contradiction.
In the latter case, either $F{\in}\Gamma$, or $H{\in}\Gamma$: a contradiction.
We claim that $\Delta{\subseteq}\Gamma$.
If not, there exists a formula $A$ such that $A{\in}\Delta$ and $A{\not\in}\Gamma$.
Thus, ${\square}{\bot}{\vee}A{\in}\Delta$.
Consequently, either ${\bot}{\in}\Lambda$, or $A{\in}\Gamma$.
Since $A{\not\in}\Gamma$, then ${\bot}{\in}\Lambda$: a contradiction.
We claim that $\Delta{\bowtie}\Lambda$.
If not, there exists a formula $A$ such that ${\square}A{\in}\Delta$ and $A{\not\in}\Lambda$.
Hence, ${\square}A{\vee}{\bot}{\in}\Delta$.
Thus, either $A{\in}\Lambda$, or ${\bot}{\in}\Gamma$.
Since $A{\not\in}\Lambda$, then ${\bot}{\in}\Gamma$: a contradiction.
\medskip
\end{proof}
%
%
%Lemma~\ref{lemma:about:what:happens:if:B:square:is:in:the:logic} is used in Section~\ref{section:soundness:and:completeness} when we show that $\L_{\min}{\oplus}{\lozenge}{\square}p{\rightarrow}p$ is the intuitionistic modal logic determined by the class of all frames $(W,{\leq},{R})$ such that for all $s,t{\in}W$, if $s{R}t$ then $t{\leq}{\circ}{R}{\circ}{\leq}s$.
%
%
\begin{lemma}\label{lemma:about:what:happens:if:B:square:is:in:the:logic}
If ${\lozenge}{\square}p{\rightarrow}p{\in}\L$ then $(W_{\L},{\leq_{\L}},{R_{\L}})$ is up-symmetric.
\end{lemma}
\begin{proof}
Suppose ${\lozenge}{\square}p{\rightarrow}p{\in}\L$.
Let $\Gamma,\Delta{\in}W_{\L}$ be such that $\Gamma{R_{\L}}\Delta$.
Hence, $\Gamma{\bowtie}\Delta$.
We claim that $(\ast)$~${\square}\Delta{\subseteq}\Gamma$.
If not, there exists a formula $A$ such that $A{\in}{\square}\Delta$ and $A{\not\in}\Gamma$.
Thus, ${\square}A{\in}\Delta$.
Since $\Gamma{\bowtie}\Delta$, then ${\lozenge}{\square}A{\in}\Gamma$.
Since ${\lozenge}{\square}p{\rightarrow}p{\in}\L$, then $A{\in}\Gamma$: a contradiction.
Let ${\mathcal S}{=}\{\Lambda:\ \Lambda$ is a $\L$-theory such that {\bf (1)}~$\Delta{\subseteq}\Lambda$ and {\bf (2)}~${\square}\Lambda{\subseteq}\Gamma\}$.
Obviously, by~$(\ast)$, $\Delta{\in}{\mathcal S}$.
Consequently, ${\mathcal S}$ is nonempty.
Moreover, for all nonempty chains $(\Lambda_{i})_{i{\in}I}$ of elements of ${\mathcal S}$, $\bigcup\{\Lambda_{i}\ :\ i{\in}I\}$ is an element of ${\mathcal S}$.
Hence, by Lemma~\ref{lemma:Zorn:lemma}, ${\mathcal S}$ possesses a maximal element $\Lambda$.
Thus, $\Lambda$ is a $\L$-theory such that $\Delta{\subseteq}\Lambda$ and ${\square}\Lambda{\subseteq}\Gamma$.
We claim that $\Lambda$ is proper.
If not, $\Lambda{=}\Fo$.
Consequently, ${\square}{\bot}{\in}\Lambda$.
Since ${\square}\Lambda{\subseteq}\Gamma$, then ${\bot}{\in}\Gamma$: a contradiction.
We claim that $\Lambda$ is prime.
If not, there exists formulas $A,B$ such that $A{\vee}B{\in}\Lambda$, $A{\not\in}\Lambda$ and $B{\not\in}\Lambda$.
Thus, by the maximality of $\Lambda$ in ${\mathcal S}$, $\Lambda{+}A{\not\in}{\mathcal S}$ and $\Lambda{+}B{\not\in}{\mathcal S}$.
Consequently, ${\square}(\Lambda{+}A){\not\subseteq}\Gamma$ and ${\square}(\Lambda{+}B){\not\subseteq}\Gamma$.
Hence, there exists formulas $C,D$ such that $C{\in}{\square}(\Lambda{+}A)$, $C{\not\in}\Gamma$, $D{\in}{\square}(\Lambda{+}B)$ and $D{\not\in}\Gamma$.
Consequently, $A{\rightarrow}{\square}C{\in}\Lambda$ and $B{\rightarrow}{\square}D
$\linebreak$
{\in}\Lambda$.
Hence, using axioms and inference rules of $\IPL$, $A{\vee}B{\rightarrow}{\square}C{\vee}{\square}D{\in}\Lambda$.
Since $A{\vee}B{\in}\Lambda$, then ${\square}C{\vee}{\square}D{\in}\Lambda$.
Thus, using axiom $(\Axiom1)$ and inference rule $(\Rule1)$, ${\square}(C{\vee}
$\linebreak$
D){\in}\Lambda$.
Since ${\square}\Lambda{\subseteq}\Gamma$, then $C{\vee}D{\in}\Gamma$.
Hence, either $C{\in}\Gamma$, or $D{\in}\Gamma$: a contradiction.
We claim that $(\ast\ast)$~for all formulas $C,D$, if $C{\vee}D{\in}{\square}\Lambda$ then either ${\lozenge}C{\in}\Lambda$, or $D{\in}\Gamma$.
If not, there exists formulas $C,D$ such that $C{\vee}D{\in}{\square}\Lambda$, ${\lozenge}C{\not\in}\Lambda$ and $D{\not\in}\Gamma$.
Thus, by the maximality of $\Lambda$ in ${\mathcal S}$, $\Lambda{+}{\lozenge}C{\not\in}{\mathcal S}$.
Consequently, ${\square}(\Lambda{+}{\lozenge}C){\not\subseteq}\Gamma$.
Hence, there exists a formula $E$ such that $E{\in}{\square}(\Lambda{+}{\lozenge}C)$ and $E{\not\in}\Gamma$.
%Thus, ${\square}E{\in}\Lambda{+}{\lozenge}C$.
Consequently, ${\lozenge}C{\rightarrow}{\square}E{\in}\Lambda$.
Hence, using axiom $(\Axiom1)$ and inference rule $(\Rule1)$, ${\lozenge}C{\rightarrow}{\square}(D{\vee}E){\in}\Lambda$.
Since $C{\vee}D{\in}{\square}\Lambda$, then ${\square}(C{\vee}D){\in}\Lambda$.
Thus, using axiom $(\Axiom1)$ and inference rule $(\Rule1)$, ${\square}(C{\vee}D{\vee}E){\in}\Lambda$.
Since ${\lozenge}C{\rightarrow}{\square}(D{\vee}E){\in}\Lambda$, then using axiom $(\Axiom2)$, ${\square}(D{\vee}E){\in}\Lambda$.
%Consequently, $D{\vee}E{\in}{\square}\Lambda$.
Since ${\square}\Lambda{\subseteq}\Gamma$, then $D{\vee}E{\in}\Gamma$.
Hence, either $D{\in}\Gamma$, or $E{\in}\Gamma$: a contradiction.
Let ${\mathcal T}{=}\{\Theta\ :\ \Theta$ is a $\L$-theory such that {\bf (1)}~${\square}\Lambda{\subseteq}\Theta$ and {\bf (2)}~for all formulas $C,D$, if $C{\vee}D{\in}\Theta$ then either ${\lozenge}C{\in}\Lambda$, or $D{\in}\Gamma\}$.
Obviously, by~$(\ast\ast)$, ${\square}\Lambda\in{\mathcal T}$.
Thus, ${\mathcal T}$ is nonempty.
Moreover, for all nonempty chains $(\Theta_{i})_{i{\in}I}$ of elements of ${\mathcal T}$, $\bigcup\{\Theta_{i}\ :\ i{\in}I\}$ is an element of ${\mathcal T}$.
Consequently, by Lemma~\ref{lemma:Zorn:lemma}, ${\mathcal T}$ possesses a maximal element $\Theta$.
Hence, $\Theta$ is a $\L$-theory such that ${\square}\Lambda{\subseteq}\Theta$ and for all formulas $C,D$, if $C{\vee}D{\in}\Theta$ then either ${\lozenge}C{\in}\Lambda$, or $D{\in}\Gamma$.
Thus, it only remains to be proved that $\Theta$ is proper and prime, $\Lambda{\bowtie}\Theta$ and $\Theta{\subseteq}\Gamma$.
We claim that $\Theta$ is proper.
If not, $\Theta{=}\Fo$.
Consequently, ${\bot}{\vee}{\bot}{\in}\Theta$.
Hence, either ${\lozenge}{\bot}{\in}\Lambda$, or ${\bot}{\in}\Gamma$: a contradiction.
We claim that $\Theta$ is prime.
If not, there exists formulas $E,F$ such that $E{\vee}F{\in}\Theta$, $E{\not\in}\Theta$ and $F{\not\in}\Theta$.
Thus, by the maximality of $\Theta$ in ${\mathcal T}$, $\Theta{+}E{\not\in}{\mathcal T}$ and $\Theta{+}F{\not\in}{\mathcal T}$.
Consequently, there exists formulas $G,H$ such that $G{\vee}H{\in}\Theta{+}E$, ${\lozenge}G{\not\in}\Lambda$ and $H{\not\in}\Gamma$ and there exists formulas $I,J$ such that $I{\vee}J{\in}\Theta{+}F$, ${\lozenge}I{\not\in}\Lambda$ and $J{\not\in}\Gamma$.
Hence, $E{\rightarrow}G{\vee}H{\in}\Theta$ and $F{\rightarrow}I{\vee}J{\in}\Theta$.
Thus, using axioms and inference rules of $\IPL$, $E{\vee}F{\rightarrow}G{\vee}I{\vee}H{\vee}J{\in}\Theta$.
Since $E{\vee}F{\in}\Theta$, then $G{\vee}I{\vee}H{\vee}J{\in}\Theta$.
Consequently, either ${\lozenge}(G{\vee}I){\in}\Lambda$, or $H{\vee}J{\in}\Gamma$.
In the former case, using axiom $(\Axiom3)$, ${\lozenge}G{\vee}{\lozenge}I{\in}\Lambda$.
Hence, either ${\lozenge}G{\in}\Lambda$, or ${\lozenge}I{\in}\Lambda$: a contradiction.
In the latter case, either $H{\in}\Gamma$, or $J{\in}\Gamma$: a contradiction.
We claim that $\Lambda{\bowtie}\Theta$.
If not, there exists a formula $E$ such that $E{\in}\Theta$ and ${\lozenge}E{\not\in}\Lambda$.
Thus, $E{\vee}{\bot}{\in}\Theta$.
Consequently, either ${\lozenge}E{\in}\Lambda$, or ${\bot}{\in}\Gamma$.
Since ${\lozenge}E{\not\in}\Lambda$, then ${\bot}{\in}\Gamma$: a contradiction.
We claim that $\Theta{\subseteq}\Gamma$.
If not, there exists a formula $E$ such that $E{\in}\Theta$ and $E{\not\in}\Gamma$.
Hence, ${\bot}{\vee}E{\in}\Theta$.
Thus, either ${\lozenge}{\bot}{\in}\Lambda$, or $E{\in}\Gamma$.
Since $E{\not\in}\Gamma$, then ${\lozenge}{\bot}{\in}\Lambda$: a contradiction.
\medskip
\end{proof}
\begin{lemma}\label{lemma:about:what:happens:if:B:lozenge:is:in:the:logic}
If $p{\rightarrow}{\square}{\lozenge}p{\in}\L$ then $(W_{\L},{\leq_{\L}},{R_{\L}})$ is down-symmetric.
\end{lemma}
\begin{proof}
Suppose $p{\rightarrow}{\square}{\lozenge}p{\in}\L$.
Let $\Gamma,\Delta{\in}W_{\L}$ be such that $\Gamma{R_{\L}}\Delta$.
Hence, $\Gamma{\bowtie}\Delta$.
We claim that $(\ast)$~${\lozenge}\Gamma{\subseteq}\Delta$.
If not, there exists a formula $A$ such that $A{\in}\Gamma$ and ${\lozenge}A{\not\in}\Delta$.
Since $\Gamma{\bowtie}\Delta$, then ${\square}{\lozenge}A{\not\in}\Gamma$.
Since $p{\rightarrow}{\square}{\lozenge}p{\in}\L$, then $A{\not\in}\Gamma$: a contradiction.
Let ${\mathcal S}{=}\{\Theta:\ \Theta$ is a $\L$-theory such that {\bf (1)}~$\Theta{\supseteq}\Gamma$ and {\bf (2)}~${\lozenge}\Theta{\subseteq}\Delta\}$.
Obviously, by~$(\ast)$, $\Gamma{\in}{\mathcal S}$.
Consequently, ${\mathcal S}$ is nonempty.
Moreover, for all nonempty chains $(\Theta_{i})_{i{\in}I}$ of elements of ${\mathcal S}$, $\bigcup\{\Theta_{i}\ :\ i{\in}I\}$ is an element of ${\mathcal S}$.
Hence, by Lemma~\ref{lemma:Zorn:lemma}, ${\mathcal S}$ possesses a maximal element $\Theta$.
Thus, $\Theta$ is a $\L$-theory such that $\Theta{\supseteq}\Gamma$ and ${\lozenge}\Theta{\subseteq}\Delta$.
We claim that $\Theta$ is proper.
If not, ${\bot}{\in}\Theta$.
Since ${\lozenge}\Theta{\subseteq}\Delta$, then ${\lozenge}{\bot}{\in}\Delta$: a contradiction.
We claim that $\Theta$ is prime.
If not, there exists formulas $A,B$ such that $A{\vee}B{\in}\Theta$, $A{\not\in}\Theta$ and $B{\not\in}\Theta$.
Hence, by the maximality of $\Theta$ in ${\mathcal S}$, $\Theta{+}A{\not\in}{\mathcal S}$ and $\Theta{+}B{\not\in}{\mathcal S}$.
Consequently, ${\lozenge}(\Theta{+}A){\not\subseteq}\Delta$ and ${\lozenge}(\Theta{+}B){\not\subseteq}\Delta$.
Thus, there exists formulas $C,D$ such that $C{\in}(\Theta{+}A)$, ${\lozenge}C{\not\in}\Delta$, $D{\in}(\Theta{+}B)$ and ${\lozenge}D{\not\in}\Delta$.
Consequently, $A{\rightarrow}C{\in}\Theta$ and $B{\rightarrow}D{\in}\Theta$.
Hence, using axioms and inference rules of $\IPL$, $A{\vee}B{\rightarrow}C{\vee}D{\in}\Theta$.
Since $A{\vee}B{\in}\Theta$, then $C{\vee}D{\in}\Theta$.
Since ${\lozenge}\Theta{\subseteq}\Delta$, then ${\lozenge}(C{\vee}D){\in}\Delta$.
Thus, using axiom $(\Axiom3)$, ${\lozenge}C{\vee}{\lozenge}D{\in}
$\linebreak$
\Delta$.
Consequently, either ${\lozenge}C{\in}\Delta$, or ${\lozenge}D{\in}\Delta$: a contradiction.
Let $\Sigma{=}\{A{\in}\Fo:$ there exists $B,C,D{\in}\Fo$ such that ${\lozenge}B{\rightarrow}A{\in}\L{+}{\lozenge}\Theta$, ${\square}C{\vee}D{\in}\L
$\linebreak$
{+}{\lozenge}\Theta$, $C{\rightarrow}B{\in}\Theta$ and $D{\not\in}\Delta\}$.
We claim that $(\ast\ast)$~for all $E,F{\in}\Fo$, if ${\square}E{\vee}F{\in}\L{+}(\Sigma{\cup}{\lozenge}\Theta)$ then either $E{\in}\Theta$, or $F{\in}\Delta$.
If not, there exists $E,F{\in}\Fo$ such that ${\square}E{\vee}F{\in}\L{+}(\Sigma{\cup}{\lozenge}\Theta)$, $E{\not\in}\Theta$ and $F{\not\in}\Delta$.
Hence, there exists $m,n{\in}\N$, there exists $A_{1},\ldots,A_{m}{\in}\Sigma$ and there exists $G_{1},\ldots,G_{n}{\in}
$\linebreak$
\Theta$ such that $A_{1}{\wedge}\ldots{\wedge}A_{m}{\wedge}{\lozenge}G_{1}{\wedge}\ldots{\wedge}{\lozenge}G_{n}{\rightarrow}{\square}E{\vee}F{\in}\L$.
Thus, for all $i{\in}\{1,\ldots,m\}$, there exists $B_{i},C_{i},D_{i}{\in}\Fo$ such that ${\lozenge}B_{i}{\rightarrow}A_{i}{\in}\L{+}{\lozenge}\Theta$, ${\square}C_{i}{\vee}D_{i}{\in}\L{+}{\lozenge}\Theta$, $C_{i}{\rightarrow}B_{i}{\in}
$\linebreak$
\Theta$ and $D_{i}{\not\in}\Delta$.
Let $\bar{A}{=}A_{1}{\wedge}\ldots{\wedge}A_{m}$ and $\bar{G}{=}G_{1}{\wedge}\ldots{\wedge}G_{m}$.
Since $G_{1},\ldots,G_{n}{\in}\Theta$, then using axioms and inference rules of $\IPL$, $\bar{G}{\in}\Theta$.
Since $A_{1}{\wedge}\ldots{\wedge}A_{m}{\wedge}{\lozenge}G_{1}{\wedge}\ldots
$\linebreak$
{\wedge}{\lozenge}G_{n}{\rightarrow}{\square}E{\vee}F{\in}\L$, then using inference rule $(\Rule2)$, $\bar{A}{\wedge}{\lozenge}\bar{G}{\rightarrow}{\square}E{\vee}F{\in}\L$.
Let $\bar{B}{=}B_{1}
$\linebreak$
{\wedge}\ldots{\wedge}B_{m}$, $\bar{C}{=}C_{1}{\wedge}\ldots{\wedge}C_{m}$ and $\bar{D}{=}D_{1}{\vee}\ldots{\vee}D_{m}$.
Since for all $i{\in}\{1,\ldots,m\}$, ${\lozenge}B_{i}
$\linebreak$
{\rightarrow}A_{i}{\in}\L{+}{\lozenge}\Theta$, then using inference rule $(\Rule2)$, ${\lozenge}\bar{B}{\rightarrow}\bar{A}{\in}\L{+}{\lozenge}\Theta$.
Since for all $i{\in}\{1,
$\linebreak$
\ldots,m\}$, ${\square}C_{i}{\vee}D_{i}{\in}\L{+}{\lozenge}\Theta$, then using axiom $(\Axiom1)$ and inference rule $(\Rule1)$, ${\square}\bar{C}{\vee}\bar{D}{\in}
$\linebreak$
\L{+}{\lozenge}\Theta$.
Since for all $i{\in}\{1,\ldots,m\}$, $C_{i}{\rightarrow}B_{i}{\in}\Theta$, then $\bar{C}{\rightarrow}\bar{B}{\in}\Theta$.
Since for all $i{\in}\{1,
$\linebreak$
\ldots,m\}$, $D_{i}{\not\in}\Delta$, then using axioms and inference rules of $\IPL$, $\bar{D}{\not\in}\Delta$.
Since ${\lozenge}\bar{B}{\rightarrow}\bar{A}{\in}
$\linebreak$
\L{+}{\lozenge}\Theta$ and ${\square}\bar{C}{\vee}\bar{D}{\in}\L{+}{\lozenge}\Theta$, then there exists $o,p{\in}\N$, there exists $H_{1},\ldots,H_{o}{\in}\Theta$ and there exists $I_{1},\ldots,I_{p}{\in}\Theta$ such that ${\lozenge}H_{1}{\wedge}\ldots{\wedge}{\lozenge}H_{o}{\wedge}{\lozenge}\bar{B}{\rightarrow}\bar{A}{\in}\L$ and ${\lozenge}I_{1}{\wedge}\ldots{\wedge}{\lozenge}I_{p}
$\linebreak$
{\rightarrow}{\square}\bar{C}{\vee}\bar{D}{\in}\L$.
Let $\bar{H}{=}H_{1}{\wedge}\ldots{\wedge}H_{o}$ and $\bar{I}{=}I_{1}{\wedge}\ldots{\wedge}I_{p}$.
Since $H_{1},\ldots,H_{o}{\in}\Theta$ and $I_{1},\ldots,I_{p}{\in}\Theta$, then using axioms and inference rules of $\IPL$, $\bar{H}{\in}\Theta$ and $\bar{I}{\in}\Theta$.
Since ${\lozenge}H_{1}{\wedge}\ldots{\wedge}{\lozenge}H_{o}{\wedge}{\lozenge}\bar{B}{\rightarrow}\bar{A}{\in}\L$ and ${\lozenge}I_{1}{\wedge}\ldots{\wedge}{\lozenge}I_{p}{\rightarrow}{\square}\bar{C}{\vee}\bar{D}{\in}\L$, then using inference rule $(\Rule2)$, ${\lozenge}\bar{H}{\wedge}{\lozenge}\bar{B}{\rightarrow}\bar{A}{\in}\L$ and ${\lozenge}\bar{I}{\rightarrow}{\square}\bar{C}{\vee}\bar{D}{\in}\L$.
Since $E{\not\in}\Theta$, then by the maximality of $\Theta$ in ${\mathcal S}$, $\Theta{+}E{\not\in}{\mathcal S}$.
Consequently, ${\lozenge}(\Theta{+}E){\not\subseteq}\Delta$.
Hence, there exists $J{\in}\Fo$ such that $J{\in}\Theta{+}E$ and ${\lozenge}J{\not\in}\Delta$.
Thus, $E{\rightarrow}J{\in}\Theta$.
Since $\bar{G}{\in}\Theta$, $\bar{C}{\rightarrow}\bar{B}{\in}\Theta$, $\bar{H}{\in}\Theta$ and $\bar{I}{\in}\Theta$, then $\bar{G}{\wedge}(\bar{C}{\rightarrow}\bar{B}){\wedge}\bar{H}{\wedge}\bar{I}{\wedge}(E{\rightarrow}J){\in}\Theta$.
Since ${\lozenge}\Theta{\subseteq}\Delta$, then ${\lozenge}(\bar{G}{\wedge}(\bar{C}{\rightarrow}\bar{B}){\wedge}\bar{H}{\wedge}\bar{I}{\wedge}(E{\rightarrow}J))
$\linebreak$
{\in}\Delta$.
Since $\bar{A}{\wedge}{\lozenge}\bar{G}{\rightarrow}{\square}E{\vee}F{\in}\L$, ${\lozenge}\bar{H}{\wedge}{\lozenge}\bar{B}{\rightarrow}\bar{A}{\in}\L$ and ${\lozenge}\bar{I}{\rightarrow}{\square}\bar{C}{\vee}\bar{D}{\in}\L$, then by
\linebreak
Lemma~\ref{first:lemma:about:the:use:of:the:special:inference:rule}, ${\lozenge}(\bar{G}{\wedge}(\bar{C}{\rightarrow}\bar{B}){\wedge}\bar{H}{\wedge}\bar{I}{\wedge}(E{\rightarrow}J)){\rightarrow}F{\vee}\bar{D}{\vee}{\lozenge}J{\in}\L$.
Since ${\lozenge}(\bar{G}{\wedge}(\bar{C}{\rightarrow}\bar{B}){\wedge}
$\linebreak$
\bar{H}{\wedge}\bar{I}{\wedge}(E{\rightarrow}J)){\in}\Delta$, then $F{\vee}\bar{D}{\vee}{\lozenge}J{\in}\Delta$.
Consequently, either $F{\in}\Delta$, or $\bar{D}{\in}\Delta$, or ${\lozenge}J{\in}\Delta$: a contradiction.
Let ${\mathcal T}{=}\{\Lambda\ :\ \Lambda$ is a $\L$-theory such that {\bf (1)}~${\lozenge}\Theta{\subseteq}\Lambda$ and {\bf (2)}~for all formulas $E,F{\in}\Fo$, if ${\square}E{\vee}F{\in}\Lambda$ then either $E{\in}\Theta$, or $F{\in}\Delta\}$.
Obviously, by~$(\ast\ast)$, $\L{+}(\Sigma{\cup}{\lozenge}\Theta){\in}{\mathcal T}$.
Hence, ${\mathcal T}$ is nonempty.
Moreover, for all nonempty chains $(\Lambda_{i})_{i{\in}I}$ of elements of ${\mathcal T}$, $\bigcup\{\Lambda_{i}\ :\ i{\in}I\}$ is an element of ${\mathcal T}$.
Consequently, by Lemma~\ref{lemma:Zorn:lemma}, ${\mathcal T}$ possesses a maximal element $\Lambda$.
Hence, $\Lambda$ is a $\L$-theory such that ${\lozenge}\Theta{\subseteq}\Lambda$ and for all formulas $E,F$, if ${\square}E{\vee}F{\in}\Lambda$ then either $E{\in}\Theta$, or $F{\in}\Delta$.
Thus, it only remains to be proved that $\Lambda$ is proper and prime, $\Delta{\supseteq}\Lambda$ and $\Lambda{\bowtie}\Theta$.
We claim that $\Lambda$ is proper.
If not, $\Lambda{=}\Fo$.
Consequently, ${\square}{\bot}{\vee}{\bot}{\in}\Lambda$.
Hence, either ${\bot}{\in}\Theta$, or ${\bot}{\in}\Delta$: a contradiction.
We claim that $\Lambda$ is prime.
If not, there exists formulas $A,B$ such that $A{\vee}B{\in}\Lambda$, $A{\not\in}\Lambda$ and $B{\not\in}\Lambda$.
Thus, by the maximality of $\Lambda$ in ${\mathcal T}$, $\Lambda{+}A{\not\in}{\mathcal T}$ and $\Lambda{+}B{\not\in}{\mathcal T}$.
Consequently, there exists formulas $E,F$ such that ${\square}E{\vee}F{\in}\Lambda{+}A$, $E{\not\in}\Theta$ and $F{\not\in}\Delta$ and there exists formulas $G,H$ such that ${\square}G{\vee}H{\in}\Lambda{+}B$, $G{\not\in}\Theta$ and $H{\not\in}\Delta$.
Hence, $A{\rightarrow}{\square}E{\vee}F{\in}
$\linebreak$
\Lambda$ and $B{\rightarrow}{\square}G{\vee}H{\in}\Lambda$.
Thus, using axioms and inference rules of $\IPL$, $A{\vee}B{\rightarrow}{\square}E{\vee}
$\linebreak$
{\square}G{\vee}F{\vee}H{\in}\Lambda$.
Since $A{\vee}B{\in}\Lambda$, then ${\square}E{\vee}{\square}G{\vee}F{\vee}H{\in}\Lambda$.
Consequently, using axiom $(\Axiom1)$ and inference rule $(\Rule1)$, ${\square}(E{\vee}G){\vee}F{\vee}H{\in}\Lambda$.
Hence, either $E{\vee}G{\in}\Theta$, or $F{\vee}H{\in}\Delta$.
In the former case, either $E{\in}\Theta$, or $G{\in}\Theta$: a contradiction.
In the latter case, either $F{\in}\Delta$, or $H{\in}\Delta$: a contradiction.
We claim that $\Delta{\supseteq}\Lambda$.
If not, there exists a formula $A$ such that $A{\not\in}\Delta$ and $A{\in}\Lambda$.
Thus, ${\square}{\bot}{\vee}A{\in}\Lambda$.
Consequently, either ${\bot}{\in}\Theta$, or $A{\in}\Delta$.
Since $A{\not\in}\Delta$, then ${\bot}{\in}\Theta$: a contradiction.
We claim that $\Lambda{\bowtie}\Theta$.
If not, there exists a formula $A$ such that ${\square}A{\in}\Lambda$ and $A{\not\in}\Theta$.
Hence, ${\square}A{\vee}{\bot}{\in}\Lambda$.
Thus, either $A{\in}\Theta$, or ${\bot}{\in}\Delta$.
Since $A{\not\in}\Theta$, then ${\bot}{\in}\Delta$: a contradiction.
\medskip
\end{proof}
%
%
%Lemma~\ref{lemma:about:what:happens:if:4:square:is:in:the:logic} is used in Section~\ref{section:soundness:and:completeness} when we show that $\L_{\min}{\oplus}{\square}p{\rightarrow}{\square}{\square}p$ is the intuitionistic modal logic determined by the class of all frames $(W,{\leq},{R})$ such that for all $s,t,u,v{\in}W$, if $s{R}t$, $t{\leq}u$ and $u{R}v$ then $s{\leq}{\circ}{R}{\circ}{\leq}v$.
%
%
\begin{lemma}\label{lemma:about:what:happens:if:4:square:is:in:the:logic}
If ${\square}p{\rightarrow}{\square}{\square}p{\in}\L$ then $(W_{\L},{\leq_{\L}},{R_{\L}})$ is up-transitive.
\end{lemma}
\begin{proof}
Suppose ${\square}p{\rightarrow}{\square}{\square}p{\in}\L$.
Let $\Gamma,\Delta,\Lambda,\Theta{\in}W_{\L}$ be such that $\Gamma{R_{\L}}\Delta$, $\Delta{\leq_{\L}}\Lambda$ and $\Lambda{R_{\L}}\Theta$.
Hence, $\Gamma{\bowtie}\Delta$, $\Delta{\subseteq}\Lambda$ and $\Lambda{\bowtie}\Theta$.
We claim that $(\ast)$~${\square}\Gamma{\subseteq}\Theta$.
If not, there exists a formula $A$ such that $A{\in}{\square}\Gamma$ and $A{\not\in}\Theta$.
Thus, ${\square}A{\in}\Gamma$.
Since ${\square}p{\rightarrow}{\square}{\square}p{\in}\L$, then ${\square}{\square}A{\in}\Gamma$.
Since $\Gamma{\bowtie}\Delta$, then ${\square}A{\in}\Delta$.
Since $\Delta{\subseteq}\Lambda$, then ${\square}A{\in}\Lambda$.
Since $\Lambda{\bowtie}\Theta$, then $A{\in}\Theta$: a contradiction.
Let ${\mathcal S}{=}\{\Phi:\ \Phi$ is a $\L$-theory such that {\bf (1)}~$\Gamma{\subseteq}\Phi$ and {\bf (2)}~${\square}\Phi{\subseteq}\Theta\}$.
Obviously, by~$(\ast)$, $\Gamma{\in}{\mathcal S}$.
Consequently, ${\mathcal S}$ is nonempty.
Moreover, for all nonempty chains $(\Phi_{i})_{i{\in}I}$ of elements of ${\mathcal S}$, $\bigcup\{\Phi_{i}\ :\ i{\in}I\}$ is an element of ${\mathcal S}$.
Hence, by Lemma~\ref{lemma:Zorn:lemma}, ${\mathcal S}$ possesses a maximal element $\Phi$.
Thus, $\Phi$ is a $\L$-theory such that $\Gamma{\subseteq}\Phi$ and ${\square}\Phi{\subseteq}\Theta$.
We claim that $\Phi$ is proper.
If not, $\Phi{=}\Fo$.
Consequently, ${\square}{\bot}{\in}\Phi$.
%Hence, ${\bot}{\in}{\square}\Phi$.
Since ${\square}\Phi{\subseteq}\Theta$, then ${\bot}{\in}\Theta$: a contradiction.
We claim that $\Phi$ is prime.
If not, there exists formulas $A,B$ such that $A{\vee}B{\in}\Phi$, $A{\not\in}\Phi$ and $B{\not\in}\Phi$.
Thus, by the maximality of $\Phi$ in ${\mathcal S}$, $\Phi{+}A{\not\in}{\mathcal S}$ and $\Phi{+}B{\not\in}{\mathcal S}$.
Consequently, ${\square}(\Phi{+}A){\not\subseteq}\Theta$ and ${\square}(\Phi{+}B){\not\subseteq}\Theta$.
Hence, there exists formulas $C,D$ such that $C{\in}{\square}(\Phi{+}A)$, $C{\not\in}\Theta$, $D{\in}{\square}(\Phi{+}B)$ and $D{\not\in}\Theta$.
%Thus, ${\square}C{\in}\Phi{+}A$ and ${\square}D{\in}\Phi{+}B$.
Consequently, $A{\rightarrow}{\square}C{\in}\Phi$ and $B{\rightarrow}{\square}D{\in}\Phi$.
Hence, using axioms and inference rules of $\IPL$, $A{\vee}B{\rightarrow}{\square}C{\vee}{\square}D{\in}\Phi$.
Since $A{\vee}B{\in}\Phi$, then ${\square}C{\vee}{\square}D{\in}\Phi$.
Thus, using axiom $(\Axiom1)$ and inference rule $(\Rule1)$, ${\square}(C{\vee}D){\in}\Phi$.
%Consequently, $C{\vee}D{\in}{\square}\Phi$.
Since ${\square}\Phi{\subseteq}\Theta$, then $C{\vee}D{\in}\Theta$.
Hence, either $C{\in}\Theta$, or $D{\in}\Theta$: a contradiction.
We claim that $(\ast\ast)$~for all formulas $C,D$, if $C{\vee}D{\in}{\square}\Phi$ then either ${\lozenge}C{\in}\Phi$, or $D{\in}\Theta$.
If not, there exists formulas $C,D$ such that $C{\vee}D{\in}{\square}\Phi$, ${\lozenge}C{\not\in}\Phi$ and $D{\not\in}\Theta$.
Thus, by the maximality of $\Phi$ in ${\mathcal S}$, $\Phi{+}{\lozenge}C{\not\in}{\mathcal S}$.
Consequently, ${\square}(\Phi{+}{\lozenge}C){\not\subseteq}\Theta$.
Hence, there exists a formula $E$ such that $E{\in}{\square}(\Phi{+}{\lozenge}C)$ and $E{\not\in}\Theta$.
%Thus, ${\square}E{\in}\Phi{+}{\lozenge}C$.
Consequently, ${\lozenge}C{\rightarrow}{\square}E{\in}\Phi$.
Hence, using axiom $(\Axiom1)$ and inference rule $(\Rule1)$, ${\lozenge}C{\rightarrow}{\square}(D{\vee}E){\in}\Phi$.
Since $C{\vee}D{\in}
$\linebreak$
{\square}\Phi$, then ${\square}(C{\vee}D){\in}\Phi$.
Thus, using axiom $(\Axiom1)$ and inference rule $(\Rule1)$, ${\square}(C{\vee}D{\vee}
$\linebreak$
E){\in}\Phi$.
Since ${\lozenge}C{\rightarrow}{\square}(D{\vee}E){\in}\Phi$, then using axiom $(\Axiom2)$, ${\square}(D{\vee}E){\in}\Phi$.
%Consequently, $D{\vee}E{\in}{\square}\Phi$.
Since ${\square}\Phi{\subseteq}
$\linebreak$
\Theta$, then $D{\vee}E{\in}\Theta$.
Hence, either $D{\in}\Theta$, or $E{\in}\Theta$: a contradiction.
Let ${\mathcal T}{=}\{\Psi\ :\ \Psi$ is a $\L$-theory such that {\bf (1)}~${\square}\Phi{\subseteq}\Psi$ and {\bf (2)}~for all formulas $C,D$, if $C{\vee}D{\in}\Psi$ then either ${\lozenge}C{\in}\Phi$, or $D{\in}\Theta\}$.
Obviously, by~$(\ast\ast)$, ${\square}\Phi\in{\mathcal T}$.
Thus, ${\mathcal T}$ is nonempty.
Moreover, for all nonempty chains $(\Psi_{i})_{i{\in}I}$ of elements of ${\mathcal T}$, $\bigcup\{\Psi_{i}\ :\ i{\in}I\}$ is an element of ${\mathcal T}$.
Consequently, by Lemma~\ref{lemma:Zorn:lemma}, ${\mathcal T}$ possesses a maximal element $\Psi$.
Hence, $\Psi$ is a $\L$-theory such that ${\square}\Phi{\subseteq}\Psi$ and for all formulas $C,D$, if $C{\vee}D{\in}\Psi$ then either ${\lozenge}C{\in}\Phi$, or $D{\in}\Theta$.
Thus, it only remains to be proved that $\Psi$ is proper and prime, $\Phi{\bowtie}\Psi$ and $\Psi{\subseteq}\Theta$.
We claim that $\Psi$ is proper.
If not, $\Psi{=}\Fo$.
Consequently, ${\bot}{\vee}{\bot}{\in}\Psi$.
Hence, either ${\lozenge}{\bot}{\in}\Phi$, or ${\bot}{\in}\Theta$: a contradiction.
We claim that $\Psi$ is prime.
If not, there exists formulas $E,F$ such that $E{\vee}F{\in}\Psi$, $E{\not\in}\Psi$ and $F{\not\in}\Psi$.
Thus, by the maximality of $\Psi$ in ${\mathcal T}$, $\Psi{+}E{\not\in}{\mathcal T}$ and $\Psi{+}F{\not\in}{\mathcal T}$.
Consequently, there exists formulas $G,H$ such that $G{\vee}H{\in}\Psi{+}E$, ${\lozenge}G{\not\in}\Phi$ and $H{\not\in}\Theta$ and there exists formulas $I,J$ such that $I{\vee}J{\in}\Psi{+}F$, ${\lozenge}I{\not\in}\Phi$ and $J{\not\in}\Theta$.
Hence, $E{\rightarrow}G{\vee}H{\in}
$\linebreak$
\Psi$ and $F{\rightarrow}I{\vee}J{\in}\Psi$.
Thus, using axioms and inference rules of $\IPL$, $E{\vee}F{\rightarrow}G{\vee}I{\vee}H
$\linebreak$
{\vee}J{\in}\Psi$.
Since $E{\vee}F{\in}\Psi$, then $G{\vee}I{\vee}H{\vee}J{\in}\Psi$.
Consequently, either ${\lozenge}(G{\vee}I){\in}\Phi$, or $H{\vee}J{\in}\Theta$.
In the former case, using axiom $(\Axiom3)$, ${\lozenge}G{\vee}{\lozenge}I{\in}\Phi$.
Hence, either ${\lozenge}G{\in}\Phi$, or ${\lozenge}I{\in}\Phi$: a contradiction.
In the latter case, either $H{\in}\Theta$, or $J{\in}\Theta$: a contradiction.
We claim that $\Phi{\bowtie}\Psi$.
If not, there exists a formula $E$ such that $E{\in}\Psi$ and ${\lozenge}E{\not\in}\Phi$.
Thus, $E{\vee}{\bot}{\in}\Psi$.
Consequently, either ${\lozenge}E{\in}\Phi$, or ${\bot}{\in}\Theta$.
Since ${\lozenge}E{\not\in}\Phi$, then ${\bot}{\in}\Theta$: a contradiction.
We claim that $\Psi{\subseteq}\Theta$.
If not, there exists a formula $E$ such that $E{\in}\Psi$ and $E{\not\in}\Theta$.
Hence, ${\bot}{\vee}E{\in}\Psi$.
Thus, either ${\lozenge}{\bot}{\in}\Phi$, or $E{\in}\Theta$.
Since $E{\not\in}\Theta$, then ${\lozenge}{\bot}{\in}\Phi$: a contradiction.
\medskip
\end{proof}
\begin{lemma}\label{lemma:about:what:happens:if:4:lozenge:is:in:the:logic}
If ${\lozenge}{\lozenge}p{\rightarrow}{\lozenge}p{\in}\L$ then $(W_{\L},{\leq_{\L}},{R_{\L}})$ is down-transitive.
\end{lemma}
\begin{proof}
Suppose ${\lozenge}{\lozenge}p{\rightarrow}{\lozenge}p{\in}\L$.
Let $\Gamma,\Delta,\Lambda,\Theta{\in}W_{\L}$ be such that $\Gamma{R_{\L}}\Delta$, $\Delta{\geq_{\L}}\Lambda$ and $\Lambda{R_{\L}}\Theta$.
Hence, $\Gamma{\bowtie}\Delta$, $\Delta{\supseteq}\Lambda$ and $\Lambda{\bowtie}\Theta$.
We claim that $(\ast)$~${\lozenge}\Theta{\subseteq}\Gamma$.
If not, there exists a formula $A$ such that $A{\in}\Theta$ and ${\lozenge}A{\not\in}\Gamma$.
Since $\Lambda{\bowtie}\Theta$, then ${\lozenge}A{\in}\Lambda$.
Since $\Delta{\supseteq}\Lambda$, then ${\lozenge}A{\in}\Delta$.
Since $\Gamma{\bowtie}\Delta$, then ${\lozenge}{\lozenge}A{\in}\Gamma$.
Since ${\lozenge}{\lozenge}p{\rightarrow}{\lozenge}p{\in}\L$, then ${\lozenge}A{\in}\Gamma$: a contradiction.
Let ${\mathcal S}{=}\{\Psi:\ \Psi$ is a $\L$-theory such that {\bf (1)}~${\lozenge}\Psi{\subseteq}\Gamma$ and {\bf (2)}~$\Psi{\supseteq}\Theta\}$.
Obviously, by~$(\ast)$, $\Gamma{\in}{\mathcal S}$.
Consequently, ${\mathcal S}$ is nonempty.
Moreover, for all nonempty chains $(\Psi_{i})_{i{\in}I}$ of elements of ${\mathcal S}$, $\bigcup\{\Psi_{i}\ :\ i{\in}I\}$ is an element of ${\mathcal S}$.
Hence, by Lemma~\ref{lemma:Zorn:lemma}, ${\mathcal S}$ possesses a maximal element $\Psi$.
Thus, $\Psi$ is a $\L$-theory such that ${\lozenge}\Psi{\subseteq}\Gamma$ and $\Psi{\supseteq}\Theta$.
We claim that $\Psi$ is proper.
If not, ${\bot}{\in}\Psi$.
Since ${\lozenge}\Psi{\subseteq}\Gamma$, then ${\lozenge}{\bot}{\in}\Gamma$: a contradiction.
We claim that $\Psi$ is prime.
If not, there exists formulas $A,B$ such that $A{\vee}B{\in}\Psi$, $A{\not\in}\Psi$ and $B{\not\in}\Psi$.
Hence, by the maximality of $\Psi$ in ${\mathcal S}$, $\Psi{+}A{\not\in}{\mathcal S}$ and $\Psi{+}B{\not\in}{\mathcal S}$.
Consequently, ${\lozenge}(\Psi{+}A){\not\subseteq}\Gamma$ and ${\lozenge}(\Psi{+}B){\not\subseteq}\Gamma$.
Thus, there exists formulas $C,D$ such that $C{\in}(\Psi{+}A)$, ${\lozenge}C{\not\in}\Gamma$, $D{\in}(\Psi{+}B)$ and ${\lozenge}D{\not\in}\Gamma$.
Consequently, $A{\rightarrow}C{\in}\Psi$ and $B{\rightarrow}D{\in}\Psi$.
Hence, using axioms and inference rules of $\IPL$, $A{\vee}B{\rightarrow}C{\vee}D{\in}\Psi$.
Since $A{\vee}B{\in}\Psi$, then $C{\vee}D{\in}\Psi$.
Since ${\lozenge}\Psi{\subseteq}\Gamma$, then ${\lozenge}(C{\vee}D){\in}\Gamma$.
Thus, using axiom $(\Axiom3)$, ${\lozenge}C{\vee}{\lozenge}D{\in}
$\linebreak$
\Gamma$.
Consequently, either ${\lozenge}C{\in}\Gamma$, or ${\lozenge}D{\in}\Gamma$: a contradiction.
Let $\Sigma{=}\{A{\in}\Fo:$ there exists $B,C,D{\in}\Fo$ such that ${\lozenge}B{\rightarrow}A{\in}\L{+}{\lozenge}\Psi$, ${\square}C{\vee}D{\in}\L
$\linebreak$
{+}{\lozenge}\Psi$, $C{\rightarrow}B{\in}\Psi$ and $D{\not\in}\Gamma\}$.
We claim that $(\ast\ast)$~for all $E,F{\in}\Fo$, if ${\square}E{\vee}F{\in}\L{+}(\Sigma{\cup}{\lozenge}\Psi)$ then either $E{\in}\Psi$, or $F{\in}\Gamma$.
If not, there exists $E,F{\in}\Fo$ such that ${\square}E{\vee}F{\in}\L{+}(\Sigma{\cup}{\lozenge}\Psi)$, $E{\not\in}\Psi$ and $F{\not\in}\Gamma$.
Hence, there exists $m,n{\in}\N$, there exists $A_{1},\ldots,A_{m}{\in}\Sigma$ and there exists $G_{1},\ldots,G_{n}{\in}
$\linebreak$
\Psi$ such that $A_{1}{\wedge}\ldots{\wedge}A_{m}{\wedge}{\lozenge}G_{1}{\wedge}\ldots{\wedge}{\lozenge}G_{n}{\rightarrow}{\square}E{\vee}F{\in}\L$.
Thus, for all $i{\in}\{1,\ldots,m\}$, there exists $B_{i},C_{i},D_{i}{\in}\Fo$ such that ${\lozenge}B_{i}{\rightarrow}A_{i}{\in}\L{+}{\lozenge}\Psi$, ${\square}C_{i}{\vee}D_{i}{\in}\L{+}{\lozenge}\Psi$, $C_{i}{\rightarrow}B_{i}{\in}
$\linebreak$
\Psi$ and $D_{i}{\not\in}\Gamma$.
Let $\bar{A}{=}A_{1}{\wedge}\ldots{\wedge}A_{m}$ and $\bar{G}{=}G_{1}{\wedge}\ldots{\wedge}G_{m}$.
Since $G_{1},\ldots,G_{n}{\in}\Psi$, then using axioms and inference rules of $\IPL$, $\bar{G}{\in}\Psi$.
Since $A_{1}{\wedge}\ldots{\wedge}A_{m}{\wedge}{\lozenge}G_{1}{\wedge}\ldots
$\linebreak$
{\wedge}{\lozenge}G_{n}{\rightarrow}{\square}E{\vee}F{\in}\L$, then using inference rule $(\Rule2)$, $\bar{A}{\wedge}{\lozenge}\bar{G}{\rightarrow}{\square}E{\vee}F{\in}\L$.
Let $\bar{B}{=}B_{1}
$\linebreak$
{\wedge}\ldots{\wedge}B_{m}$, $\bar{C}{=}C_{1}{\wedge}\ldots{\wedge}C_{m}$ and $\bar{D}{=}D_{1}{\vee}\ldots{\vee}D_{m}$.
Since for all $i{\in}\{1,\ldots,m\}$, ${\lozenge}B_{i}
$\linebreak$
{\rightarrow}A_{i}{\in}\L{+}{\lozenge}\Psi$, then using inference rule $(\Rule2)$, ${\lozenge}\bar{B}{\rightarrow}\bar{A}{\in}\L{+}{\lozenge}\Psi$.
Since for all $i{\in}\{1,
$\linebreak$
\ldots,m\}$, ${\square}C_{i}{\vee}D_{i}{\in}\L{+}{\lozenge}\Psi$, then using axiom $(\Axiom1)$ and inference rule $(\Rule1)$, ${\square}\bar{C}{\vee}\bar{D}{\in}
$\linebreak$
\L{+}{\lozenge}\Psi$.
Since for all $i{\in}\{1,\ldots,m\}$, $C_{i}{\rightarrow}B_{i}{\in}\Psi$, then $\bar{C}{\rightarrow}\bar{B}{\in}\Psi$.
Since for all $i{\in}\{1,
$\linebreak$
\ldots,m\}$, $D_{i}{\not\in}\Gamma$, then using axioms and inference rules of $\IPL$, $\bar{D}{\not\in}\Gamma$.
Since ${\lozenge}\bar{B}{\rightarrow}\bar{A}{\in}
$\linebreak$
\L{+}{\lozenge}\Psi$ and ${\square}\bar{C}{\vee}\bar{D}{\in}\L{+}{\lozenge}\Psi$, then there exists $o,p{\in}\N$, there exists $H_{1},\ldots,H_{o}{\in}\Psi$ and there exists $I_{1},\ldots,I_{p}{\in}\Psi$ such that ${\lozenge}H_{1}{\wedge}\ldots{\wedge}{\lozenge}H_{o}{\wedge}{\lozenge}\bar{B}{\rightarrow}\bar{A}{\in}\L$ and ${\lozenge}I_{1}{\wedge}\ldots{\wedge}{\lozenge}I_{p}
$\linebreak$
{\rightarrow}{\square}\bar{C}{\vee}\bar{D}{\in}\L$.
Let $\bar{H}{=}H_{1}{\wedge}\ldots{\wedge}H_{o}$ and $\bar{I}{=}I_{1}{\wedge}\ldots{\wedge}I_{p}$.
Since $H_{1},\ldots,H_{o}{\in}\Psi$ and $I_{1},\ldots,I_{p}{\in}\Psi$, then using axioms and inference rules of $\IPL$, $\bar{H}{\in}\Psi$ and $\bar{I}{\in}\Psi$.
Since ${\lozenge}H_{1}{\wedge}\ldots{\wedge}{\lozenge}H_{o}{\wedge}{\lozenge}\bar{B}{\rightarrow}\bar{A}{\in}\L$ and ${\lozenge}I_{1}{\wedge}\ldots{\wedge}{\lozenge}I_{p}{\rightarrow}{\square}\bar{C}{\vee}\bar{D}{\in}\L$, then using inference rule $(\Rule2)$, ${\lozenge}\bar{H}{\wedge}{\lozenge}\bar{B}{\rightarrow}\bar{A}{\in}\L$ and ${\lozenge}\bar{I}{\rightarrow}{\square}\bar{C}{\vee}\bar{D}{\in}\L$.
Since $E{\not\in}\Psi$, then by the maximality of $\Psi$ in ${\mathcal S}$, $\Psi{+}E{\not\in}{\mathcal S}$.
Consequently, ${\lozenge}(\Psi{+}E){\not\subseteq}\Gamma$.
Hence, there exists $J{\in}\Fo$ such that $J{\in}\Psi{+}E$ and ${\lozenge}J{\not\in}\Gamma$.
Thus, $E{\rightarrow}J{\in}\Psi$.
Since $\bar{G}{\in}\Psi$, $\bar{C}{\rightarrow}\bar{B}{\in}\Psi$, $\bar{H}{\in}\Psi$ and $\bar{I}{\in}\Psi$, then $\bar{G}{\wedge}(\bar{C}{\rightarrow}\bar{B}){\wedge}\bar{H}{\wedge}\bar{I}{\wedge}(E{\rightarrow}J){\in}\Psi$.
Since ${\lozenge}\Psi{\subseteq}\Gamma$, then ${\lozenge}(\bar{G}{\wedge}(\bar{C}{\rightarrow}\bar{B}){\wedge}\bar{H}{\wedge}\bar{I}{\wedge}(E{\rightarrow}J))
$\linebreak$
{\in}\Gamma$.
Since $\bar{A}{\wedge}{\lozenge}\bar{G}{\rightarrow}{\square}E{\vee}F{\in}\L$, ${\lozenge}\bar{H}{\wedge}{\lozenge}\bar{B}{\rightarrow}\bar{A}{\in}\L$ and ${\lozenge}\bar{I}{\rightarrow}{\square}\bar{C}{\vee}\bar{D}{\in}\L$, then by
\linebreak
Lemma~\ref{first:lemma:about:the:use:of:the:special:inference:rule}, ${\lozenge}(\bar{G}{\wedge}(\bar{C}{\rightarrow}\bar{B}){\wedge}\bar{H}{\wedge}\bar{I}{\wedge}(E{\rightarrow}J)){\rightarrow}F{\vee}\bar{D}{\vee}{\lozenge}J{\in}\L$.
Since ${\lozenge}(\bar{G}{\wedge}(\bar{C}{\rightarrow}\bar{B}){\wedge}
$\linebreak$
\bar{H}{\wedge}\bar{I}{\wedge}(E{\rightarrow}J)){\in}\Gamma$, then $F{\vee}\bar{D}{\vee}{\lozenge}J{\in}\Gamma$.
Consequently, either $F{\in}\Gamma$, or $\bar{D}{\in}\Gamma$, or ${\lozenge}J{\in}
$\linebreak$
\Gamma$: a contradiction.
Let ${\mathcal T}{=}\{\Phi\ :\ \Phi$ is a $\L$-theory such that {\bf (1)}~${\lozenge}\Psi{\subseteq}\Phi$ and {\bf (2)}~for all formulas $E,F{\in}\Fo$, if ${\square}E{\vee}F{\in}\Phi$ then either $E{\in}\Psi$, or $F{\in}\Gamma\}$.
Obviously, by~$(\ast\ast)$, $\L{+}(\Sigma{\cup}{\lozenge}\Psi){\in}{\mathcal T}$.
Hence, ${\mathcal T}$ is nonempty.
Moreover, for all nonempty chains $(\Phi_{i})_{i{\in}I}$ of elements of ${\mathcal T}$, $\bigcup\{\Phi_{i}\ :\ i{\in}I\}$ is an element of ${\mathcal T}$.
Consequently, by Lemma~\ref{lemma:Zorn:lemma}, ${\mathcal T}$ possesses a maximal element $\Phi$.
Hence, $\Phi$ is a $\L$-theory such that ${\lozenge}\Psi{\subseteq}\Phi$ and for all formulas $E,F{\in}\Fo$, if ${\square}E{\vee}F{\in}\Phi$ then either $E{\in}\Psi$, or $F{\in}\Gamma$.
Thus, it only remains to be proved that $\Phi$ is proper and prime, $\Gamma{\supseteq}\Phi$ and $\Phi{\bowtie}\Psi$.
We claim that $\Phi$ is proper.
If not, $\Phi{=}\Fo$.
Consequently, ${\square}{\bot}{\vee}{\bot}{\in}\Phi$.
Hence, either ${\bot}{\in}\Psi$, or ${\bot}{\in}\Gamma$: a contradiction.
We claim that $\Phi$ is prime.
If not, there exists formulas $A,B$ such that $A{\vee}B{\in}\Phi$, $A{\not\in}\Phi$ and $B{\not\in}\Phi$.
Thus, by the maximality of $\Phi$ in ${\mathcal T}$, $\Phi{+}A{\not\in}{\mathcal T}$ and $\Phi{+}B{\not\in}{\mathcal T}$.
Consequently, there exists formulas $E,F$ such that ${\square}E{\vee}F{\in}\Phi{+}A$, $E{\not\in}\Psi$ and $F{\not\in}\Gamma$ and there exists formulas $G,H$ such that ${\square}G{\vee}H{\in}\Phi{+}B$, $G{\not\in}\Psi$ and $H{\not\in}\Gamma$.
Hence, $A{\rightarrow}{\square}E{\vee}F{\in}
$\linebreak$
\Phi$ and $B{\rightarrow}{\square}G{\vee}H{\in}\Phi$.
Thus, using axioms and inference rules of $\IPL$, $A{\vee}B{\rightarrow}{\square}E{\vee}
$\linebreak$
{\square}G{\vee}F{\vee}H{\in}\Phi$.
Since $A{\vee}B{\in}\Phi$, then ${\square}E{\vee}{\square}G{\vee}F{\vee}H{\in}\Phi$.
Consequently, using axiom $(\Axiom1)$ and inference rule $(\Rule1)$, ${\square}(E{\vee}G){\vee}F{\vee}H{\in}\Phi$.
Hence, either $E{\vee}G{\in}\Psi$, or $F{\vee}H{\in}\Gamma$.
In the former case, either $E{\in}\Psi$, or $G{\in}\Psi$: a contradiction.
In the latter case, either $F{\in}\Gamma$, or $H{\in}\Gamma$: a contradiction.
We claim that $\Gamma{\supseteq}\Phi$.
If not, there exists a formula $A$ such that $A{\not\in}\Gamma$ and $A{\in}\Phi$.
Thus, ${\square}{\bot}{\vee}A{\in}\Phi$.
Consequently, either ${\bot}{\in}\Psi$, or $A{\in}\Gamma$.
Since $A{\not\in}\Gamma$, then ${\bot}{\in}\Psi$: a contradiction.
We claim that $\Phi{\bowtie}\Psi$.
If not, there exists a formula $A$ such that ${\square}A{\in}\Phi$ and $A{\not\in}\Psi$.
Hence, ${\square}A{\vee}{\bot}{\in}\Phi$.
Thus, either $A{\in}\Psi$, or ${\bot}{\in}\Gamma$.
Since $A{\not\in}\Psi$, then ${\bot}{\in}\Gamma$: a contradiction.
\medskip
\end{proof}
\section{Soundness and completeness}\label{section:soundness:and:completeness}
\begin{definition}[Canonicity]
An intuitionistic modal logic $\L$ is {\em canonical}\/ if $(W_{\L},{\leq_{\L}},
$\linebreak$
{R_{\L}}){\models}\L$.
\end{definition}
\begin{proposition}\label{proposition:canonicity}
$\L_{\min}$ is canonical.
Moreover, the following intuitionistic modal logics are canonical: $\L_{\fcfra}$, $\L_{\bcfra}$, $\L_{\dcfra}$, $\L_{\fbcfra}$, $\L_{\fdcfra}$, $\L_{\bdcfra}$ and $\L_{\fbdcfra}$.
\end{proposition}
\begin{proof}
By Lemmas~\ref{lemma:about:axiom:fc:soundness}, \ref{lemma:about:axiom:2}, \ref{lemma:about:rule:3}, \ref{fc:canonical:frame:is:forward:confluent}, \ref{bc:canonical:frame:is:backward:confluent} and~\ref{dc:canonical:frame:is:downward:confluent}.
Indeed, suppose for example that $\L_{\fcfra}$ is not canonical.
Hence, $(W_{\L_{\fcfra}},{\leq_{\L_{\fcfra}}},{R_{\L_{\fcfra}}}){\not\models}\L_{\fcfra}$.
Since by Lemma~\ref{fc:canonical:frame:is:forward:confluent}, $(W_{\L_{\fcfra}},{\leq_{\L_{\fcfra}}},
$\linebreak$
{R_{\L_{\fcfra}}})$ is forward confluent, then by Lemma~\ref{lemma:about:axiom:fc:soundness}, \ref{lemma:about:axiom:2} and~\ref{lemma:about:rule:3}, $(W_{\L_{\fcfra}},{\leq_{\L_{\fcfra}}},{R_{\L_{\fcfra}}}){\models}\L_{\fcfra}$: a contradiction.
\medskip
\end{proof}
\begin{proposition}
The following intuitionistic modal logics are canonical: $\L_{\reflexive}$, $\L_{\symmetric}$ and $\L_{\reflexive}{\oplus}\L_{\symmetric}$.
\end{proposition}
\begin{proof}
Similar to the proof of Proposition~\ref{proposition:canonicity}, this time using Lemmas~\ref{lemma:correspondence:formulas:T:B:4:elementary:conditions}, \ref{lemma:about:axiom:2}, \ref{lemma:about:rule:3}, \ref{lemma:about:what:happens:if:T:square:is:in:the:logic}, \ref{lemma:about:what:happens:if:T:lozenge:is:in:the:logic}, \ref{lemma:about:what:happens:if:B:square:is:in:the:logic}, \ref{lemma:about:what:happens:if:B:lozenge:is:in:the:logic}, \ref{lemma:about:what:happens:if:4:square:is:in:the:logic} and~\ref{lemma:about:what:happens:if:4:lozenge:is:in:the:logic}.
\medskip
\end{proof}
\begin{proposition}
The following intuitionistic modal logics are canonical: $\L_{\ureflexive}$, $\L_{\dreflexive}$, $\L_{\usymmetric}$, $\L_{\dsymmetric}$, $\L_{\utransitive}$ and $\L_{\dtransitive}$.
\end{proposition}
\begin{proof}
Similar to the proof of Proposition~\ref{proposition:canonicity}, this time using Lemmas~\ref{lemma:correspondence:formulas:T:B:4:elementary:conditions}, \ref{lemma:about:axiom:2}, \ref{lemma:about:rule:3}, \ref{lemma:about:what:happens:if:T:square:is:in:the:logic}, \ref{lemma:about:what:happens:if:T:lozenge:is:in:the:logic}, \ref{lemma:about:what:happens:if:B:square:is:in:the:logic}, \ref{lemma:about:what:happens:if:B:lozenge:is:in:the:logic}, \ref{lemma:about:what:happens:if:4:square:is:in:the:logic} and~\ref{lemma:about:what:happens:if:4:lozenge:is:in:the:logic}.
\medskip
\end{proof}
\begin{proposition}\label{proposition:soundness:completeness}
$\L_{\min}$ is equal to $\Log({\mathcal C}_{\allfra})$.
Moreover,
\begin{itemize}
\item $\L_{\fcfra}$ is equal to $\Log({\mathcal C}_{\fcfra})$ and $\Log({\mathcal C}_{\qfcfra})$,
\item $\L_{\bcfra}$ is equal to $\Log({\mathcal C}_{\bcfra})$ and $\Log({\mathcal C}_{\qbcfra})$,
\item $\L_{\dcfra}$ is equal to $\Log({\mathcal C}_{\dcfra})$ and $\Log({\mathcal C}_{\qdcfra})$,
\item $\L_{\fbcfra}$ is equal to $\Log({\mathcal C}_{\fbcfra})$,
\item $\L_{\fdcfra}$ is equal to $\Log({\mathcal C}_{\fdcfra})$,
\item $\L_{\bdcfra}$ is equal to $\Log({\mathcal C}_{\bdcfra})$,
\item $\L_{\fbdcfra}$ is equal to $\Log({\mathcal C}_{\fbdcfra})$.
\end{itemize}
\end{proposition}
\begin{proof}
By Lemmas~\ref{lemma:about:modal:definability:quasi:fbd:frames}, \ref{lemma:about:axiom:fc:soundness}, \ref{lemma:about:axiom:2} and~\ref{lemma:about:rule:3}, Lindenbaum Lemma and Lemmas~\ref{lemma:truth:lemma}, \ref{fc:canonical:frame:is:forward:confluent}, \ref{bc:canonical:frame:is:backward:confluent} and \ref{dc:canonical:frame:is:downward:confluent}.
Indeed, suppose for example that $\L_{\fcfra}$ is not $\Log({\mathcal C}_{\fcfra})$.
Hence, either $\L_{\fcfra}{\not\subseteq}\Log({\mathcal C}_{\fcfra})$, or $\L_{\fcfra}{\not\supseteq}\Log({\mathcal C}_{\fcfra})$.
In the former case, there exists a formula $A$ such that $A{\in}\L_{\fcfra}$ and $A{\not\in}\Log({\mathcal C}_{\fcfra})$.
Thus, by Lemmas~\ref{lemma:about:axiom:fc:soundness}, \ref{lemma:about:axiom:2} and~\ref{lemma:about:rule:3}, $A{\in}\Log({\mathcal C}_{\fcfra})$: a contradiction.
In the latter case, there exists a formula $A$ such that $A{\not\in}\L_{\fcfra}$ and $A{\in}\Log({\mathcal C}_{\fcfra})$.
Consequently, by Lindenbaum Lemma, there exists a prime $\L_{\fcfra}$-theory $\Gamma$ such that $A{\not\in}\Gamma$.
Hence, by Lemma~\ref{lemma:truth:lemma}, $\Gamma{\not\models}A$.
Thus, $(W_{\L_{\fcfra}},{\leq_{\L_{\fcfra}}},{R_{\L_{\fcfra}}},V_{\L_{\fcfra}}){\not\models}A$.
Consequently, $(W_{\L_{\fcfra}},{\leq_{\L_{\fcfra}}},{R_{\L_{\fcfra}}}){\not\models}A$.
Since by Lemma~\ref{fc:canonical:frame:is:forward:confluent}, $(W_{\L_{\fcfra}},{\leq_{\L_{\fcfra}}},{R_{\L_{\fcfra}}})$ is forward confluent, then $A{\not\in}\Log({\mathcal C}_{\fcfra})$: a contradiction.
\medskip
\end{proof}
By Lemma~\ref{lemma:some:specific:formulas:valid:non:valid} and Proposition~\ref{proposition:soundness:completeness}, we immediately obtain the following
\begin{lemma}\label{lemma:disjunction:property:does:not:hold:preliminary:result}
Let $\L$ be an intuitionistic modal logic.
If $\L$ contains $\L_{\dcfra}$ and $\L$ is contained in $\L_{\fbdcfra}$ then ${\square}{\bot}{\vee}{\lozenge}{\top}$ is in $\L$, ${\square}{\bot}$ is not in $\L$ and ${\lozenge}{\top}$ is not in $\L$.
\end{lemma}
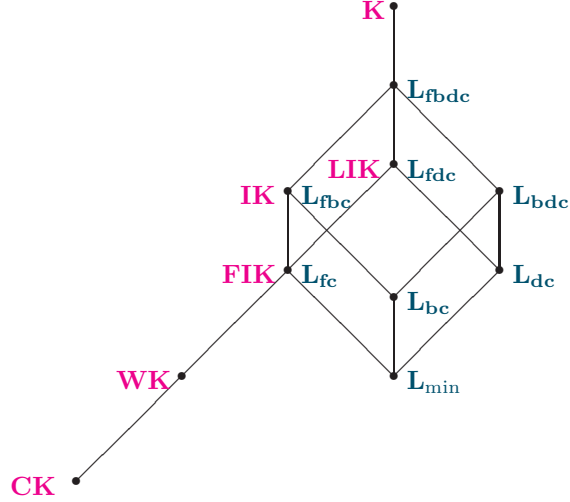
\begin{figure}[ht]
\begin{center}
\vspace{-1.10cm}
\begin{picture}(140,140)(-50,10)
\put(40.00,0.00){\circle*{3}}
\put(0.00,40.00){\circle*{3}}
\put(80.00,40.00){\circle*{3}}
\put(40.00,80.00){\circle*{3}}
\put(40.00,0.00){\line(-1,1){40.00}}
\put(40.00,0.00){\line(1,1){40.00}}
\put(0.00,40.00){\line(1,1){40.00}}
\put(80.00,40.00){\line(-1,1){40.00}}
\put(40.00,110.00){\circle*{3}}
\put(40.00,-30.00){\circle*{3}}
\put(00.00,10.00){\circle*{3}}
\put(80.00,10.00){\circle*{3}}
\put(40.00,50.00){\circle*{3}}
\put(40.00,-30.00){\line(-1,1){40.00}}
\put(40.00,-30.00){\line(1,1){40.00}}
\put(0.00,10.00){\line(1,1){40.00}}
\put(80.00,10.00){\line(-1,1){40.00}}
\put(40.00,-30.00){\line(0,1){30.00}}
\put(0.00,10.00){\line(0,1){30.00}}
\put(80.00,10.00){\line(0,1){30.00}}
\put(40.00,50.00){\line(0,1){30.00}}
%
%
%\put(0.00,70.00){\circle*{3}}
%\put(0.00,40.00){\line(0,1){30.00}}
%
%
\put(40.00,80.00){\line(0,1){30.00}}
\put(45.00,-35.00){\textcolor{monblu}{$\L_{\min}$}}
\put(5.00,5.00){\textcolor{monblu}{$\L_{\fcfra}$}}
\put(45.00,-5.00){\textcolor{monblu}{$\L_{\bcfra}$}}
\put(85.00,5.00){\textcolor{monblu}{$\L_{\dcfra}$}}
\put(5.00,35.00){\textcolor{monblu}{$\L_{\fbcfra}$}}
\put(45.00,45.00){\textcolor{monblu}{$\L_{\fdcfra}$}}
\put(85.00,35.00){\textcolor{monblu}{$\L_{\bdcfra}$}}
\put(45.00,75.00){\textcolor{monblu}{$\L_{\fbdcfra}$}}
\put(-18.00,35.00){\textcolor{monmag}{$\IK$}}
\put(-25.00,5.00){\textcolor{monmag}{$\FIK$}}
\put(+15.00,45.00){\textcolor{monmag}{$\LIK$}}
%
%
%\put(-21.00,65.00){$\IS4$}
%
%
\put(28.00,105.00){\textcolor{monmag}{$\K$}}
\put(-40.00,-30.00){\circle*{3}}
\put(-80.00,-70.00){\circle*{3}}
\put(-40.00,-30.00){\line(1,1){40.00}}
\put(-80.00,-70.00){\line(1,1){40.00}}
\put(-65.00,-35.00){\textcolor{monmag}{$\WK$}}
\put(-105.00,-75.00){\textcolor{monmag}{$\CK$}}
\end{picture}
\vspace{2.91cm}
\caption{The inclusion relationships considered in Propositions~\ref{fcfra:FIK}, \ref{proposition:LFCFRA:WK}, \ref{fbcfra:IK}, \ref{fdcfra:LIK}, \ref{strict:inclusions:between:logics} and~\ref{fbdcfra:K}}\label{figure:lattice:IML:so:far}
\end{center}
\end{figure}
%
%
%\vspace{-0.39cm}
%
%
$\FIK$ is an intuitionistic modal logic presented~---~axioms, inference rules and semantics~---~by Balbiani {\it et al.}~\cite{Balbiani:et:al:2024}.
By considering, on one hand, the axioms and inference rules defining $\L_{\fcfra}$ and, on the other hand, the axioms and inference rules defining $\FIK$, the reader may easily notice that all axioms of $\L_{\fcfra}$ are valid in the semantics of $\FIK$, all inference rules of $\L_{\fcfra}$ preserve validity in this semantics, all axioms of $\FIK$ are valid in the semantics of $\L_{\fcfra}$ and all inference rules of $\FIK$ preserve validity in this semantics.
As a result,
\begin{proposition}\label{fcfra:FIK}
$\L_{\fcfra}$ and $\FIK$ are equal.
\end{proposition}
$\WK$ is an intuitionistic modal logic presented~---~axioms, inference rules and semantics~---~by Wijesekera~\cite{Wijesekera:1990}.
By considering, on one hand, the axioms and inference rules defining $\WK$ and, on the other hand, the axioms and inference rules defining $\L_{\fcfra}$, the reader may easily notice that all axioms of $\WK$ are valid in the semantics of $\L_{\fcfra}$, all inference rules of $\WK$ preserve validity in this semantics and some formulas belonging to $\L_{\fcfra}$ are outside $\WK$.
Indeed, ${\lozenge}(p{\vee}q){\rightarrow}{\lozenge}p{\vee}{\lozenge}q$ belongs to $\L_{\fcfra}$ but the model $(W,{\leq},{R},V)$ defined as follows shows that ${\lozenge}(p{\vee}q){\rightarrow}{\lozenge}p{\vee}{\lozenge}q$ is outside $\WK$: $W{=}\{a,b,c,d\}$, $a{\leq}c$, $a{R}b$ and $c{R}d$ and $V(p){=}\{b\}$ and $V(q){=}\{d\}$.
As a result,
\begin{proposition}\label{proposition:LFCFRA:WK}
$\L_{\fcfra}$ strictly contains $\WK$.
\end{proposition}
$\IK$ is an intuitionistic modal logic presented~---~axioms, inference rules and semantics~---~by Fischer Servi~\cite{FischerServi:1984}.
By considering, on one hand, the axioms and inference rules defining $\L_{\fbcfra}$ and, on the other hand, the axioms and inference rules defining $\IK$, the reader may easily notice that all axioms of $\L_{\fbcfra}$ are valid in the semantics of $\IK$, all inference rules of $\L_{\fbcfra}$ preserve validity in this semantics, all axioms of $\IK$ are valid in the semantics of $\L_{\fbcfra}$ and all inference rules of $\IK$ preserve validity in this semantics.
As a result,
\begin{proposition}\label{fbcfra:IK}
$\L_{\fbcfra}$ and $\IK$ are equal.
\end{proposition}
$\LIK$ is an intuitionistic modal logic presented~---~axioms, inference rules and semantics~---~by Balbiani {\it et al.}~\cite{Balbiani:et:al:2024b}.
By considering, on one hand, the axioms and inference rules defining $\L_{\fdcfra}$ and, on the other hand, the axioms and inference rules defining $\LIK$, the reader may easily notice that all axioms of $\L_{\fdcfra}$ are valid in the semantics of $\LIK$, all inference rules of $\L_{\fdcfra}$ preserve validity in this semantics, all axioms of $\LIK$ are valid in the semantics of $\L_{\fdcfra}$ and all inference rules of $\LIK$ preserve validity in this semantics.
As a result,
\begin{proposition}\label{fdcfra:LIK}
$\L_{\fdcfra}$ and $\LIK$ are equal.
\end{proposition}
Indeed,
\begin{proposition}\label{strict:inclusions:between:logics}
\begin{enumerate}
\item $\L_{\min}$ is strictly contained in $\L_{\fcfra}$,
\item $\L_{\min}$ is strictly contained in $\L_{\bcfra}$,
\item $\L_{\min}$ is strictly contained in $\L_{\dcfra}$,
\item $\L_{\fcfra}$ is strictly contained in $\L_{\fbcfra}$,
\item $\L_{\fcfra}$ is strictly contained in $\L_{\fdcfra}$,
\item $\L_{\bcfra}$ is strictly contained in $\L_{\fbcfra}$,
\item $\L_{\bcfra}$ is strictly contained in $\L_{\bdcfra}$,
\item $\L_{\dcfra}$ is strictly contained in $\L_{\fdcfra}$,
\item $\L_{\dcfra}$ is strictly contained in $\L_{\bdcfra}$,
\item $\L_{\fbcfra}$ is strictly contained in $\L_{\fbdcfra}$,
\item $\L_{\fdcfra}$ is strictly contained in $\L_{\fbdcfra}$,
\item $\L_{\bdcfra}$ is strictly contained in $\L_{\fbdcfra}$.
\end{enumerate}
\end{proposition}
\begin{proof}
$\mathbf{(1)}$~For the sake of the contradiction, suppose $\L_{\min}$ is not strictly contained in $\L_{\fcfra}$.
Thus, ${\lozenge}(p{\rightarrow}q){\rightarrow}({\square}p{\rightarrow}{\lozenge}q)$ is in $\L_{\min}$.
Let $(W,{\leq},{R})$ be the frame defined by $W{=}\{a,b,c\}$, $a{\leq}c$ and $a{R}b$.
Let $V\ :\ \At{\longrightarrow}\wp(W)$ be a valuation on $(W,{\leq},{R})$ such that $V(p){=}\emptyset$ and $V(q){=}\emptyset$.
Obviously, $(W,{\leq},{R},V),c{\not\models}{\lozenge}(p{\rightarrow}q){\rightarrow}({\square}p{\rightarrow}{\lozenge}q)$.
Consequently, $(W,{\leq},{R},V){\not\models}{\lozenge}(p{\rightarrow}q){\rightarrow}({\square}p{\rightarrow}{\lozenge}q)$.
Hence, $(W,{\leq},{R}){\not\models}{\lozenge}(p{\rightarrow}q){\rightarrow}({\square}p{\rightarrow}
$\linebreak$
{\lozenge}q)$.
Thus, by the soundness of $\L_{\min}$, ${\lozenge}(p{\rightarrow}q){\rightarrow}({\square}p{\rightarrow}{\lozenge}q){\not\in}\L_{\min}$: a contradiction.
$\mathbf{(2)}$~For the sake of the contradiction, suppose $\L_{\min}$ is not strictly contained in $\L_{\bcfra}$.
Consequently, $({\lozenge}p{\rightarrow}{\square}q){\rightarrow}{\square}(p{\rightarrow}q)$ is in $\L_{\min}$.
Let $(W,{\leq},{R})$ be the frame defined by $W{=}\{a,b,d\}$, $b{\leq}d$ and $a{R}b$.
Let $V\ :\ \At{\longrightarrow}\wp(W)$ be a valuation on $(W,{\leq},{R})$ such that $V(p){=}\{d\}$ and $V(q){=}\emptyset$.
Obviously, $(W,{\leq},{R},V),a{\not\models}({\lozenge}p{\rightarrow}{\square}q){\rightarrow}{\square}(p{\rightarrow}q)$.
\linebreak
Hence, $(W,{\leq},{R},V){\not\models}({\lozenge}p{\rightarrow}{\square}q){\rightarrow}{\square}(p{\rightarrow}q)$.
Thus, $(W,{\leq},{R}){\not\models}({\lozenge}p{\rightarrow}{\square}q){\rightarrow}{\square}(p{\rightarrow}
$\linebreak$
q)$.
Consequently, by the soundness of $\L_{\min}$, $({\lozenge}p{\rightarrow}{\square}q){\rightarrow}{\square}(p{\rightarrow}q){\not\in}\L_{\min}$: a contradiction.
$\mathbf{(3)}$~For the sake of the contradiction, suppose $\L_{\min}$ is not strictly contained in $\L_{\dcfra}$.
Hence, ${\square}(p{\vee}q){\rightarrow}{\lozenge}p{\vee}{\square}q$ is in $\L_{\min}$.
Let $(W,{\leq},{R})$ be the frame defined by $W{=}\{a,c,
$\linebreak$
d\}$, $a{\leq}c$ and $c{R}d$.
Let $V\ :\ \At{\longrightarrow}\wp(W)$ be a valuation on $(W,{\leq},{R})$ such that $V(p){=}\{d\}$ and $V(q){=}\emptyset$.
Obviously, $(W,{\leq},{R},V),a{\not\models}{\square}(p{\vee}q){\rightarrow}{\lozenge}p{\vee}{\square}q$.
Thus, $(W,
$\linebreak$
{\leq},{R},V){\not\models}{\square}(p{\vee}q){\rightarrow}{\lozenge}p{\vee}{\square}q$.
Consequently, $(W,{\leq},{R}){\not\models}{\square}(p{\vee}q){\rightarrow}{\lozenge}p{\vee}{\square}q$.
Hence, by the soundness of $\L_{\min}$, ${\square}(p{\vee}q){\rightarrow}{\lozenge}p{\vee}{\square}q{\not\in}\L_{\min}$: a contradiction.
%
%
%\\
%\\
%
%
%All in all, we have proved that $\L_{\min}$ is strictly contained in $\L_{\fcfra}$, $\L_{\bcfra}$ and $\L_{\dcfra}$.
%
%

%
%
$\mathbf{(4)}$~For the sake of the contradiction, suppose $\L_{\fcfra}$ is not strictly contained in $\L_{\fbcfra}$.
Hence, $({\lozenge}p{\rightarrow}{\square}q){\rightarrow}{\square}(p{\rightarrow}q)$ is in $\L_{\fcfra}$.
Let $(W,{\leq},{R})$ be the frame defined above in $\mathbf{(2)}$.
As we already know, $(W,{\leq},{R}){\not\models}({\lozenge}p{\rightarrow}{\square}q){\rightarrow}{\square}(p{\rightarrow}q)$.
However, $(W,{\leq},{R})$ being forward confluent, by Lemma~\ref{lemma:about:axiom:fc:soundness}, $(W,{\leq},{R}){\models}\L_{\fcfra}$.
Since $(W,{\leq},{R}){\not\models}({\lozenge}p{\rightarrow}{\square}q){\rightarrow}
$\linebreak$
{\square}(p{\rightarrow}q)$, then $({\lozenge}p{\rightarrow}{\square}q){\rightarrow}{\square}(p{\rightarrow}q)$ is not in $\L_{\fcfra}$: a contradiction.
$\mathbf{(5)}$~For the sake of the contradiction, suppose $\L_{\fcfra}$ is not strictly contained in $\L_{\fdcfra}$.
Thus, ${\square}(p{\vee}q){\rightarrow}{\lozenge}p{\vee}{\square}q$ is in $\L_{\fcfra}$.
Let $(W,{\leq},{R})$ be the frame defined above in $\mathbf{(3)}$.
As we already know, $(W,{\leq},{R}){\not\models}{\square}(p{\vee}q){\rightarrow}{\lozenge}p{\vee}{\square}q$.
However, $(W,{\leq},{R})$ being forward confluent, by Lemma~\ref{lemma:about:axiom:fc:soundness}, $(W,{\leq},{R}){\models}\L_{\fcfra}$.
Since $(W,{\leq},{R}){\not\models}{\square}(p{\vee}q){\rightarrow}{\lozenge}p{\vee}{\square}q$, then ${\square}(p{\vee}q){\rightarrow}{\lozenge}p{\vee}{\square}q$ is not in $\L_{\fcfra}$: a contradiction.
%
%
%\\
%\\
%
%
%All in all, we have proved that $\L_{\fcfra}$ is strictly contained in $\L_{\fbcfra}$ and $\L_{\fdcfra}$.
%
%

%
%
$\mathbf{(6)}$~For the sake of the contradiction, suppose $\L_{\bcfra}$ is not strictly contained in $\L_{\fbcfra}$.
Thus, ${\lozenge}(p{\rightarrow}q){\rightarrow}({\square}p{\rightarrow}{\lozenge}q)$ is in $\L_{\bcfra}$.
Let $(W,{\leq},{R})$ be the frame defined above in $\mathbf{(1)}$.
As we already know, $(W,{\leq},{R}){\not\models}{\lozenge}(p{\rightarrow}q){\rightarrow}({\square}p{\rightarrow}{\lozenge}q)$.
However, $(W,{\leq},{R})$ being backward confluent, by Lemma~\ref{lemma:about:axiom:fc:soundness}, $(W,{\leq},{R}){\models}\L_{\bcfra}$.
Since $(W,{\leq},{R}){\not\models}{\lozenge}(p{\rightarrow}q){\rightarrow}
$\linebreak$
({\square}p{\rightarrow}{\lozenge}q)$, then ${\lozenge}(p{\rightarrow}q){\rightarrow}({\square}p{\rightarrow}{\lozenge}q)$ is not in $\L_{\bcfra}$: a contradiction.
$\mathbf{(7)}$~For the sake of the contradiction, suppose $\L_{\bcfra}$ is not strictly contained in $\L_{\bdcfra}$.
Consequently, ${\square}(p{\vee}q){\rightarrow}{\lozenge}p{\vee}{\square}q$ is in $\L_{\bcfra}$.
Let $(W,{\leq},{R})$ be the frame defined above in $\mathbf{(3)}$.
As we already know, $(W,{\leq},{R}){\not\models}{\square}(p{\vee}q){\rightarrow}{\lozenge}p{\vee}{\square}q$.
However, $(W,{\leq},{R})$ being backward confluent, by Lemma~\ref{lemma:about:axiom:fc:soundness}, $(W,{\leq},{R}){\models}\L_{\bcfra}$.
Since $(W,{\leq},{R}){\not\models}{\square}(p{\vee}q){\rightarrow}
$\linebreak$
{\lozenge}p{\vee}{\square}q$, then ${\square}(p{\vee}q){\rightarrow}{\lozenge}p{\vee}{\square}q$ is not in $\L_{\bcfra}$: a contradiction.
%
%
%\\
%\\
%
%
%All in all, we have proved that $\L_{\bcfra}$ is strictly contained in $\L_{\fbcfra}$ and $\L_{\bdcfra}$.
%
%

%
%
$\mathbf{(8)}$~For the sake of the contradiction, suppose $\L_{\dcfra}$ is not strictly contained in $\L_{\fdcfra}$.
Consequently, ${\lozenge}(p{\rightarrow}q){\rightarrow}({\square}p{\rightarrow}{\lozenge}q)$ is in $\L_{\dcfra}$.
Let $(W,{\leq},{R})$ be the frame defined above in $\mathbf{(1)}$.
As we already know, $(W,{\leq},{R}){\not\models}{\lozenge}(p{\rightarrow}q){\rightarrow}({\square}p{\rightarrow}{\lozenge}q)$.
However, $(W,
$\linebreak$
{\leq},{R})$ being downward confluent, by Lemma~\ref{lemma:about:axiom:fc:soundness}, $(W,{\leq},{R}){\models}\L_{\dcfra}$.
Since $(W,{\leq},{R})
$\linebreak$
{\not\models}{\lozenge}(p{\rightarrow}q){\rightarrow}({\square}p{\rightarrow}{\lozenge}q)$, then ${\lozenge}(p{\rightarrow}q){\rightarrow}({\square}p{\rightarrow}{\lozenge}q)$ is not in $\L_{\dcfra}$: a contradiction.
$\mathbf{(9)}$~For the sake of the contradiction, suppose $\L_{\dcfra}$ is not strictly contained in $\L_{\bdcfra}$.
Hence, $({\lozenge}p{\rightarrow}{\square}q){\rightarrow}{\square}(p{\rightarrow}q)$ is in $\L_{\dcfra}$.
Let $(W,{\leq},{R})$ be the frame defined above in $\mathbf{(2)}$.
As we already know, $(W,{\leq},{R}){\not\models}({\lozenge}p{\rightarrow}{\square}q){\rightarrow}{\square}(p{\rightarrow}q)$.
However, $(W,{\leq},{R})$ being downward confluent, by Lemma~\ref{lemma:about:axiom:fc:soundness}, $(W,{\leq},{R}){\models}\L_{\dcfra}$.
Since $(W,{\leq},{R}){\not\models}({\lozenge}p{\rightarrow}
$\linebreak$
{\square}q){\rightarrow}{\square}(p{\rightarrow}q)$, then $({\lozenge}p{\rightarrow}{\square}q){\rightarrow}{\square}(p{\rightarrow}q)$ is not in $\L_{\dcfra}$: a contradiction.
%
%
%\\
%\\
%
%
%All in all, we have proved that $\L_{\dcfra}$ is strictly contained in $\L_{\fdcfra}$ and $\L_{\bdcfra}$.
%
%

%
%
$\mathbf{(10)}$~For the sake of the contradiction, suppose $\L_{\fbcfra}$ is not strictly contained in $\L_{\fbdcfra}$.
Hence, ${\square}(p{\vee}q){\rightarrow}{\lozenge}p{\vee}{\square}q$ is in $\L_{\fbcfra}$.
Let $(W,{\leq},{R})$ be the frame defined above in $\mathbf{(3)}$.
As we already know, $(W,{\leq},{R}){\not\models}{\square}(p{\vee}q){\rightarrow}{\lozenge}p{\vee}{\square}q$.
However, $(W,{\leq},{R})$ being forward and backward confluent, by Lemma~\ref{lemma:about:axiom:fc:soundness}, $(W,{\leq},{R}){\models}\L_{\fbcfra}$.
Since $(W,{\leq},
$\linebreak$
{R}){\not\models}{\square}(p{\vee}q){\rightarrow}{\lozenge}p{\vee}{\square}q$, then ${\square}(p{\vee}q){\rightarrow}{\lozenge}p{\vee}{\square}q$ is not in $\L_{\fbcfra}$: a contradiction.
$\mathbf{(11)}$~For the sake of the contradiction, suppose $\L_{\fdcfra}$ is not strictly contained in $\L_{\fbdcfra}$.
Thus, $({\lozenge}p{\rightarrow}{\square}q){\rightarrow}{\square}(p{\rightarrow}q)$ is in $\L_{\fdcfra}$.
Let $(W,{\leq},{R})$ be the frame defined above in $\mathbf{(2)}$.
As we already know, $(W,{\leq},{R}){\not\models}({\lozenge}p{\rightarrow}{\square}q){\rightarrow}{\square}(p{\rightarrow}q)$.
However, $(W,{\leq},{R})$ being forward and downward confluent, by Lemma~\ref{lemma:about:axiom:fc:soundness}, $(W,{\leq},{R}){\models}\L_{\fdcfra}$.
Since $(W,{\leq},{R}){\not\models}({\lozenge}p{\rightarrow}{\square}q){\rightarrow}{\square}(p{\rightarrow}q)$, then $({\lozenge}p{\rightarrow}{\square}q){\rightarrow}{\square}(p{\rightarrow}q)$ is not in $\L_{\fdcfra}$: a contradiction.
$\mathbf{(12)}$~For the sake of the contradiction, suppose $\L_{\bdcfra}$ is not strictly contained in $\L_{\fbdcfra}$.
Consequently, ${\lozenge}(p{\rightarrow}q){\rightarrow}({\square}p{\rightarrow}{\lozenge}q)$ is in $\L_{\bdcfra}$.
Let $(W,{\leq},{R})$ be the frame defined above in $\mathbf{(1)}$.
As we already know, $(W,{\leq},{R}){\not\models}{\lozenge}(p{\rightarrow}q){\rightarrow}({\square}p{\rightarrow}{\lozenge}q)$.
However, $(W,{\leq},{R})$ being backward and downward confluent, by Lemma~\ref{lemma:about:axiom:fc:soundness}, $(W,{\leq},{R}){\models}\L_{\dcfra}$.
Since $(W,{\leq},{R}){\not\models}{\lozenge}(p{\rightarrow}q){\rightarrow}({\square}p{\rightarrow}{\lozenge}q)$, then ${\lozenge}(p{\rightarrow}q){\rightarrow}({\square}p{\rightarrow}{\lozenge}q)$ is not in $\L_{\bdcfra}$: a contradiction.
%
%
%\\
%\\
%
%
%All in all, we have proved that $\L_{\fbcfra}$, $\L_{\fdcfra}$ and $\L_{\bdcfra}$ are strictly contained in $\L_{\fbdcfra}$.
%
%
\medskip
\end{proof}
$\K$ is the least modal logic presented~---~axioms, inference rules and semantics~---~in~\cite[Chapter~$2$]{Blackburn:et:al:2001} and~\cite[Chapter~$3$]{Chagrov:Zakharyaschev:1997}.
By considering, on one hand, the axioms and inference rules defining $\L_{\fbdcfra}$ and, on the other hand, the axioms and inference rules defining $\K$, the reader may easily notice that all axioms of $\L_{\fbdcfra}$ are valid in the semantics of $\K$, all inference rules of $\L_{\fbdcfra}$ preserve validity in this semantics and some formulas belonging to $\K$ are outside $\L_{\fbdcfra}$.
Indeed, ${\neg}{\lozenge}{\neg}p{\rightarrow}{\square}p$ belongs to $\K$ but the model $(W,{\leq},{R},V)$ defined as follows shows that ${\neg}{\lozenge}{\neg}p{\rightarrow}{\square}p$ is outside $\L_{\fbdcfra}$: $W{=}\{a,b,c,d,e,f\}$, $a{\leq}c$, $b{\leq}d$, $c{\leq}e$ and $d{\leq}f$, $a{R}b$, $c{R}d$ and $e{R}f$ and $V(p){=}\{f\}$.
As a result,
\begin{proposition}\label{fbdcfra:K}
$\L_{\fbdcfra}$ is strictly contained in $\K$.
\end{proposition}
\begin{proposition}\label{lemma:table:7:lemma}
\begin{itemize}
\item $\L_{\reflexive}$ is equal to $\Log({\mathcal C}_{\reflexive})$,
\item $\L_{\symmetric}$ is equal to $\Log({\mathcal C}_{\symmetric})$,
\item $\L_{\min}$ is equal to $\Log({\mathcal C}_{\transitive})$,
\item $\L_{\reflexive}{\oplus}\L_{\symmetric}$ is equal to $\Log({\mathcal C}_{\reflexive}{\cap}{\mathcal C}_{\symmetric})$,
\item $\L_{\reflexive}$ is equal to $\Log({\mathcal C}_{\reflexive}{\cap}{\mathcal C}_{\transitive})$,
\item $\L_{\reflexive}{\oplus}\L_{\symmetric}$ is equal to $\Log({\mathcal C}_{\partition})$,
\end{itemize}
\end{proposition}
\begin{proof}
Similar to the proof of Proposition~\ref{proposition:soundness:completeness}, this time using Lemmas~\ref{lemma:same:valid:formulas}, \ref{lemma:correspondence:formulas:T:B:4:elementary:conditions}, \ref{lemma:about:axiom:2} and~\ref{lemma:about:rule:3}, Lindenbaum Lemma and Lemmas~\ref{lemma:truth:lemma}, \ref{lemma:about:what:happens:if:T:square:is:in:the:logic}, \ref{lemma:about:what:happens:if:T:lozenge:is:in:the:logic}, \ref{lemma:about:what:happens:if:B:square:is:in:the:logic}, \ref{lemma:about:what:happens:if:B:lozenge:is:in:the:logic}, \ref{lemma:about:what:happens:if:4:square:is:in:the:logic} and~\ref{lemma:about:what:happens:if:4:lozenge:is:in:the:logic}.
\medskip
\end{proof}
\begin{proposition}\label{proposition:completeness:uref:dref:etc}
\begin{itemize}
\item $\L_{\ureflexive}$ is equal to $\Log({\mathcal C}_{\ureflexive})$,
\item $\L_{\dreflexive}$ is equal to $\Log({\mathcal C}_{\dreflexive})$,
\item $\L_{\usymmetric}$ is equal to $\Log({\mathcal C}_{\usymmetric})$,
\item $\L_{\dsymmetric}$ is equal to $\Log({\mathcal C}_{\dsymmetric})$,
\item $\L_{\utransitive}$ is equal to $\Log({\mathcal C}_{\utransitive})$,
\item $\L_{\dtransitive}$ is equal to $\Log({\mathcal C}_{\dtransitive})$.
\end{itemize}
\end{proposition}
\begin{proof}
Similar to the proof of Proposition~\ref{proposition:soundness:completeness}, this time using Lemmas~\ref{lemma:correspondence:formulas:T:B:4:elementary:conditions}, \ref{lemma:about:axiom:2} and~\ref{lemma:about:rule:3}, Lindenbaum Lemma and Lemmas~\ref{lemma:truth:lemma}, \ref{lemma:about:what:happens:if:T:square:is:in:the:logic}, \ref{lemma:about:what:happens:if:T:lozenge:is:in:the:logic}, \ref{lemma:about:what:happens:if:B:square:is:in:the:logic}, \ref{lemma:about:what:happens:if:B:lozenge:is:in:the:logic}, \ref{lemma:about:what:happens:if:4:square:is:in:the:logic} and~\ref{lemma:about:what:happens:if:4:lozenge:is:in:the:logic}.
\medskip
\end{proof}
%
%
%Unfortunately, we do not know whether $\L_{\min}{\oplus}p{\rightarrow}{\lozenge}p$ is the intuitionistic modal logic determined by the class of all frames $(W,{\leq},{R})$ such that for all $s{\in}W$, $s{\geq}{\circ}{R}{\circ}{\geq}s$, $\L_{\min}{\oplus}p{\rightarrow}{\square}{\lozenge}p$ is the intuitionistic modal logic determined by the class of all frames $(W,{\leq},{R})$ such that for all $s,t{\in}W$, if $s{R}t$ then $t{\geq}{\circ}{R}{\circ}{\geq}s$ and $\L_{\min}{\oplus}{\lozenge}p{\rightarrow}{\lozenge}{\lozenge}p$ is the intuitionistic modal logic determined by the class of all frames $(W,{\leq},{R})$ such that for all $s,t,u,v{\in}W$, if $s{R}t$, $t{\geq}u$ and $u{R}v$ then $s{\geq}{\circ}{R}{\circ}{\geq}v$.
%
%
%
%
\section{Decidability}\label{section:finite:frame:property}
In this section, we address the membership problem and the consistency problem.
\begin{definition}[Membership problem]
For all intuitionistic modal logics $\L$, the {\em membership problem in $\L$}\/ is the following decision problem:
\begin{description}
\item[input:] a formula $A$,
\item[output:] determine whether $A$ is in $\L$.
\end{description}
\end{definition}
\begin{definition}[Consistency problem]
For all intuitionistic modal logics $\L$, the {\em consistency problem in $\L$}\/ is the following decision problem:
\begin{description}
\item[input:] a formula $A$,
\item[output:] determine whether $\L{+}A$ is proper.
\end{description}
\end{definition}
Obviously, since for all formulas $A$, $\L{+}A$ is proper if and only if ${\neg}A$ is not in $\L$, then if the membership problem in $\L$ is decidable then the consistency problem in $\L$ is decidable.
The usual approach to solving these problems is to build finite models.
Filtration is a fundamental concept for building finite models.
As far as we are aware, with a few exceptions such as~\cite{Hasimoto:2001,Sotirov:1984,Takano:2003}, it has not been so much adapted to intuitionistic modal logics, probably because it does not easily work with conditions involving the composition of binary relations such as our elementary conditions of confluence.
In this section, we adopt a definition of filtrations which is as general as possible and which, nevertheless, allows us to prove the corresponding Filtration Lemma.
\begin{definition}[Filtrations]
The {\em equivalence setting determined by a model $(W,{\leq},{R},
$\linebreak$
V)$ and a closed set $\Sigma$ of formulas}\/ is the equivalence relation $\simeq$ on $W$ defined by
\begin{itemize}
\item $s\simeq t$ if and only if for all formulas $A$ in $\Sigma$, $s{\models}A$ if and only if $t{\models}A$.
\end{itemize}
For all models $(W,{\leq},{R},V)$, for all closed sets $\Sigma$ of formulas and for all $s{\in}W$, the equivalence class of $s$ modulo $\simeq$ is denoted $\lbrack s\rbrack$.
For all models $(W,{\leq},{R},V)$, for all closed sets $\Sigma$ of formulas and for all $X{\in}\wp(W)$, the quotient set of $X$ modulo $\simeq$ is denoted $X/{\simeq}$.
A model $(W^{\prime},{\leq^{\prime}},{R^{\prime}},V^{\prime})$ is a {\em filtration of a model $(W,{\leq},{R},V)$ with respect to a closed set $\Sigma$ of formulas}\/ if
\begin{itemize}
\item $W^{\prime}{=}W/{\simeq}$,
\item for all $s,t{\in}W$, if $s{\leq}t$ then $\lbrack s\rbrack\leq^{\prime}\lbrack t\rbrack$,
\item for all formulas $A,B$, if $A{\rightarrow}B{\in}\Sigma$ then for all $s,t{\in}W$, if $s{\models}A{\rightarrow}B$, $\lbrack s\rbrack\leq^{\prime}\lbrack t\rbrack$ and $t{\models}A$ then $t{\models}B$,
\item for all $s,t{\in}W$, if $s{\leq}{\circ}Rt$ then $\lbrack s\rbrack{\leq^{\prime}}{\circ}R^{\prime}{\circ}{\leq^{\prime}}\lbrack t\rbrack$,
\item for all formulas $A$, if ${\square}A{\in}\Sigma$ then for all $s,t{\in}W$, if $s{\models}{\square}A$ and $\lbrack s\rbrack{\leq^{\prime}}{\circ}R^{\prime}\lbrack t\rbrack$ then $t{\models}A$,
\item for all $s,t{\in}W$, if $s{\geq}{\circ}Rt$ then $\lbrack s\rbrack{\geq^{\prime}}{\circ}R^{\prime}{\circ}{\geq^{\prime}}\lbrack t\rbrack$,
\item for all formulas $A$, if ${\lozenge}A{\in}\Sigma$ then for all $s,t{\in}W$, if $t{\models}A$ and $\lbrack s\rbrack{\geq^{\prime}}{\circ}R^{\prime}\lbrack t\rbrack$ then $s{\models}{\lozenge}A$,
\item for all atoms $p$, if $p{\in}\Sigma$ then $V^{\prime}(p){=}V(p)/{\simeq}$.
\end{itemize}
\end{definition}
\begin{lemma}\label{lemma:finite:classes:filtration:contains}
If the model $(W^{\prime},{\leq^{\prime}},{R^{\prime}},V^{\prime})$ is a filtration of the model $(W,{\leq},{R},V)$ with respect to a finite closed set $\Sigma$ of formulas then $\card(W^{\prime})\leq2^{\card(\Sigma)}$.
\end{lemma}
\begin{proof}
Suppose the model $(W^{\prime},{\leq^{\prime}},{R^{\prime}},V^{\prime})$ is a filtration of the model $(W,{\leq},{R},V)$ with respect to a finite closed set $\Sigma$ of formulas.
Let $f\ :\ W^{\prime}{\longrightarrow}\wp(\Sigma)$ be the function such that for all $s{\in}W$, $f(\lbrack s\rbrack){=}\{A{\in}\Sigma\ :\ s{\models}A\}$.
Obviously, $f$ is injective.
Hence, $\card(W^{\prime})\leq2^{\card(\Sigma)}$.
\medskip
\end{proof}
\begin{lemma}\label{filtration:lemma}
Let the model $(W^{\prime},{\leq^{\prime}},{R^{\prime}},V^{\prime})$ be a filtration of the model $(W,{\leq},{R},V)$ with respect to a closed set $\Sigma$ of formulas.
For all formulas $A$, if $A{\in}\Sigma$ then for all $s{\in}W$, $\lbrack s\rbrack{\models}A$ if and only if $s{\models}A$.
\end{lemma}
\begin{proof}
By induction on $A$.
\medskip
\end{proof}
\begin{definition}[Smallest filtrations]
A {\em smallest filtration of a model $(W,{\leq},{R},V)$
\linebreak
with respect to a closed set $\Sigma$ of formulas}\/ is a structure $(W^{\prime},{\leq^{\prime}},{R^{\prime}},V^{\prime})$ such that
\begin{itemize}
\item $W^{\prime}{=}W/{\simeq}$,
\item for all $s,t{\in}W$, $\lbrack s\rbrack\leq^{\prime}\lbrack t\rbrack$ if and only if for all formulas $A$, if $A{\in}\Sigma$ and $s{\models}A$ then $t{\models}A$,
\item for all $s,t{\in}W$, $\lbrack s\rbrack R^{\prime}\lbrack t\rbrack$ if and only if $s{\simeq}{\circ}R{\circ}{\simeq}t$,
\item for all atoms $p$, if $p{\in}\Sigma$ then $V^{\prime}(p){=}V(p)/{\simeq}$.
\end{itemize}
\end{definition}
Obviously,
\begin{lemma}\label{lemma:smallest:filtrations}
The smallest filtrations of a model $(W,{\leq},{R},V)$ with respect to a closed set $\Sigma$ of formulas are filtrations of $(W,{\leq},{R},V)$ with respect to $\Sigma$.
\end{lemma}
\begin{definition}[Largest filtrations]
A {\em largest filtration of a model $(W,{\leq},{R},V)$ with respect to a closed set $\Sigma$ of formulas}\/ is a structure $(W^{\prime},{\leq^{\prime}},{R^{\prime}},V^{\prime})$ such that
\begin{itemize}
\item $W^{\prime}{=}W/{\simeq}$,
\item for all $s,t{\in}W$, $\lbrack s\rbrack\leq^{\prime}\lbrack t\rbrack$ if and only if for all formulas $A$, if $A{\in}\Sigma$ and $s{\models}A$ then $t{\models}A$,
\item for all $s,t{\in}W$, $\lbrack s\rbrack R^{\prime}\lbrack t\rbrack$ if and only if
\begin{itemize}
\item for all formulas $A$, if ${\square}A{\in}\Sigma$ and $s{\models}{\square}A$ then $t{\models}A$,
\item for all formulas $A$, if ${\lozenge}A{\in}\Sigma$ and $t{\models}A$ then $s{\models}{\lozenge}A$,
\end{itemize}
\item for all atoms $p$, if $p{\in}\Sigma$ then $V^{\prime}(p){=}V(p)/{\simeq}$.
\end{itemize}
\end{definition}
Obviously,
\begin{lemma}\label{lemma:largest:filtrations}
The largest filtrations of a model $(W,{\leq},{R},V)$ with respect to a closed set $\Sigma$ of formulas are filtrations of $(W,{\leq},{R},V)$ with respect to $\Sigma$.
\end{lemma}
By Lemma~\ref{size:of:least:closed:set:containing:some:given:formula}, Lindenbaum Lemma and Lemmas~\ref{lemma:about:consistent:logic:and:proper:theories}, \ref{lemma:truth:lemma}, \ref{lemma:finite:classes:filtration:contains}, \ref{filtration:lemma}, \ref{lemma:smallest:filtrations} and~\ref{lemma:largest:filtrations}, we obtain the
\begin{proposition}[Finite Frame Property of $\L_{\min}$]\label{proposition:ffp:of:Lmin}
For all formulas $A$, if $A{\not\in}\L_{\min}$
\linebreak
then there exists a frame $(W,{\leq},{R})$ such that $(W,{\leq},{R}){\not\models}A$ and $\card(W)\leq2^{{\parallel}A{\parallel}}$.
\end{proposition}
%
%
%The proof of Proposition~\ref{proposition:complexity:satisfiability:Kg:decidable} is similar to the proof of Theorem~$6.15$ in~\cite{Blackburn:et:al:2001}.
%
%
\begin{proposition}[Decidability of $\L_{\min}$]\label{proposition:complexity:satisfiability:Kg:decidable}
The membership problem in $\L_{\min}$ is decidable.
\end{proposition}
\begin{proof}
By~\cite[Theorem~$6.15$]{Blackburn:et:al:2001} and Proposition~\ref{proposition:ffp:of:Lmin}.
\medskip
\end{proof}
Moreover, by Lemma~\ref{size:of:least:closed:set:containing:some:given:formula}, \ref{lemma:about:up:and:down:frames:and:the:corresponding:intersectional:frames}, \ref{lemma:satisfiability:does:not:change:if:intersectional:update:of:a:model}, Lindenbaum Lemma and Lemmas~\ref{lemma:about:consistent:logic:and:proper:theories}, \ref{lemma:truth:lemma}, \ref{lemma:about:what:happens:if:T:square:is:in:the:logic}, \ref{lemma:about:what:happens:if:T:lozenge:is:in:the:logic}, \ref{lemma:about:what:happens:if:B:square:is:in:the:logic}, \ref{lemma:about:what:happens:if:B:lozenge:is:in:the:logic}, \ref{lemma:finite:classes:filtration:contains}, \ref{filtration:lemma} and \ref{lemma:smallest:filtrations}, using the fact that the smallest filtrations of a model based on a reflexive frame are based on a reflexive frame and the smallest filtrations of a model based on a symmetric frame are based on a symmetric frame, we obtain the
\begin{proposition}[Finite Frame Property of $\L_{\reflexive}$, $\L_{\symmetric}$ and $\L_{\reflexive}{\oplus}\L_{\symmetric}$]\label{proposition:ffp:of:Lmin:plus:T:without:sans:up:and:down}
Let $A$ be a formula $A$.
\begin{enumerate}
\item If $A{\not\in}\L_{\reflexive}$ then there exists a frame $(W,{\leq},{R})$ such that $(W,{\leq},{R}){\models}\L_{\reflexive}$, $(W,
$\linebreak$
{\leq},{R}){\not\models}A$ and $\card(W)\leq2^{{\parallel}A{\parallel}}$,
\item if $A{\not\in}\L_{\symmetric}$ then there exists a frame $(W,{\leq},{R})$ such that $(W,{\leq},{R}){\models}\L_{\symmetric}$, $(W,{\leq},{R}){\not\models}A$ and $\card(W)\leq2^{{\parallel}A{\parallel}}$,
\item if $A{\not\in}\L_{\reflexive}{\oplus}\L_{\symmetric}$ then there exists a frame $(W,{\leq},{R})$ such that $(W,{\leq},{R}){\models}
$\linebreak$
\L_{\reflexive}{\oplus}\L_{\symmetric}$, $(W,{\leq},{R}){\not\models}A$ and $\card(W)\leq2^{{\parallel}A{\parallel}}$.
\end{enumerate}
\end{proposition}
\begin{proposition}[Decidability of $\L_{\reflexive}$, $\L_{\symmetric}$ and $\L_{\reflexive}{\oplus}\L_{\symmetric}$]\label{proposition:complexity:satisfiability:Lmin:plus:B:lozenge:decidable:abc}
The membership problems in the following intuitionistic modal logics are decidable: $\L_{\reflexive}$, $\L_{\symmetric}$ and $\L_{\reflexive}{\oplus}\L_{\symmetric}$.
\end{proposition}
\begin{proof}
By~\cite[Theorem~$6.15$]{Blackburn:et:al:2001} and Proposition~\ref{proposition:ffp:of:Lmin:plus:T:without:sans:up:and:down}.
\medskip
\end{proof}
In other respect, by Lemma~\ref{size:of:least:closed:set:containing:some:given:formula}, Lindenbaum Lemma and Lemmas~\ref{lemma:about:consistent:logic:and:proper:theories}, \ref{lemma:truth:lemma}, \ref{lemma:about:what:happens:if:T:square:is:in:the:logic}, \ref{lemma:about:what:happens:if:T:lozenge:is:in:the:logic}, \ref{lemma:about:what:happens:if:B:square:is:in:the:logic}, \ref{lemma:about:what:happens:if:B:lozenge:is:in:the:logic}, \ref{lemma:finite:classes:filtration:contains}, \ref{filtration:lemma} and \ref{lemma:smallest:filtrations}, using the fact that the smallest filtrations of a model based on an up-reflexive frame are based on an up-reflexive frame, the smallest filtrations of a model based on a down-reflexive frame are based on a down-reflexive frame, the smallest filtrations of a model based on an up-symmetric frame are based on an up-symmetric frame and the smallest filtrations of a model based on a down-symmetric frame are based on a down-symmetric frame, we obtain the
\begin{proposition}[Finite Frame Property of $\L_{\ureflexive}$, $\L_{\dreflexive}$, $\L_{\usymmetric}$ and $\L_{\dsymmetric}$]\label{proposition:ffp:of:Lmin:plus:T:square}
Let $A$ be a formula $A$.
\begin{enumerate}
\item If $A{\not\in}\L_{\ureflexive}$ then there exists a frame $(W,{\leq},{R})$ such that $(W,{\leq},{R}){\models}\L_{\ureflexive}$, $(W,{\leq},{R}){\not\models}A$ and $\card(W)\leq2^{{\parallel}A{\parallel}}$,
\item if $A{\not\in}\L_{\dreflexive}$ then there exists a frame $(W,{\leq},{R})$ such that $(W,{\leq},{R}){\models}\L_{\dreflexive}$, $(W,{\leq},{R}){\not\models}A$ and $\card(W)\leq2^{{\parallel}A{\parallel}}$,
\item if $A{\not\in}\L_{\usymmetric}$ then there exists a frame $(W,{\leq},{R})$ such that $(W,{\leq},{R}){\models}\L_{\usymmetric}$, $(W,{\leq},{R}){\not\models}A$ and $\card(W)\leq2^{{\parallel}A{\parallel}}$,
\item if $A{\not\in}\L_{\dsymmetric}$ then there exists a frame $(W,{\leq},{R})$ such that $(W,{\leq},{R}){\models}\L_{\dsymmetric}$, $(W,{\leq},{R}){\not\models}A$ and $\card(W)\leq2^{{\parallel}A{\parallel}}$.
\end{enumerate}
\end{proposition}
\begin{proposition}[Decidability of $\L_{\ureflexive}$, $\L_{\dreflexive}$, $\L_{\usymmetric}$ and $\L_{\dsymmetric}$]\label{proposition:complexity:satisfiability:Lmin:plus:B:lozenge:decidable}
The membership
\linebreak
problems in the following intuitionistic modal logics are decidable: $\L_{\ureflexive}$, $\L_{\dreflexive}$, $\L_{\usymmetric}$ and $\L_{\dsymmetric}$.
\end{proposition}
\begin{proof}
By~\cite[Theorem~$6.15$]{Blackburn:et:al:2001} and Proposition~\ref{proposition:ffp:of:Lmin:plus:T:square}.
\medskip
\end{proof}
We do not know if the membership problems in $\L_{\utransitive}$ and $\L_{\dtransitive}$ are decidable.
\section{Miscellaneous}\label{section:miscellaneous}
Intuitionistic modal logics are modal logics whose underlying logic is $\IPL$.
There exists multifarious intuitionistic modal logics.
Therefore, one may ask how natural are the variants $\L_{\min}$, $\L_{\fcfra}$, $\L_{\bcfra}$, $\L_{\dcfra}$, $\L_{\fbcfra}$, $\L_{\fdcfra}$, $\L_{\bdcfra}$ and $\L_{\fbdcfra}$ we consider here.
In Chapter~$3$ of his doctoral thesis~\cite{Simpson:1994}, Simpson discusses the following formal features that might be expected of an intuitionistic modal logic, say $\L$:
\begin{itemize}
\item $\L$ is conservative over $\IPL$,
\item $\L$ contains all substitution instances of $\IPL$ and is closed under modus ponens,
\item for all formulas $A,B$, if $A{\vee}B$ is in $\L$ then either $A$ is in $\L$, or $B$ is in $\L$,
\item the addition of the law of excluded middle to $\L$ yields modal logic $\K$,
\item ${\square}$ and ${\lozenge}$ are independent in $\L$.
\end{itemize}
\begin{lemma}\label{lemma:l:min:C1}
For all intuitionistic modal logics $\L$, if $\L$ is contained in $\L_{\fbdcfra}$ then $\L$ is conservative over $\IPL$.
\end{lemma}
\begin{proof}
By the soundness of $\L_{\fbdcfra}$ and the fact that for all frames $(W,{\leq},{R})$, if $R{=}{\emptyset}$ then $(W,{\leq},{R}){\models}\L_{\fbdcfra}$.
\medskip
\end{proof}
\begin{lemma}\label{lemma:l:min:C2}
For all intuitionistic modal logics $\L$, $\L$ contains all substitution instances of $\IPL$ and is closed under modus ponens.
\end{lemma}
\begin{proof}
By definition of intuitionistic modal logics.
\medskip
\end{proof}
\begin{lemma}[Disjunction Property]\label{lemma:l:min:C3}
For all formulas $A,B$,
\begin{itemize}
\item if $A{\vee}B$ is in $\L_{\min}$ then either $A$ is in $\L_{\min}$, or $B$ is in $\L_{\min}$,
\item if $A{\vee}B$ is in $\L_{\fcfra}$ then either $A$ is in $\L_{\fcfra}$, or $B$ is in $\L_{\fcfra}$,
\item if $A{\vee}B$ is in $\L_{\bcfra}$ then either $A$ is in $\L_{\bcfra}$, or $B$ is in $\L_{\bcfra}$,
\item if $A{\vee}B$ is in $\L_{\fbcfra}$ then either $A$ is in $\L_{\fbcfra}$, or $B$ is in $\L_{\fbcfra}$.
\end{itemize}
\end{lemma}
\begin{proof}
Let $A,B$ be formulas.
Suppose $A{\vee}B$ is in $\L_{\min}$.
For the sake of the contradiction, suppose neither $A$ is in $\L_{\min}$, nor $B$ is in $\L_{\min}$.
Hence, by the completeness of $\L_{\min}$, there exists a frame $(W_{1},{\leq_{1}},{R_{1}})$ such that $(W_{1},{\leq_{1}},{R_{1}}){\not\models}A$ and there exists a frame $(W_{2},{\leq_{2}},{R_{2}})$ such that $(W_{2},{\leq_{2}},{R_{2}}){\not\models}B$.
Thus, there exists a model $(W_{1},{\leq_{1}},{R_{1}},V_{1})$ based on $(W_{1},{\leq_{1}},
$\linebreak$
{R_{1}})$ such that $(W_{1},{\leq_{1}},{R_{1}},V_{1}){\not\models}A$ and there exists a model $(W_{2},{\leq_{2}},{R_{2}},V_{2})$ based on $(W_{2},{\leq_{2}},{R_{2}})$ such that $(W_{2},{\leq_{2}},{R_{2}},V_{2}){\not\models}B$.
Consequently, there exists $s_{1}{\in}W_{1}$ such that $(W_{1},{\leq_{1}},{R_{1}},V_{1}),s_{1}{\not\models}A$ and there exists $s_{2}{\in}W_{2}$ such that $(W_{2},{\leq_{2}},{R_{2}},V_{2}),
$\linebreak$
s_{2}{\not\models}B$.
Let $s$ be a new state and $(W,{\leq},{R})$ be the frame defined as follows:
\begin{itemize}
\item $W{=}\{s\}{\cup}W_{1}{\cup}W_{2}$,
\item ${\leq}{=}\{(s,t_{1})$: $t_{1}{\in}W_{1}$ and $s_{1}{\leq}t_{1}\}{\cup}\{(s,t_{2})$: $t_{2}{\in}W_{2}$ and $s_{2}{\leq}t_{2}\}{\cup}{\leq_{1}}{\cup}{\leq_{2}}$,
\item ${R}{=}{R_{1}}{\cup}{R_{2}}$.
\end{itemize}
Let $V\ :\ \At{\longrightarrow}\wp(W)$ be the valuation on $(W,{\leq},{R})$ such that for all atoms $p$, $V(p){=}V_{1}(p){\cup}V_{2}(p)$.
Since $A{\vee}B$ is in $\L_{\min}$, then by the soundness of $\L_{\min}$, $(W,{\leq},{R})
$\linebreak$
{\models}A{\vee}B$.
Hence, $(W,{\leq},{R},V){\models}A{\vee}B$.
Thus, $(W,{\leq},{R},V),s{\models}A{\vee}B$.
Consequently, either $(W,{\leq},{R},V),s{\models}A$, or $(W,{\leq},{R},V),s{\models}B$.
In the former case, since $s{\leq}s_{1}$, then $(W,{\leq},{R},V),s_{1}{\models}A$.
\begin{claim}
For all formulas $C$ and for all $t_{1}{\in}W_{1}$, $(W,{\leq},{R},V),t_{1}{\models}C$ if and only if $(W_{1},{\leq_{1}},{R_{1}},V_{1}),t_{1}{\models}C$.
\end{claim}
\begin{proofclaim}
By induction on $C$.
\medskip
\end{proofclaim}
Since $(W,{\leq},{R},V),s_{1}{\models}A$, then $(W_{1},{\leq_{1}},{R_{1}},V_{1}),s_{1}{\models}A$: a contradiction.
In the latter case, since $s{\leq}s_{2}$, then $(W,{\leq},{R},V),s_{2}{\models}B$.
\begin{claim}
For all formulas $C$ and for all $t_{2}{\in}W_{2}$, $(W,{\leq},{R},V),t_{2}{\models}C$ if and only if $(W_{2},{\leq_{2}},{R_{2}},V_{2}),t_{2}{\models}C$.
\end{claim}
\begin{proofclaim}
By induction on $C$.
\medskip
\end{proofclaim}
Since $(W,{\leq},{R},V),s_{2}{\models}B$, then $(W_{2},{\leq_{2}},{R_{2}},V_{2}),s_{2}{\models}B$: a contradiction.
Concerning $\L_{\fcfra}$, $\L_{\bcfra}$ and $\L_{\fbcfra}$, the proof is similar.
\medskip
\end{proof}
\begin{lemma}\label{lemma:disjunction:property:does:not:hld:when:downward}
For all intuitionistic modal logics $\L$, if $\L$ contains $\L_{\dcfra}$ and $\L$ is contained in $\L_{\fbdcfra}$ then the Disjunction Property described in Lemma~\ref{lemma:l:min:C3} does not hold for $\L$.
\end{lemma}
\begin{proof}
By Lemma~\ref{lemma:disjunction:property:does:not:hold:preliminary:result}.
\medskip
\end{proof}
\begin{lemma}\label{lemma:l:min:C4}
For all intuitionistic modal logics $\L$, if $\L$ is contained in $\L_{\fbdcfra}$ then the addition of the law of excluded middle to $\L$ yields modal logic $\K$.
\end{lemma}
\begin{proof}
Let $\L$ be an intuitionistic modal logic.
Suppose $\L$ is contained in $\L_{\fbdcfra}$.
Hence, by Proposition~\ref{fbdcfra:K}, $\L$ is contained in $\K$.
Let $\L^{+}{=}\L{\oplus}p{\vee}\neg p$.
Since $\L$ is contained in $\K$, then $\L^{+}$ is contained in $\K$.
Obviously, $\L^{+}$ contains all substitution instances of Classical Propositional Logic and is closed under modus ponens.
Moreover, it contains all substitution instances of axiom $(\Axiom1)$ and is closed under inference rule$(\Rule1)$.
Therefore, since $\L^{+}$ is contained in $\K$, then in order to demonstrate that the addition of the law of excluded middle to $\L$ yields modal logic $\K$, it suffices to demonstrate that the formulas ${\lozenge}p{\rightarrow}\neg{\square}\neg p$ and $\neg{\square}\neg p{\rightarrow}{\lozenge}p$ are in $\L^{+}$.
Obviously, ${\square}(p{\rightarrow}{\bot}){\vee}\neg{\square}\neg p$ is in $\L^{+}$.
Thus, ${\lozenge}p{\rightarrow}\neg{\square}\neg p{\vee}{\square}(p{\rightarrow}{\bot})$ is in $\L^{+}$.
Consequently, using $(\Rule3)$, ${\lozenge}p{\rightarrow}\neg{\square}\neg p{\vee}{\lozenge}{\bot}$ is in $\L^{+}$.
Hence, using axiom $(\Axiom4)$, ${\lozenge}p{\rightarrow}
$\linebreak$
\neg{\square}\neg p$ is in $\L^{+}$.
Obviously, $p{\vee}\neg p$ is in $\L^{+}$.
Thus, using $(\Rule1)$, ${\square}(p{\vee}\neg p)$ is in $\L^{+}$.
Consequently, using axiom $(\Axiom2)$, $({\lozenge}p{\rightarrow}{\square}\neg p){\rightarrow}{\square}\neg p$ is in $\L^{+}$.
Hence, $\neg{\square}\neg p{\rightarrow}{\lozenge}p$ is in $\L^{+}$.
\medskip
\end{proof}
\begin{lemma}\label{lemma:l:min:C5}
For all intuitionistic modal logics $\L$, if $\L$ is contained in $\L_{\fbdcfra}$ then ${\square}$ and ${\lozenge}$ are independent in $\L$.
\end{lemma}
\begin{proof}
Let $\L$ be an intuitionistic modal logic.
Suppose $\L$ is contained in $\L_{\fbdcfra}$.
In order to demonstrate that ${\square}$ and ${\lozenge}$ are independent in $\L$, it suffices to demonstrate that neither there exists a ${\square}$-free formula $A$ such that ${\square}p{\leftrightarrow}A$ is in $\L$, nor there exists a ${\lozenge}$-free formula $A$ such that ${\lozenge}p{\leftrightarrow}A$ is in $\L$.
For the sake of the contradiction, suppose either there exists a ${\square}$-free formula $A$ such that ${\square}p{\leftrightarrow}A$ is in $\L$, or there exists a ${\lozenge}$-free formula $A$ such that ${\lozenge}p{\leftrightarrow}A$ is in $\L$.
Hence, we have to consider the following $2$ cases.
$\mathbf{(1)}$ Case ``there exists a ${\square}$-free formula $A$ such that ${\square}p{\leftrightarrow}A$ is in $\L$'':
Without loss of generality, we may assume that $p$ is the only atom that may occur in $A$.
Since $\L$ is contained in $\L_{\fbdcfra}$ and ${\square}p{\leftrightarrow}A$ is in $\L$, then ${\square}p{\leftrightarrow}A$ is in $\L_{\fbdcfra}$.
Thus, by the soundness of $\L_{\fbdcfra}$, ${\square}p{\leftrightarrow}A$ is valid in the class of all forward, backward and downward confluent frames.
Let $(W,{\leq},{R},V)$ be the forward, backward and downward confluent model defined by $W{=}\{a,b,c,d\}$, $a\leq c$, $b\leq d$, $aRb$, $aRd$, $cRd$ and $V(p){=}\{d\}$.
\begin{claim}
For all ${\square}$-free formulas $B$, $a{\models}B$ if and only if $c{\models}B$.
\end{claim}
\begin{proofclaim}
By induction on $B$.
\medskip
\end{proofclaim}
Since $A$ is ${\square}$-free, then $a{\models}A$ if and only if $c{\models}A$.
Since ${\square}p{\leftrightarrow}A$ is valid in the class of all forward, backward and downward confluent frames, then $a{\models}{\square}p$ if and only if $c{\models}{\square}p$.
This contradicts the facts that $a{\not\models}{\square}p$ and $c{\models}{\square}p$.
$\mathbf{(2)}$ Case ``there exists a ${\lozenge}$-free formula $A$ such that ${\lozenge}p{\leftrightarrow}A$ is in $\L$'':
Without loss of generality, we may assume that $p$ is the only atom that may occur in $A$.
Since $\L$ is contained in $\L_{\fbdcfra}$ and ${\lozenge}p{\leftrightarrow}A$ is in $\L$, then ${\lozenge}p{\leftrightarrow}A$ is in $\L_{\fbdcfra}$.
Consequently, by the soundness of $\L_{\fbdcfra}$, ${\lozenge}p{\leftrightarrow}A$ is valid in the class of all forward, backward and downward confluent frames.
Let $(W,{\leq},{R},V)$ be the forward, backward and downward confluent model defined by $W{=}\{a,b,c,d\}$, $a\leq c$, $b\leq d$, $aRb$, $cRb$, $cRd$ and $V(p){=}\{d\}$.
\begin{claim}
For all ${\lozenge}$-free formulas $B$, $a{\models}B$ if and only if $c{\models}B$.
\end{claim}
\begin{proofclaim}
By induction on $B$.
\medskip
\end{proofclaim}
Since $A$ is $\Diamond$-free, then $a{\models}A$ if and only if $c{\models}A$.
Since ${\lozenge}p{\leftrightarrow}A$ is valid in the class of all forward, backward and downward confluent frames, then $a{\models}{\lozenge}p$ if and only if $c{\models}{\lozenge}p$.
This contradicts the facts that $a{\not\models}{\lozenge}p$ and $c{\models}{\lozenge}p$.
%
%
%\\
%\\
%
%
%All in all, we have proved that neither there exists a ${\square}$-free formula $A$ such that ${\square}p{\leftrightarrow}A$ is in $\L$, nor there exists a ${\lozenge}$-free formula $A$ such that ${\lozenge}p{\leftrightarrow}A$ is in $\L$.
%
%
\medskip
\end{proof}
\section{Conclusion}
Much remains to be done.
For example,
\begin{enumerate}
%
%
%\item determine whether ${\mathcal C}_{\fdcfra}$, ${\mathcal C}_{\fucfra}$, ${\mathcal C}_{\bdcfra}$, ${\mathcal C}_{\bucfra}$, ${\mathcal C}_{\fbdcfra}$, ${\mathcal C}_{\fbucfra}$, ${\mathcal C}_{\fducfra}$, ${\mathcal C}_{\bducfra}$, ${\mathcal C}_{\fbducfra}$ and ${\mathcal C}_{\qucfra}$ are modally definable,
\item determine whether ${\mathcal C}_{\qucfra}$ is modally definable,
\item determine whether $\Log({\mathcal C}_{\ucfra})$, $\Log({\mathcal C}_{\fucfra})$, $\Log({\mathcal C}_{\bucfra})$, $\Log({\mathcal C}_{\ducfra})$, $\Log({\mathcal C}_{\fbucfra})$, $\Log({\mathcal C}_{\fducfra})$, $\Log({\mathcal C}_{\bducfra})$, $\Log({\mathcal C}_{\fbducfra})$, $\Log({\mathcal C}_{\qucfra})$ and $\Log({\mathcal C}_{\symmetric}{\cap}{\mathcal C}_{\transitive})$ are finitely axiomatizable,
\item determine whether the membership problem in $\L_{\min}$ and the consistency problem in $\L_{\min}$ are $\PSPACE$-complete,
\item determine whether the membership problems in $\L_{\bcfra}$, $\L_{\dcfra}$, $\L_{\bdcfra}$, $\L_{\fbdcfra}$, $\L_{\utransitive}$ and $\L_{\dtransitive}$ and the consistency problems in $\L_{\bcfra}$, $\L_{\dcfra}$, $\L_{\bdcfra}$, $\L_{\fbdcfra}$, $\L_{\utransitive}$ and $\L_{\dtransitive}$ are decidable,
\item study the intuitionistic modal logic obtained by the addition of a connective $\lozenge^{\prime}$ {\it \`a la}\/ Wijesekera to our language,
\item axiomatize the $\square$-free fragments of our intuitionistic modal logics and the $\lozenge$-free fragments of our intuitionistic modal logics.
\end{enumerate}
In order to address~$\mathbf{(1)}$, one may use automated tools such as, for example, the one implementating the algorithm $\PEARL$~\cite{Conradie:et:al:2021}.
Regarding $\mathbf{(2)}$, we conjecture in particular that $\L_{\min}$, $\Log({\mathcal C}_{\ucfra})$ and $\Log({\mathcal C}_{\qucfra})$ are equal.
In order to show this conjecture, one may use, for example, an argument based on an adaptation of the unravelling transformation.
See~\cite[Chapter~$2$]{Blackburn:et:al:2001} and~\cite[Chapter~$3$]{Chagrov:Zakharyaschev:1997}.
With respect to $\mathbf{(3)}$, one may use, for example, translations {\em \`a la}\/ G\"odel-Tarski or tableaux-like methods.
See~\cite{Lin:Ma:2022,Mints:2012,Wolter:Zakharyaschev:1997} and~\cite{Ladner:1977,Spaan:1993}.
As for $\mathbf{(4)}$, we conjecture that the membership problems in $\L_{\bcfra}$, $\L_{\dcfra}$, $\L_{\bdcfra}$, $\L_{\fbdcfra}$, $\L_{\utransitive}$ and $\L_{\dtransitive}$ and the consistency problems in $\L_{\bcfra}$, $\L_{\dcfra}$, $\L_{\bdcfra}$, $\L_{\fbdcfra}$, $\L_{\utransitive}$ and $\L_{\dtransitive}$ are decidable.
In order to fix this conjecture, one may use, for example, a technique based on either the two-variable monadic guarded fragment, or selective filtration, or terminating sequent calculi~\cite{Alechina:Shkatov:2006,Grefe:1996,Iemhoff:2018,Lin:Ma:2019}.\footnote{The reader should be aware that it is by no means easy to determine whether the membership problem in such-or-such intuitionistic modal logic and the consistency problem in such-or-such intuitionistic modal logic are decidable, witness the fact that the decidability of the membership problem in
%$\IK{\oplus}\{{\square}p{\rightarrow}p,p{\rightarrow}{\lozenge}p,{\square}p{\rightarrow}{\square}{\square}p,{\lozenge}{\lozenge}p{\rightarrow}{\lozenge}p\}$ has only been proved recently~\cite{Girlando:et:al:2023}.}
$\IS4$~---~the intuitionistic modal logic obtained by adding to $\IK$ the formulas ${\square}p{\rightarrow}p$, $p{\rightarrow}{\lozenge}p$, ${\square}p{\rightarrow}{\square}{\square}p$ and ${\lozenge}{\lozenge}p{\rightarrow}{\lozenge}p$~---~has only been proved recently~\cite{Girlando:et:al:2023}.}
%Of course, the reader may easily verify that the membership problem in $\L_{\min}$ is in co-$\NEXPTIME$.
%Nevertheless, from now on, getting no further than Proposition~\ref{proposition:complexity:satisfiability:Kg:decidable}, we conjecture that the membership problem in $\L_{\min}$ is $\PSPACE$-complete.
%Translations {\em \`a la}\/ G\"odel-Tarski may be used for proving this conjecture~\cite{Lin:Ma:2022,Mints:2012,Wolter:Zakharyaschev:1997}.
%Tableaux-like methods can be used as well~\cite{Ladner:1977,Spaan:1993}.
%In other respect, as is well-known, the consistency problem in $\IPL$ is $\NPTIME$-complete.
%However, we conjecture that the consistency problem in $\L_{\min}$ is $\PSPACE$-complete.
%Again, translations {\em \`a la}\/ G\"odel-Tarski may be used for proving this conjecture.
%And tableaux-like methods can be used as well.
Concerning $\L_{\fcfra}$, $\L_{\bcfra}$, $\L_{\dcfra}$, $\L_{\fbcfra}$, $\L_{\fdcfra}$, $\L_{\bdcfra}$ and $\L_{\fbdcfra}$, we do not know whether filtration can be used in order to show that the membership problems in these intuitionistic modal logics are decidable.\footnote{Notice that the decidability of the membership problem in $\L_{\fcfra}$~---~remind Proposition~\ref{fcfra:FIK}~---~has been proved in~\cite{Balbiani:et:al:2024} by a technique based on nested sequents, the decidability of the membership problem in $\L_{\fbcfra}$~---~remind Proposition~\ref{fbcfra:IK}~---~has been proved in~\cite{Grefe:1996} by a technique based on selective filtration and the decidability of the membership problem in $\L_{\fdcfra}$~---~remind Proposition~\ref{fdcfra:LIK}~---~has been proved in~\cite{Balbiani:et:al:2024b} by a technique based on nested sequents.}
About $\mathbf{(5)}$, we conjecture that the intuitionistic modal logic obtained by the addition of a connective $\lozenge^{\prime}$ {\it \`a la}\/ Wijesekera to our language is completely axiomatized by adding to the axioms and inference rules of $\L_{\min}$~---~expressed in the language based on $\square$ and $\lozenge$~---~and $\WK$~---~expressed in the language based on $\square$ and $\lozenge^{\prime}$~---~the formulas ${\lozenge^{\prime}}p{\rightarrow}{\lozenge}p$ and ${\lozenge^{\prime}}(p{\vee}q){\rightarrow}(({\lozenge}p{\rightarrow}{\lozenge^{\prime}}q){\rightarrow}{\lozenge^{\prime}}q)$.
Concerning $\mathbf{(6)}$, we conjecture that either some $\square$-free fragments of our intuitionistic modal logics, or some $\lozenge$-free fragments of our intuitionistic modal logics are not finitely axiomatizable.
See~\cite{Das:Marin:2023} for a discussion about the $\lozenge$-free fragments of different intuitionistic modal logics.
%
%
%\\
%\\
%
%
%Other questions remain unsolved:
%
%
%\begin{itemize}
%
%
%\item determine whether $\L_{\min}{\oplus}p{\rightarrow}{\lozenge}p$ is canonical, $\L_{\min}{\oplus}p{\rightarrow}{\square}{\lozenge}p$ is canonical and $\L_{\min}{\oplus}{\lozenge}p{\rightarrow}{\lozenge}{\lozenge}p$ is canonical,
%
%
%\item determine whether $\L_{\min}{\oplus}p{\rightarrow}{\lozenge}p$ is the intuitionistic modal logic determined by the class of all frames $(W,{\leq},{R})$ such that for all $s{\in}W$, $s{\geq}{\circ}{R}{\circ}{\geq}s$, $\L_{\min}{\oplus}p{\rightarrow}
%
%
%
%
%
%
%
%
%
%
%
%
%$\linebreak$
%
%
%
%
%
%
%
%
%
%
%
%
%{\square}{\lozenge}p$ is the intuitionistic modal logic determined by the class of all frames $(W,{\leq},
%
%
%
%
%
%
%
%
%
%
%
%
%$\linebreak$
%
%
%
%
%
%
%
%
%
%
%
%
%{R})$ such that for all $s,t{\in}W$, if $s{R}t$ then $t{\geq}{\circ}{R}{\circ}{\geq}s$ and $\L_{\min}{\oplus}{\lozenge}p{\rightarrow}{\lozenge}{\lozenge}p$ is the intuitionistic modal logic determined by the class of all frames $(W,{\leq},{R})$ such that for all $s,t,u,v{\in}W$, if $s{R}t$, $t{\geq}u$ and $u{R}v$ then $s{\geq}{\circ}{R}{\circ}{\geq}v$.
%
%
%\end{itemize}
%
%
%
%
\section*{Acknowledgements}
We wish to thank Han Gao (Aix-Marseille University), Zhe Lin (Xiamen University), Nicola Olivetti (Aix-Marseille University) and Vladimir Sotirov (Bulgarian Academy of Sciences) for their valuable remarks.
Special acknowledgement is also granted to our colleagues of the Toulouse Institute of Computer Science Research for many stimulating discussions about the subject of this paper.
%
%
%
%
%
%We make a point of thanking as well the referees for their feedback: their useful suggestions have been essential for improving the readability of a preliminary version of this paper.
%
%
%
%
\bibliographystyle{named}
\end{document}